\documentclass[12pt]{amsart}
\usepackage{amssymb, a4wide, mathdots, url,hyperref,graphicx}
\usepackage[usenames]{color}
\usepackage{enumerate}

\newcommand{\abs}[1]{\left|{#1}\right|}

\newcommand{\as}{{\mathrm {As}}}
\newcommand{\As}{{\mathrm {As}}}
\newcommand{\JL}{{\mathrm {JL}}}
\newcommand{\GL}{\operatorname{GL}}

\newcommand{\SL}{\operatorname{SL}}

\newcommand{\Ad}{\operatorname{Ad}}

\newcommand{\sprod}[2]{\left\langle{#1},{#2}\right\rangle}

\newcommand{\ord}{\operatorname{ord}}

\newcommand{\Lie}{\operatorname{Lie}}

\renewcommand{\Re}{\operatorname{Re}}

\newcommand{\Irr}{\operatorname{Irr}}
\newcommand{\gen}{\operatorname{gen}}

\newcommand{\A}{\mathbb{A}}
\newcommand{\Z}{\mathbb{Z}}
\newcommand{\C}{\mathbb{C}}
\newcommand{\R}{\mathbb{R}}
\newcommand{\N}{\mathbb{N}}

\newcommand{\bs}{\backslash}

\newcommand{\diag}{\operatorname{diag}}

\newcommand{\Hom}{\operatorname{Hom}}
\newcommand{\Res}{\operatorname{Res}}
\newcommand{\Ind}{\operatorname{Ind}}

\newcommand{\Sym}{\operatorname{Sym}}

\newcommand{\Id}{{\operatorname{Id}}}

\newcommand{\Ord}{\operatorname{Ord}}
\newcommand{\triv}{{\bf{1}}}

\newcommand{\M}{\mathcal{M}}
\newcommand{\J}{\mathcal{J}}
\renewcommand{\O}{\mathcal{O}}

\newcommand{\sm}[4]{\left(\begin{smallmatrix}{#1}&{#2}\\{#3}&{#4}\end{smallmatrix}\right)}

\newcommand{\inv}{{\operatorname{inv}}}

\renewcommand{\P}{\mathcal{P}}
\newcommand{\D}{\mathcal{D}}

\newcommand{\BA}{\mathbb{A}}
\newcommand{\BZ}{\mathbb{Z}}
\newcommand{\BR}{\mathbb{R}}
\newcommand{\BN}{\mathbb{N}}
\newcommand{\BC}{{\mathrm {BC}}}
\renewcommand{\L}{\mathcal{L}}
\newcommand{\fra}{\mathfrak{a}}
\newcommand{\ra}{\rightarrow}

\newcommand{\nrd}{\textup{Nrd}}
\newcommand{\ve}{\varepsilon}
\newcommand{\DL}{\mathcal{L}_*}

\newcommand{\si}{\mathrm{SI}}
\newcommand{\esi}{\mathrm{ESI}}
\newcommand{\SI}{\mathcal{S}}
\newcommand{\ESI}{\mathcal{ES}}
\newcommand{\BH}{{\mathbb {H}}}
\newcommand{\St}{\mathrm{St}}
\newcommand{\st}{\mathrm{St}}
\newcommand{\CC}{\mathcal{C}}
\newcommand{\WD}{\mathrm{WD}}

\newcommand{\CM}{{\mathcal {M}}}
\newcommand{\CL}{{\mathcal {L}}}
\newcommand{\x}{{\mathfrak{x}}}

\newcommand{\GG}{\mathcal{G}}

\newcommand{\jl}{\textup{JL}}\newcommand{\steinberg}{\textup{St}}

\newtheorem{theorem}{Theorem}[section]
\newtheorem{lemma}[theorem]{Lemma}
\newtheorem{definition}[theorem]{Definition}
\newtheorem{proposition}[theorem]{Proposition}
\newtheorem{corollary}[theorem]{Corollary}
\theoremstyle{remark}
\newtheorem{remark}[theorem]{Remark}
\newtheorem{example}[theorem]{Example}
\newtheorem{cv}{Convention}[section]

\newcommand{\CE}{{\mathcal {E}}} 
  
\newcommand{\CP}{{\mathcal {P}}}

\newcommand{\itemname}[2]{\hypertarget{#1}{}\expandafter\gdef\csname #1\endcsname{#2}\textbf{(#2)}}
\newcommand{\namelink}[1]{\hyperlink{#1}{(\textbf{\csname #1\endcsname})}}

\makeatletter
\@namedef{subjclassname@2020}{%
  \textup{2020} Mathematics Subject Classification}
\makeatother

\begin{document}


\newtheorem{thm}{Theorem}[section] 
\newtheorem*{thm*}{Theorem} 
\newtheorem{cor}[thm]{Corollary}
\newtheorem{lem}[thm]{Lemma}  
\newtheorem{prop}[thm]{Proposition}
\newtheorem{prob}{Problem}[section]
\newtheorem{df}{Definition}[section]
\newtheorem{fct}{Fact}[section]
\newtheorem {conj}[thm]{Conjecture} \newtheorem{defn}[thm]{Definition}
\newtheorem{term}[thm]{Terminology}
\newtheorem {rem}[thm]{Remark} 
\newtheorem {notation}[thm]{Notation}
\newtheorem {claim}[thm]{Claim}


\title[Intertwining periods, L-functions and local-global principles]{Intertwining periods, L-functions and local-global principles for distinction of automorphic representations}
\author{Nadir Matringe}
\address{Nadir Matringe. Institute of Mathematical Sciences, NYU Shanghai, 3663 Zhongshan Road North Shanghai, 200062, China and
		Institut de Math\'ematiques de Jussieu-Paris Rive Gauche, Universit\'e Paris Cit\'e, 75205, Paris, France}
	\email{nrm6864@nyu.edu and matringe@img-prg.fr}

\author{Omer Offen}
\address{Omer Offen. Department of Mathematics, Brandeis University, Waltham MA 02453}
\email{offen@brandeis.edu}

\author{Chang Yang}
\address{Chang Yang. Key Laboratory of High Performance Computing and Stochastic Information Processing (HPCSIP), Hunan Normal University, School of Mathematics and Statistics, Changsha, 410081, China}
	\email{cyang@hunnu.edu.cn}



\begin{abstract} We provide a criterion for non-vanishing of period integrals on automorphic representations of a general linear group over a division algebra. We consider three different periods: linear periods, twisted-linear periods and Galois periods. Our criterion is a local-global principle, which is stated in terms of local distinction, a further local obstruction, and poles of certain global $L$-functions associated to the underlying involution via the Jacquet-Langlands correspondence. 

Our local-global principle follows from a new method, relying on the Maass-Selberg relations and a careful analysis of singularities of local and global intertwining periods. Our results generalize to inner forms, known results for split general linear groups. Moreover, our result for twisted linear periods is new even in the split situation. As a consequence of our local-global principle, we complete the proof of one direction of the Guo-Jacquet conjecture.
\end{abstract}

\maketitle

\goodbreak 

\tableofcontents

\goodbreak

\textbf{Dedication}. The first two named autors dedicate this work to the memory of their coauthor and collegue Chang Yang who passed away during the preparation of this manuscript. This project was initially started a long time ago by them, and left aside for a while. Chang joined the project after the first named author visited him in Hunan Normal University of Changsha in fall 2021, and his impact on putting it back on track has been decisive. As friends, we hope that it will honor his memory and wishes.

\section{Introduction} \label{intro}

The relative Langlands program explores relations between special values of automorphic $L$-functions, period integrals of automorphic forms and the images of functorial transfers. Following \cite{MR3764130} it has recently seen some new perspectives in \cite{BZ-S-V}. Families of examples of the above mentioned interrelations have been studied with a wide variety of methods. We mention a few of them and a rather incomplete list of references: integral representations of automorphic $L$-functions (see for example \cite{MR1044830}, \cite{MR1241129}, \cite{MR984899}, \cite{MR2718823}), the theta correspondence (see \cite{MR783511} or \cite{ChenGan} for a more recent example), the residue method introduced in \cite{MR1142486} (see \cite{MR4255059} for recent developments concerning special $L$-values as well as a historical survey), and probably the most powerful, Jacquet's relative trace formula (for example \cite{MR909385}, \cite{MR1277452}, \cite{MR3245011}, \cite{MR4298750}, \cite{MR4426741}, \cite{XueZhang}). 

Here we propose a new approach, relying on intertwining periods. These are certain meromorphic families of invariant linear forms on induced representations that appear naturally in the spectral side of the relative trace formula, and more directly in the computation of the regularized periods of Eisenstein series. 

Our approach has similarities with the residue method, however, both our perspective and our set-up are quite different from previous applications of the method. The residue method was introduced by Jacquet and Rallis in order to compute the period integral of a residual automorphic form in terms of  the residue of the period integral of a truncated Eisenstein series. In our current work we compute such a residue in order to study period integrals on the inducing data. Furthermore, every application of the method known to the authors relies on vanishing of the regularized period of the Eisenstein series for a generic complex parameter. In contrast, in our work, no such vanishing occurs and the regularized period is expressed in terms of an intertwining period.  

In order to explain our main results we introduce some further notation and terminology. Let $F$ be a number field with ring of adeles $\A$. 
Let $G$ be a reductive group defined over $F$ and $H$ a reductive subgroup. Denote by $A_G$ the maximal split torus in the center of $G$ and let $A_G^+=\Res_{F/\mathbb{Q}}(A_G)(\R_{>0})\hookrightarrow A_G(F\otimes \R)\hookrightarrow A_G(\A)$. The period integral
\[
\P_H(\phi)=\int_{(H(\A)\cap A_G^+)H(F)\bs H(\A)}\phi(h)\ dh
\]
converges for any cuspidal automorphic form on $A_G^+G(F)\bs G(\A)$ \cite{MR1233493}. 

An irreducible, cuspidal automorphic representation $\pi$ of $G(\A)$ is called \emph{$H(\A)$-distinguished} if its central character is trivial on $A_G^+$ and $\P_H$ does not vanish identically on $\pi$. 

Recall that $\pi$ is isomorphic to a restricted tensor product $\otimes_v' \pi_v$ over all places $v$ of $F$. We say that $\pi$ is \emph{locally $H$-distinguished} if the space $\Hom_{H(F_v)}(\pi_v,\C)$ of $H(F_v)$-invariant linear forms on $\pi_v$ is non-zero for every place $v$ of $F$. 

It is an easy observation that if $\pi$ is $H(\A)$-distinguished then it is locally $H$-distinguished. In this work we study several cases where the converse does not hold.

Let $D$ be a central division $F$-algebra of degree $d$, that is, so that $\dim_F(D)=d^2$. 

For $m\in \Z_{\ge 0}$ let $G_D(m)$ be the algebraic group defined over $F$ with group of rational points 
\[
G_D(m,F)=\GL_m(D)
\]
and set $D_\A=D\otimes_F \A$.
 The Jacquet-Langlands correspondence, established in \cite{MR2329758} and \cite{MR2684298} (relying on \cite{MR771672}) attaches to any irreducible discrete automorphic representation
$\pi$ of $G_D(m,\A)=\GL_m(D_\A)$ an irreducible discrete automorphic representation $\JL(\pi)$ of $G_F(dm,\A)=\GL_{dm}(\A)$.

We consider distinction problems related with inner forms of general linear groups. Our criterion for global distinction (non-vanishing of period integrals) is in terms of a combination of a global and a local condition. The global condition is expressed in terms of special values of $L$-functions. The local condition is a combination of 
local distinction (existence of invariant linear forms) and another local compatibility condition to the period subgroup. 

In \cite[Corollary 10.3]{MR2930996} a global distinction criterion is obtained for cuspidal representations on general linear groups over a quadratic extension and period integrals over unitary groups. When the unitary group is non-quasi-split a local obstruction of a similar nature occurs. Hence the criterion for distinction obtained in \cite[Corollary 10.3]{MR2930996} is of similar nature to that obtain here. Indeed, the condition there amounts to local distinction, a local obstruction, and a global condition, namely belonging to the image of the base change map, which can be restated as an appropriate Rankin-Selberg $L$-function having a pole, necessarily simple, at $s=1$. 

We point out that our local-global principal for twisted linear periods generalizes to central simple algebras a celebrated result of Waldspurger, \cite[Th\'eor\`eme 2]{MR783511} for quaternion algebras. We observe that Waldspurger's principle actually extends in two directions. One, up to using incidental isomorphisms, is the Gross-Prasad conjectures for special orthogonal groups \cite{MR1186476}, \cite{MR3202556}. The other more natural extension is the result that we obtain here. Furthermore, as a consequence we complete the proof of one direction of a conjecture of Guo and Jacquet.

In this paper, as we deal with inner forms, we moreover have to make use of the Jacquet-Langlands transfer, as in the following statements.

\subsection{The main result: Galois periods}
Let $E/F$ be a quadratic extension of number fields and $D$ a central division $F$-algebra of degree $d$. 
Let $H=G_D(m)$ and $G=\Res_{E/F}(H_E)$ be the Weil restriction of scalars of the base-change of $H$ to $E$.

\begin{theorem}\label{thm main gal odd}
Let $\pi$ be an irreducible, cuspidal automorphic representation of $G(\A)$ such that $\JL(\pi)$ is also cuspidal. If $d$ is odd then the following are equivalent:
\begin{enumerate}
\item\label{cond gal dist} $\pi$ is $H$-distinguished;
\item\label{cond local gal} $\pi$ is locally $H$-distinguished and the Asai $L$-function $L(s,\JL(\pi),\As^+)$ has a pole at $s=1$;
\item\label{cond global pole} the Asai $L$-function $L(s,\JL(\pi),\As^+)$ has a pole at $s=1$;
\item\label{cond gal jl dist} The irreducible, cuspidal, automorphic representation $\JL(\pi)$ of $\GL_{dm}(\A_E)$ is $\GL_{dm}(\A)$-distinguished. 
\end{enumerate}
When these conditions are satisfied the pole of $L(s,\JL(\pi),\As^+)$ at $s=1$ is simple.
\end{theorem}
The case $d=1$ of the theorem is a consequence of the main results of \cite{MR984899} and \cite{MR1344660}. 
For $d$ even our criterion for distinction involves a local obstruction. We refer to the body of the work for its definition. For the case $d=2$ and $m=1$ it was already observed in \cite[Theorem 0.2]{MR1277452}. We say that a cuspidal, automorphic representation $\pi$ of $G$ is $H$-compatible if its local component $\pi_v$ is $H(F_v)$-compatible in the sense of Definition \ref{def compatible} (see also Corollary \ref{cor dist comp} and Remark \ref{rem dist comp}) for every place $v$ of $F$ that is inert in $E$. We remark that a representation of $G(F_v)$ may be $H(F_v)$-distinguished and not $H(F_v)$-compatible. 

\begin{theorem}\label{thm main gal even}
Let $\pi$ be an irreducible, cuspidal automorphic representation of $G(\A)$ such that $\JL(\pi)$ is also cuspidal. If $d$ is even then the following are equivalent:
\begin{enumerate}
\item\label{cond gal dist} $\pi$ is $H$-distinguished;
\item $\pi$ is locally $H$-distinguished and $H$-compatible and the Asai $L$-function $L(s,\JL(\pi),\As^+)$ has a pole at $s=1$;
\item $\pi$ is $H$-compatible and the Asai $L$-function $L(s,\JL(\pi),\As^+)$ has a pole at $s=1$.
\end{enumerate}
When these conditions are satisfied the pole of $L(s,\JL(\pi),\As^+)$ at $s=1$ is simple and furthermore
the irreducible, cuspidal, automorphic representation $\JL(\pi)$ of $\GL_{dm}(\A_E)$ is $G_F(dm)=\GL_{dm}$-distinguished. 
\end{theorem}

For general $D$, Flicker and Hakim applied in \cite[Theorem 0.5]{MR1277452} a simple relative trace formula to prove a variant of these two theorems under some local restrictions. Our result removes these restrictions and further explicates the local obstruction. 
\subsection{The main result: Linear periods}\label{ss main lin}
Let $D$ be a central division $F$-algebra of degree $d$. Once again, the criterion for distinction involves a local obstruction. We say that an irreducible, cuspidal, automorphic representation $\pi$ of $G_D(2m,\A)$ is $G_D(m)\times G_D(m)$-compatible if its local component $\pi_v$ is $G_D(m,F_v)\times G_D(m,F_v)$-compatible in the sense of Definition \ref{def compatible} (see also Corollary \ref{cor dist comp} and Remark \ref{rem dist comp}) for every place $v$ of $F$. If $d$ is odd, by Lemma \ref{lem compatible}, $\pi$ is automatically $H$-compatible whenever $\JL(\pi)$ is cuspidal.
 
\begin{theorem}\label{thm main lin}
Let $\pi$ be an irreducible, cuspidal automorphic representation of $G_D(2m,\A)=\GL_{2m}(D_\A)$ such that $\JL(\pi)$ is also cuspidal. 

\begin{enumerate}
\item If  $d$ is odd then the following are equivalent:
\begin{enumerate}
\item\label{cond dist} $\pi$ is $G_D(m)\times G_D(m)$-distinguished;
\item\label{cond L} The central value $L(\frac12,\JL(\pi))$ is nonzero and the exterior square $L$-function $L(s,\JL(\pi),\wedge^2)$ has a pole at $s=1$.
\item\label{cond local lin} $\pi$ is locally $H$-distinguished, $L(\frac12,\JL(\pi))$ is nonzero and the exterior square $L$-function $L(s,\JL(\pi),\wedge^2)$ has a pole at $s=1$.
\item\label{cond lin jl dist} The irreducible, cuspidal, automorphic representation $\JL(\pi)$ of $\GL_{2dm}(\A)$ is $\GL_{dm}(\A)\times\GL_m(\A)$-distinguished. 
\end{enumerate}
\item If  $d$ is even then the following are equivalent:
\begin{enumerate}
\item\label{cond dist} $\pi$ is $G_D(m)\times G_D(m)$-distinguished;
\item\label{cond L} $\pi$ is locally $G_D(m)\times G_D(m)$-distinguished and $G_D(m)\times G_D(m)$-compatible, \[L(\frac12,\JL(\pi))\ne 0\] and the exterior square $L$-function $L(s,\JL(\pi),\wedge^2)$ has a pole at $s=1$.
\end{enumerate}
Furthermore, if these two equivalent conditions hold then the irreducible, cuspidal, automorphic representation $\JL(\pi)$ of $\GL_{2dm}(\A)$ is $\GL_{dm}(\A)\times \GL_{dm}(\A)$-distinguished. 
\item In either case, if the equivalent conditions hold then the pole of $L(s,\JL(\pi),\wedge^2)$ at $s=1$ is simple.
\end{enumerate}
\end{theorem}
This solves, in particular, \cite[Conjecture 1.1]{MR3299843} (which claims that if $\pi$ admits a linear period then so does $\JL(\pi)$). 

In the case $d=1$, it is well-known that
for an irreducible, cuspidal representation $\pi$ of $\GL_{2m}(\A)$ the following conditions are equivalent:
\begin{enumerate}
\item $\pi$ is $\GL_m(\A)\times\GL_m(\A)$-distinguished;
\item $L(\frac12,\pi)\ne 0$ and $L(s,\pi,\wedge^2)$ has a pole at $s=1$
\end{enumerate}
and when these conditions are satisfied the pole is simple.
This is a consequence of the work of Friedberg and Jacquet \cite{MR1241129}. Alternatively, see \cite[Theorem 4.7]{MR3372874} for a more direct approach.

\subsection{The main result: Twisted linear periods}
Let $E/F$ be a quadratic extension of number fields. Let $D$ be a central division $F$-algebra of degree $d$ and $m\in \N$ be such that $E$ imbeds in $\M_m(D)$ (the space of $m\times m$ matrices with entries in $D$). In particular, $dm$ is even. Fix such an imbedding and let $G=G_D(m)$ and $H=C_G(E)$ be the centralizer of $E$ in $G$. 

As in the previous cases the criterion for distinction involves a local obstruction. We say that a cuspidal, automorphic representation $\pi$ of $G$ is $H$-compatible if its local component $\pi_v$ is $H(F_v)$-compatible in the sense of Definition \ref{def compatible} (see also Corollary \ref{cor dist comp} and Remark \ref{rem dist comp}) for every place $v$ of $F$. Again, a representation of $G(F_v)$ may be $H(F_v)$-distinguished and not $H(F_v)$-compatible. We remark that when $md=2$ this does not happen.

\begin{theorem}\label{thm main twlin}
Let $\pi$ be an irreducible, cuspidal automorphic representation of $G_D(m,\A)=\GL_m(D_\A)$ such that $\JL(\pi)$ is also cuspidal. Then the following are equivalent:
\begin{enumerate}
\item\label{part globdist} $\pi$ is $H(\A)$-distinguished;
\item\label{part locond} $\pi$ is locally $H$-distinguished and $H$-compatible, $L(\frac12,\BC_F^E(\JL(\pi)))\ne 0$ and $L(s,\JL(\pi),\wedge^2)$ has a pole at $s=1$.
Here, $\BC_F^E$ stands for quadratic base-change. \end{enumerate}
\end{theorem}
This is a generalization of a famous local global principle \cite[Th\'eor\`eme 2]{MR783511} of Waldspurger for inner forms of $\GL_2$.
For the case $d=1$ it is proved in \cite[Corollary 1.3]{MR4255059}, using the residue method, that \eqref{part globdist} implies the $L$-value conditions that $L(\frac12,\BC_F^E(\JL(\pi)))\ne 0$ and $L(s,\JL(\pi),\wedge^2)$ has a pole at $s=1$. For $d\leq 2$, it is proved in \cite{MR4226986} and \cite{XueZhang} under local constraints. Next we explain how our result provides one implication of the Guo-Jacquet conjecture.

The Guo-Jacquet conjecture (see the conjecture in the introduction of \cite{MR1382478}) relates between the twisted linear and linear cases via the Jacquet-Langlands correspondence. The conjecture consists of two implications and our main result proves the first one. 
\begin{corollary}\label{cor gj} In the notation of Theorem \ref{thm main twlin} if $\pi$ is $H(\A)$-distinguished then both $\JL(\pi)$ and $\JL(\pi)\otimes \eta_{E/F}$ are $\GL_{dm/2}(\A) \times \GL_{dm/2}(\A)$-distinguished where $\eta_{E/F}$ is the quadratic idele class character of $\A^\times$ associated to $E/F$ by class field theory composed with determinant.
\end{corollary}
\begin{proof}
Since
\[L(s,\BC_F^E(\JL(\pi)))=L(s,\JL(\pi))L(s,\JL(\pi)\otimes \eta_{E/F}),
\]
both factors on the right hand side are entire and moreover  
\[
L(s,\JL(\pi),\wedge^2)=L(s,\JL(\pi)\otimes \eta_{E/F},\wedge^2),
\]
the corollary is a direct consequence of Theorem \ref{thm main twlin} and the discussion at the end of Section \ref{ss main lin} (the Friedberg-Jacquet result for linear periods).
\end{proof}

In his paper \cite{MR1382478}, Guo suggests a relative trace formula approach to the problem. This approach was pursued in many subsequent works of which we mention \cite{MR3299843} and \cite{MR3805647}. We also mention a project started by Huajie Li, aiming to prove the full relative trace formula suggested by Guo. The project is currently pursued by Chaudouard and Li. See \cite{MR4424024}, \cite{MR4350885} and \cite{MR4681295} for results in this direction obtained at the Lie algebra level.

A generalization of the Guo-Jacquet conjecture was formulated in \cite[Conjucture 1.1]{XueZhang} to allow twists by characters in the special case where either $D=F$ or $D$ is quaternionic. Xue and Zhang suggest a new relative trace formula comparison. By comparing the elliptic parts they obtain both implications of their conjecture under some local restrictions. For the direct implication, our result removes the restriction on the dimension of $D$ as well as the local restrictions when the twisting character is trivial.

There is also a converse statement in the Guo-Jacquet conjecture (see also \cite[Conjecture 1.1 (ii)]{XueZhang}). We adress it in Section \ref{sec::GJconverse}. Our local-global principle allows us to get, under local assumptions, a partial version of this converse in Theorem \ref{thm::GJconverse}. If $D=F$ or $D$ is a quaternion algebra, then Xue and Zhang also obtained a partial version of this converse in \cite[Theorem 1.5]{XueZhang}, under local and global assumptions. Our local assumptions and their local assumptions are different.

\begin{theorem}
 Let $\pi$ be an irreducible, cuspidal automorphic representation of $\GL_n(\BA)$ where $n$ is even, and write $n=2^a b$ where $a\geq 1$ and $b$ is odd. Assume that:
\begin{enumerate}
\item  at all places $v$ of $F$ that are inert in $E$ (i.e. such that $E_v$ is a field)
\begin{itemize}
\item either $\pi_v$ is a discrete series,
\item or $\pi_v$ has odd essentially square-integrable support (see Definition \ref{df odd esi})\footnote{for example $\pi_v$ could be any unitary principal series induced from the Borel subgroup}. 
\end{itemize}
\item \label{global shal} $L(\frac12,\BC_F^E(\pi))\ne 0$ and $L(s,\pi,\wedge^2)$ has a pole at $s=1$. 
\end{enumerate}
Then either $\pi$ is $\GL_{n/2}(\A_E)$-distinguished or there exists a central $F$-division algebra $D$ of degree $2^a$ such that $E$ imbeds in $D$, and there exists  an irreducible, cuspidal automorphic representation $\pi'$ of 
$\GL_b(D_{\BA})$ with $\JL(\pi')=\pi$, such that $\pi'$ is $\GL_b(C_{\A_E})$-distinguished where $C$ is the centralizer of $E$ in $D$.
\end{theorem}

We believe that our method will have applications to study distinction for more symmetric pairs, both globally as in this paper, and locally as in \cite{epsilonintertwining}.

\subsection{Our technique of proof and the main local result}
Since the structure of proof of the main Theorems is similar in all three cases, we unify notation and consider Galois, linear and twisted linear periods at once.  However, along the way, some statements require a case by case consideration for the nuances in their proofs.

We fix a triple $(G,H,\theta)$ where the group $G$ and the period subgroup $H$ correspond to the set-up in one of our three main results above and $\theta$ is the involution such that $H=G^\theta$. We point out that in all cases there is a central $F'$-division algebra $\D$, with $F'=E$ in the Galois case and $F'=F$ otherwise, and $a\in \N$ such that $G(F)=\GL_a(\D)$. 
For an irreducible, cuspidal automorphic representation $\pi$ of $G(\A)$ such that $\JL(\pi)$ is also cuspidal we consider the product of $L$-functions
\[
\L(s,\pi,\theta)=\begin{cases}
L(2s,\JL(\pi),\As^+) & \text{for Galois periods} \\
L(s+\frac12,\JL(\pi))^2\, L(2s,\JL(\pi),\wedge^2) & \text{for linear periods} \\
L(s+\frac12,\BC_F^E(\JL(\pi)))\,L(2s,\JL(\pi),\wedge^2) & \text{for twisted linear periods}.
\end{cases}
\]
Our local-global principle, Theorem \ref{thm::Main}, is the equivalence of the following two conditions:
\begin{enumerate}
\item $\pi$ is $H(\A)$-distinguished;
\item $\pi$ is locally $H$-distinguished and $H$-compatible and $\L(s,\pi,\theta)$ has a pole at $s=0$.
\end{enumerate}
Theorems \ref{thm main gal odd}, \ref{thm main gal even}, \ref{thm main lin} and  \ref{thm main twlin} follow, by applying the functional equation of the corresponding $L$-functions and some further local results obtained along the way.

In order to prove Theorem \ref{thm::Main} we \emph{double} the set-up. Consider a triple $(G', H',\theta')$ of the same type as $(G,H,\theta)$ but with double the rank. In particular, $G'(F)=\GL_{2a}(\D)$. We consider the standard parabolic subgroup $P$ of $G'$ with Levi factor $M= G\times G$.
In all three cases there is a unique open $P$-orbit on $G'/H'$ and its stabilizer in $M$ is the $\theta$-twisted diagonal imbedding of $G$, $\{(g,\theta(g)):g\in G\}$ and there is a unique closed $P$-orbit on $G'/H'$ with stabilizer $H\times H$ in $M$. 

The technical heart of this work is a local result that we explain first. Fix a place $v$ of $F$ and let $\pi_v$ be a smooth, irreducible representation of $G_v=\GL_a(\D\otimes_F F_v)$ such that $\pi_v^\vee \simeq \pi_v^\theta$ (its contragradient is isomorphic to its $\theta$-twist). We similarly write $X_v=X(F_v)$ for any algebraic group $X$ defined over $F$. For $s\in \C$ let $\pi_v[s]$ be the twist of $\pi_v$ by $\abs{\cdot}^s\circ \nu$ where $\nu$ is the reduced norm on $G_v$. Let $I(s)=\pi_v[s]\times \pi_v[-s]$ be the representation of $G'_v$ obtained by normalized parabolic induction from $P_v$ and the representation $\pi_v[s]\otimes \pi_v[-s]$ of $M_v$.

The results of \cite{MR2401221} and \cite{MR1274587} imply that there exists a non-zero meromorphic family 
\[
J(s)=\J_{\pi_v\otimes \pi_v}(s)
\] 
of $H'_v$-invariant linear forms on $I(s)$. These are the local analogues of the global intertwining periods introduced in \cite{MR1625060}. Furthermore, restricted to the $H_v'$-invariant subspace of sections supported on the open $P$-orbit in $G'/H'$, $J(s)$ is holomorphic and non-zero at every $s\in \C$. Consequently, there exists $l\in \Z_{\ge 0}$ such that the leading term of $J(s)$ at $s=0$, namely, 
\[
\J_0=\underset{s\to 0}{\lim} s^l J(s)
\] 
is a non-zero $H'_v$-invariant linear form on $I(0)=\pi_v\times \pi_v$.
If $l>0$ then $\J_0$ is supported away from the open orbit. By analyzing the $P$-orbits on $G'/H'$ and applying the geometric lemma of Bernstein and Zelevinsky we can then deduce that $\Hom_{H_v}(\pi_v,\C)\ne 0$, that is, that $\pi_v$ is $H_v$-distinguished. Thus, if $\pi_v$ is not $H_v$-distinguished then $J(s)$ is holomorphic at $s=0$.
If $\pi_v$ is $H_v$-distinguished, the determination of $l$, the order of pole of $J(s)$ at $s=0$, is a delicate problem. 
Assume in addition that $\JL(\pi_v)$ is generic. Using the techniques developed in \cite{MR4308058} we show  in Section \ref{sec order local pole}
that $l$ is bounded by the order of pole at $s=0$ of the $L$-factor $\L(s,\pi_v,\theta_v)$ and we characterize the condition for equality between these two integers by the property we called $H_v$-compatibility above. Namely, we prove in Theorem \ref{thm mainloc} the following result.

\begin{theorem}
Let $\pi$ be an irreducible and distinguished generic unitary repreresentation of $G'$, or assume more generally that $\pi$ belongs to $\Pi_\D(-\frac12,\frac12)$ as in Definition \eqref{df exp 12}. Then we have the following inequality between the orders of the poles at $s=0$:
\[
\Ord_{s=0}(\J_{\pi\otimes \pi}(s))\le \Ord_{s=0}(\L(s,\pi,\theta)).
\]
Moreover equality holds if and only if $\pi$ is $H'$-compatible (see Definition \ref{def compatible}).
\end{theorem}

Our global treatment is inspired by \cite[Example 6]{MR1625060}. 
The global version of the above idea is encoded in the so called Maass-Selberg relations (see for example \cite{MR1625060}, \cite{MR2010737}). 
These relations can be viewed as a global version of the geometric lemma, taking into account the contribution of intertwining periods. 

Let $\pi$ be an irreducible, cuspidal automorphic representation of $G(\A)$ and let $I(\pi,s)=\pi[s]\times \pi^*[-s]$ be the representation of $G'(\A)$ obtained by normalized parabolic induction from $P(\A)$ and $\pi[s]\otimes \pi^*[-s]$ where $\pi^*=(\pi^\vee)^\theta$ is the $\theta$-twist of the contragradient of $\pi$ and, as in the local set-up, $\pi[s]$ is the twist of $\pi$ by $\abs{\cdot}\circ \nu$ and $\nu$ is the reduced norm on $G(\A)$. For $\varphi\in I(0)$ let $E(\varphi,s)$ be the corresponding Eisenstein series.
In this setup, the Maass-Selberg relations take the form
\[
\P_{H'}(\Lambda^TE(\varphi,s))=J(\varphi,s)+ \frac{e^{sT}}{s}I(\varphi)- \frac{e^{sT}}{s}I(M(s)\varphi).
\]
Here, $T$ is a positive enough truncation parameter, $\Lambda^T$ is Zydor's relative truncation operator with respect to $(G',H')$ defined in \cite{MR4411860}, 
 $J(s)$ is the open ($P$-orbit in $G'/H'$) intertwining period on $I(\pi,s)$, $I$ is the closed intertwining period on $I(0)$, which, if the central character of $\pi$ is moreover trivial on $A_G^+$, is zero except when $\pi$ is $H(\A)$-distinguished, and $M(s):I(\pi,s)\rightarrow I(\pi^*,-s)$ is the standard intertwining operator.
 
Assume that the central character of $\pi$ is trivial on $A_G^+$, and assume further that $\JL(\pi)$ is cuspidal and that $\pi^*=\pi$, (this equality is automatic if $\pi$ is locally $H$-distinguished thanks to strong multiplicity one and local results on distinction). A careful analysis of the Maass-Selberg relations multiplied by $s$ when $s\to 0$ shows that $\pi$ is $H(\A)$-distinguished if and only if the global intertwining period $J(s)$ has a (necessarily simple) pole at $s=0$, as we prove in Theorem \ref{thm pole of GIP}. 

As a consequence, for decomposable $\varphi$, by local multiplicity one and unramified computations (due to Jacquet-Lapid-Rogawski \cite{MR1625060} in the Galois case and Suzuki-Xue \cite{MR4721777} following Offen \cite{MR2060496} and Lapid-Offen \cite{MR3776281} in the other cases) the global intertwining period $J(\varphi,s)$
can be factorized as 
\[
J(\varphi,s)=J_S(\varphi_S,s)\frac{\L^S(s,\pi,\theta)}{\DL^S(s,\pi,\theta)}.
\] 
Here $S$ is a finite set of places of $F$ so that the data is non-archimedean and unramified outside of $S$, the subscript $S$ stands for the product over places in $S$ and the superscript for the corresponding partial $L$-functions away from $S$. The denominator is defined in terms of another $L$-function $\DL(s,\pi,\theta)$ prescribed to the data $\pi$ and $\theta$ in Section \ref{ss Lfactors for pairs}. It is well known that $\DL(s,\pi,\theta)$ is holomorphic and nonzero at $s=0$. If $\pi$ is locally $H$-distinguished, our main local results from Section \ref{sec order local pole} imply that the order of the pole at $s=0$ of $J_S(s)$ is at most that of $\L_S(s,\pi,\theta)$ and equality holds if and only if $\pi$ is $H$-compatible. Our local-global principle, Theorem \ref{thm::Main}, follows from the above discussion.

Finally, in Appendix \ref{app failure}, we prove the failure of the \emph{naive} local-global principle in the case of Galois periods. That is, we show existence of cuspidal automorphic representations that are locally $H$-distinguished but not $H(\A)$-distinguished.

\textbf{Acknowledgement.} We thank Rapha\"el Beuzart-Plessis for useful conversations leading to Appendix \ref{app failure}.

	\section{Notation and preliminaries}

	Let $F$ be either a number field-\emph{the global set-up} or a local field of characteristic zero-\emph{the local set-up}. When $F$ is a number field, denote by $\A = \A_F$ its ring of ad\`{e}les and let $F_v$ be the completion of $F$ with respect to a place $v$ of $F$. We further denote by $R_v=R\otimes_F F_v$ the completion of an $F$-vector space $R$. If $R$ is an $F$-algebra then $R_v$ is an $F_v$-algebra. Denote by $\abs{\cdot}$ the standard absolute value on $\A^*$ in the global set-up and on $F^*$ in the local set-up.

Let $D$ be a central simple $F$-algebra. For $a\in \N$ denote by $G_D(a)$ the algebraic group defined over $F$ such that its rational points are given by 
\[
G_D(a,F)=\GL_a(D).
\]

We denote by $e$ the identity element in a group. For integers $a,\, b$ let $[a,b]$ be the interval of all integers $x$ such that $a\le x\le b$.
For $r\in \N$ let $\mathfrak{S}_r$ be the permutation group on $[1,r]$.

For nonzero meromorphic functions $\alpha(s)$ and $\beta(s)$ on $\C$ we write $\alpha(s)\sim\beta(s)$ if the quotient $\alpha(s)/\beta(s)$ is a nowhere vanishing entire function. We further write $\alpha(s)\underset{A^\times}{\sim}\beta(s)$ if $\alpha(s)/\beta(s)$ belongs to the unit group $A^\times$ of a subring $A$ of the ring of meromorphic functions on $\C$. 

Let $\ell_\lambda$ be a family of linear forms on a complex vector space $V$ parameterized by a finite dimensional complex vector space $\lambda\in \fra$. We say that $\ell_\lambda$ is holomorphic at $\lambda=\lambda_0$ if $\lambda\mapsto \ell_\lambda(v)$ is holomorphic at $\lambda=\lambda_0$ for all $v\in V$ and we say that $\ell_\lambda$ is meromorphic at $\lambda=\lambda_0$ if there is a non-zero polynomial $p(\lambda)$ on $\fra$ such that $p(\lambda)\ell_\lambda$ is holomorphic at $\lambda=\lambda_0$. If $\fra=\C$ we write 
\[
\Ord_{\lambda=\lambda_0}(\ell_\lambda)=k\in \Z
\] 
if $(\lambda-\lambda_0)^k\ell_\lambda$ is holomorphic and not identically zero at $\lambda=\lambda_0$. We similarly define the order of pole $\Ord_{\lambda=\lambda_0}(f(\lambda))$ for any meromorphic function $f$ on $\C$.
\subsection{Generalities on reductive groups}
	
If $X$ is an algebraic variety defined over $F$, we sometimes write $X = X(F)$ for its $F$-points by abuse of notation. For an algebraic group $Q$ defined over $F$ we denote by $X^{\ast}(Q)$ the abelian group of $F$-rational characters of $Q$. We set $\fra_Q^{\ast} = X^{\ast}(Q) \otimes_{\BZ} \BR$ and let $\fra_Q = \Hom_{\BR} (\fra_Q^{\ast}, \BR)$ be its dual vector space. Let $\fra_\C=\fra\otimes_\R\C$ be the complexification of a real vector space $\fra$. We denote by $\delta_Q$ the modulus character of $Q(\BA)$ resp. $Q(F)$ when $F$ is a number field resp. a local field. 
	
	Let $G$ be a connected reductive group defined over $F$. Fix a maximal $F$-split torus $A_0$ of $G$ and a minimal parabolic subgroup $P_0$ of $G$ that contains $A_0$. Parabolic subgroups of $G$ containing $P_0$ resp. $A_0$ are called standard resp. semi-standard. If $P$ is a semi-standard parabolic subgroup of $G$, then it contains a unique Levi subgroup $M$ containing $M_0$, the centralizer of $A_0$ in $G$. Let $U$ be the unipotent radical of $P$, then $P = M \ltimes U$ is called the standard Levi decomposition of $P$. We denote by $A_P$ or $A_M$ the split center of $M$. By a standard Levi of $G$ we mean the Levi subgroup in the standard Levi decomposition of a standard parabolic subgroup. 
	
	In what follows, unless otherwise specified, by a parabolic (resp. Levi) subgroup we always mean a standard parabolic (resp. Levi) subgroup. By writing $P= M U$ we mean the standard Levi decomposition of $P$ with standard Levi $M$ and unipotent radicals $U$.
		
	Let $P = MU \subset Q = LV$ be two parabolic subgroups. There is a canonical direct sum decomposition $\fra_{M} = \fra_L \oplus \fra_M^L$. A similar decomposition holds for the dual space. Write $\fra_0$ and $\fra_0^*$ for $\fra_{M_0}$ and $\fra_{M_0}^*$ respectively. We denote by $R(A_M,L)$ (resp. $R(A_M,P \cap L)$) the set of roots of $A_M$ acting on the Lie algebra of $L$ (resp. of $P \cap L$). For $\alpha \in R(A_M,L)$ we write  $\alpha > 0$ if $\alpha \in R(A_M,P \cap L)$ and $\alpha < 0$ otherwise. Recall that $R(A_0,L)$ forms a root system and let $\Delta_0^L$ be its basis of simple roots with respect to $P_0 \cap L$. Let $\Delta_M^L$ be the set of non-zero restrictions to $A_M$ of the elements of $\Delta_0^L$. The set $\Delta_M^L$ forms a basis of $(\fra_M^L)^*$. We sometimes also denote $\Delta_M^L$ by $\Delta_P^Q$. When $L =G$, we often omit the superscript $G$. We also define the positive chamber
	\begin{align*}
		\fra_P^+ = \{ H \in \fra_P \mid \langle H, \alpha \rangle > 0, \forall \alpha \in \Delta_P \}.
	\end{align*}
	Note that $R(A_0,G)$ lies in $\fra_0^*$. Let $\rho_0 \in \fra_0^*$ be the half-sum of the positive roots of $A_0$ (counted with multiplicities), and $\rho_P$ be its projection on $\fra_{M}^*$. 
	
	
		
	Let $W = N_G(A_0) / M_0$ be the Weyl group of $G$ with respect to $A_0$. For a Levi subgroup $M$ of $G$ let $W^M = N_M(A_0) / M_0$ be the Weyl group of $M$ with respect to $A_0$. For two Levi subgroups $M$ and $M'$ let ${}_{M'}W_M$ be the set of Weyl elements $w\in W$ that are of minimal length in $W^M w W^{M'}$. It is a complete set of representatives for the double cosets $W^{M'} \backslash W /W^M$. For two Levi subgroups $M \subset L$ let $W^L(M)$ be the set of elements $w \in W^L$ such that $w$ is of minimal length in $wW^M$ and $w M w^{-1}$ is a standard Levi subgroup of $L$. Set $W(M) = W^G(M)$. According to \cite[I.1.7, I.1.8]{MR1361168}, one can decompose elements of $W(M)$ into products of elementary symmetries attached to simple roots in $\Delta_{M'}^G$ for Levi subgroups $M'$ of $G$ that are conjugate to $M$. In turn, this allows one to define a length function $\ell_M$ on $W(M)$. There is a unique element in $W^L(M)$ for which $\ell_M$ is maximal, and we denote it by $w_M^L$.
	
	When $F$ is a number field, we take a maximal compact subgroup $K$ of $G(\BA)$ that is adapted to $M_0$ (\cite[I.1.4]{MR1361168}). Let $P = MU$. We have the Harish-Chandra map $H_M : M(\BA) \ra \fra_{M}$ given by
	\begin{align*}
		e^{ \langle \chi, H_M(m) \rangle}  = \lvert \chi (m) \rvert, \quad \chi \in X^{\ast}(M), \ m \in M(\BA).
	\end{align*}
	We then extend $H_M$ to  $G(\BA)$ as the unique left $U(\BA)$-invariant and right $K$-invariant extension via the Iwasawa decomposition $G(\BA)=U(\BA)M(\BA)K$. The modulus character on $P(\BA)$ is given by $e^{ \langle 2\rho_P, H_M(\cdot) \rangle}$. We further denote by $G(\A)^1$ the kernel of $H_G$.
	
	When $F$ is a local field, we take a maximal compact subgroup $K$ of $G(F)$ that is adpated to $M_0$. Let $P = MU$. Similarly we have the Harish-Chandra map given by
	\begin{align*}
		e^{\langle \chi,H_M(m) \rangle}  = \lvert \chi (m) \rvert, \quad \chi \in X^{\ast}(M), \ m \in M(F).
	\end{align*}
	We then extend $H_M$ to a function on $G(F)$ that is left $U(F)$-invariant and right $K$-invariant. Likewise, the modulus function on $P(F)$ is given by $e^{ \langle 2 \rho_P, H_M(\cdot) \rangle}$.

\subsection{Representations}\label{sec::reps}
For an algebraic group ${\mathbf Q}$ defined over $F$ write in this section $Q$ for ${\mathbf Q}(F)$ if $F$ is local and for ${\mathbf Q}(\A)$ if $F$ is global. 
Let $\mathbf{G}$ be a reductive linear algebraic group defined over $F$.

When $F$ is $p$-adic, by a representation of $G$ we always mean a smooth admissible representation with coefficients in $\mathbb{C}$. When $F$ is archimedean, by a representation of $G$ we mean a smooth admissible Fr\'echet representation of moderate growth (see \cite{MR1013462} or \cite[Chapter 11]{MR1170566}). When $F$ is global, for automorphic representations, we follow \cite[Section 2.7]{MR4426741} for the notion of smooth automorphic representations of $G$, but only consider $K$-finite vectors in the space of such a representation since part of the literature that we use is written in this setting.  In particular their archimedean components are Harish-Chandra modules. On the other hand for our local results, we consider smooth admissible representations. The correspondence between these two versions of archimedean representations is given by the Casselman-Wallach completion functor (\cite[Chapter 11]{MR1170566}) in one direction, and by taking $K$-finite vectors in the other. 

When using local and global results together, especially when dealing with invariant linear forms, the coherence of these two approaches requires the results of \cite{MR1176208}, automatic continuity of invariant linear forms. However, when this confusion does not create any ambiguity, in order to simplify notation, we will sometimes identify smooth admissible representations and their underlying Harish-Chandra module. For example we say that the smooth admissible archimedean representation $\pi_v$ is an archimedean component of the automorphic representation $\pi$, when actually only its underlying Harish-Chandra module is. 

For a smooth representation $\pi$ of $G$ and a subgroup ${\bf H}$ of ${\bf G}$ we denote by $\Hom_H(\pi,\C)$ the space of (continuous) $H$-invariant linear forms on the space of $\pi$. If $\pi$ is a Harish-Chandra module we use the same notation for the space of linear forms on the space of $\pi$ that are $H\cap K$-invariant and such that their kernel contains the image of the action of $\Lie(H)$.

We tacitly use the following results throughout the paper:
\begin{itemize}
\item  Assume that ${\bf H}$ be the group of fixed points of an involution on ${\bf G}$. 
Let $\pi$ be a finite length smooth admissible representation of $G$ in the sense of \cite[Chapter 11]{MR1170566}, and let $\pi_{f}$ be the underlying Harish-Chandra module of $K$-finite vectors in $\pi$. Then restriction to $\pi_f$ induces an isomorphism
\[
\Hom_H(\pi,\C)\simeq \Hom_H(\pi_f,\C).
\]
 \item Let $\pi_s$ be a holomorphic family of smooth admissible representations ($s\in \C$) of finite length on the same vector space $V$. Consider a meromorphic family of continuous Linear forms $\ell_s$ on $V$. Then there exists a $K$-finite vector $v$ in $\pi_f$ such that
 \[
 \Ord_{s=0}(\ell_s)= \Ord_{s=0}(\ell_s(v)).
 \]
\end{itemize}
The first fact is a consequence of \cite[Th\'eor\`eme 1]{MR1176208}. Since any term in the Laurent expansion of $\ell_s$ is a continuous linear form (it is given by Cauchy's integral formula), the second fact is a consequence of the subspace of $K$-finite vectors in $V$.

Finally we observe that a cuspidal automorpic representation is unitary only up to a character twist.
	
\subsubsection{Parabolic induction}
Let ${\mathbf{P}}={\mathbf{MU}}$ be a parabolic subgroup of ${\mathbf{G}}$ and $\sigma$ a representation of $M$. Denote by $I_P^G(\sigma)$ the representation of $G$ defined by normalized parabolic induction. 

For $\lambda \in \fra_{M,\C}^*$ and $\varphi \in I_P^G(\sigma)$ write $\varphi_{\lambda} (g) = e^{ \langle \lambda, H_M(g) \rangle} \varphi (g)$ for the twist of $\varphi$ by $\lambda$. Let $I_P^G(\sigma,\lambda)$ be the representation of $G$ on the space of $I_P^G(\sigma)$ defined as \begin{align*}
		(I_P^G(g,\sigma,\lambda)\varphi)_{\lambda}(x) = \varphi_{\lambda} (xg).
	\end{align*}
	Let $\sigma[\lambda]$ denote the representation of $M$ on the space of $\sigma$ given by 
	\[\sigma[\lambda](m) = e^{\langle \lambda, H_M(m) \rangle} \sigma (m).\] The map $\varphi \mapsto \varphi_{\lambda}$ is an isomorphism of representations $I_P^G(\sigma,\lambda) \ra I_P^G(\sigma[\lambda])$.

	Let $Q = LV$ be a parabolic subgroup of $G$ containing $P$. Transitivity of parabolic induction is the natural isomorphism $F: I_P^G(\sigma) \ra I_Q^G(I_{P \cap L}^L (\sigma))$, $\varphi \mapsto F_{\varphi}$ of $G$-representations defined by
	\begin{align*}
		F_{\varphi}(g)(l) = \delta_Q^{-1/2}(l)\varphi(lg),\ \ \ l\in L,  g\in G.
	\end{align*}
	For $\lambda \in \fra^*_{L,\C}$, we have
	\begin{align}\label{formula::holo-section-transitivity}
		(F_{\varphi})_{\lambda}(g)(l) =  e^{ -\langle \lambda + \rho_Q , H_L(l) \rangle} \varphi_{\lambda}(lg).
	\end{align}
	For $g \in G$ and $\varphi \in I_P^G (\sigma)$, set $\varphi [g] : = F_{\varphi}(g) \in I_{P \cap L}^L (\sigma)$. 
\subsubsection{Intertwining operators}
Let ${\mathbf{P}}={\mathbf{MU}}$ be a parabolic subgroup of $\mathbf{G}$, $\sigma$ a representation of $M$ and
$w \in W(M)$ and choose a representative $n$ of $w$ in $N_G(A_0)$. Let $M' = wMw^{-1}$ and ${\mathbf{P'}}={\mathbf{M'U'}}$ be the corresponding parabolic subgroup. 
Let $w\sigma$ be the representation of $M'$ on the space of $\sigma$ given by $w\sigma(m)=\sigma(n^{-1}mn)$, $m\in M'$. (The isomorphism class of this representation is independent of $n\in w$).
We denote by 
\begin{align*}
	M(n,\sigma,\lambda) : I_P^G(\sigma,\lambda) \ra I_{P'}^G(w\sigma,w\lambda)
\end{align*} 
the standard intertwining operator defined by the meromorphic continuation of the integral
\begin{align}\label{formula::defn--intertwinner}
	( M(n,\sigma,\lambda)\varphi )_{w \lambda} (g) = \int  \varphi_{\lambda} (n^{-1} ug) du
\end{align} 
convergent for $\Re(\lambda)$ in some positive cone in $\fra_M^*$. Here the integral is over the quotient $U' \cap wUw^{-1} \backslash U' $ in the local case and the automorphic quotient ${\mathbf{U'}(F)}(U' \cap wUw^{-1}) \backslash U' $ in the global case.

\subsection{Symmetric pairs-generalities}\label{ss sym}

Let $F$ be a local field, ${\mathbf{G}}$ a reductive linear algebraic group defined over $F$ and $\theta$ an involution on ${\mathbf{G}}$. Let $G={\mathbf{G}}(F)$ and consider the associated symmetric $G$-space
\[
X=\{x\in G:x=\theta(x)^{-1}\}
\]
with the $G$-action $g\cdot x=gx\theta(g)^{-1}$, $g\in G$, $x\in X$.
For a subgroup $Q$ of $G$ let $Q_x$ denote the stabilizer of $x$ in $Q$. Note that for $x\in X$, $\theta_x=\Ad(x)\circ\theta$ is an involution on $G$ and $Q_x=Q^{\theta_x}$ is the subgroup of $Q$ fixed by $\theta_x$. 

Let $H=G^\theta=G_e$. We refer to $(G,H)$ as a symmetric pair, however, we often introduce the triplet $(G,H,\theta)$ and still refer to it as a symmetric pair.

We follow \cite{MR3541705} and recall the analysis of parabolic orbits on $X$ as well as consequences of the geometric lemma of Bernstein and Zelevinsky.

\subsubsection{Parabolic orbits}\label{ss paraborb}
Fix once and for all a $\theta$-stable maximal split torus $A_0$ and a minimal parabolic subgroup $P_0$ containing $A_0$. Let $w_\star\in W$ be such that $\theta(P_0)=w_\star P_0w_\star^{-1}$. Fix a representative  $n_\star$ of $w_\star$ in $G$ and let $\theta'=\Ad(n_\star^{-1})\circ \theta$ be the corresponding automorphism of $P_0$. It is not necessarily an involution of $\mathbf{G}$, however, $\theta'(P_0)=P_0$ and it defines an involution on $\fra_0^*$ and on $W$ that we still denote by $\theta'$.

Let $P=M U$ be a standard parabolic subgroup of $G$. The double coset space $P\bs G/H$ is in bijection with the $P$-orbits in $G\cdot e\subseteq X$.
In what follows we recall some generalities on $P$-orbits on $X$ from \cite[Section 3]{MR3541705}. 

For every $w\in {}_MW_{\theta'(M)}$ the group $M(w)=M\cap w\theta'(M)w^{-1}$ is a standard Levi subgroup of $G$. 
There is a map $\iota_P$ from $P\bs X$ to the subset of $w\in  {}_MW_{\theta'(M)}$ such that $w\theta'(w)=e$ characterized by the property that for $x\in X$ we have $\iota_P(P\cdot x)=w$\footnote{We deviate from the convention in \cite{MR3541705} where $w$ is replaced by $ww_\star$} if 
\[
Pxn_\star \theta'(P)=Pw\theta'(P).
\]
For $L=M(w)$ the intersection $P\cdot x\cap Lww_\star^{-1}$ is a unique $L$-orbit (and in particular non-empty). If $Q=LV$ is the parabolic subgroup of $G$
with Levi subgroup $L$ then for $y\in P\cdot x\cap Lww_\star^{-1}$ we have $P_y=Q_y=L_y\rtimes V_y$.

We call any element in $P\cdot x\cap Lww_\star^{-1}$ a $P$-\emph{good} representative of $P\cdot x$.

\subsection{The geometric lemma}
Assume that $F$ is a $p$-adic field. Let $P=MU$ be a parabolic subgroup of $G$.
By \cite[Section 1.5]{MR0425030}, we can order the $P$-orbits in $G\cdot e$ as $\{ P \cdot y_i\}_{i=1}^N$ in such a way  that $\cup_{j = 1}^i P\cdot y_j$ is open in $G\cdot e$ for all $i\in [1,N]$. That is, choosing $u_i\in G$ such that $u_i\cdot e=y_i$ we have that 
	\begin{align*}
		Y_i = \cup_{j = 1}^i P u_i H
	\end{align*}
	is open in $G$ for all $i = 1, \cdots,N$. 
	We further choose each representative $y_i$ to be $P$-good. 
	
	Let $\sigma$ be a representation of $M$ and 
	\begin{align*}
		V_i = \{  \varphi \in I_P^G(\sigma) \mid \textup{Supp}(\varphi)  \subset Y_i\}.
	\end{align*}
	By \cite[Proposition 4.1]{MR3541705}, we have
	\begin{align}\label{eq geom lem}
		\Hom_H(V_i / V_{i-1}, \C)  \cong  \Hom_{L_{y_i}} (r_{L,M}(\sigma),\delta_{Q_{y_i}}\delta_Q^{-1/2}),
	\end{align}
	where we set $w_i=\iota_P(P\cdot y_i)$ and $L=M(w_i)$ and let $Q$ be the parabolic subgroup of $G$ with Levi part $L$ and $r_{L,M}$ be the normalized Jacquet functor. 
	This isomorphism motivates the following definition.
	\begin{definition}
		We say that $P \cdot y_i$ is relevant to $\sigma$ if the vector space on the right hand side of \eqref{eq geom lem} is non-zero.
	\end{definition}

\subsection{Symmetric pairs for inner forms of $\GL$}\label{sec: symmetric pair}
	In this work we consider symmetric pairs in three arithmetic families: those associated with \emph{linear} periods, \emph{twisted linear} periods (also known as of Prasad and Takloo-Bighash type) and \emph{Galois} periods for inner forms of general linear groups. 
	We choose explicit realizations for those families of symmetric pairs in a way that is convenient for our analysis of parabolic orbits.

For $k\in \N$ and a ring $R$ denote by $\M_k(R)$ the ring of $k\times k$ matrices with entries in $R$.

	Let $E/F$ be a quadratic field extension and $D$ a central division $F$-algebra. Let $d\in\BN$ be the degree of $D$ over $F$, that is, the positive integer such that $d^2$ is the dimension of $D$ over $F$. Fix once and for all $\delta\in E$ such that $E=F[\delta]$ and $\delta^2\in F$ and set $\kappa=\delta^2$. 

		Recall that if $E$ imbeds in the central simple $F$-algebra $\M_m(D)$ of $m\times m$ matrices with entries in $D$ then, by the Skolem-Noether Theorem, such an imbedding is unique up to an inner automorphism by an element of $\GL_m(D)$. Furthermore, in the local set-up $E$ imbeds in $\M_m(D)$ if and only if $dm$ is even. In the global set-up, if $E$ imbeds in $\M_m(D)$ then $E_v$ imbeds in $\M_m(D_v)$ for any place $v$ of $F$ and in particular $dm$ is even.  Note further that $E$ naturally imbeds in $\M_2(F)$ as the centralizer of $\sm{0}{\kappa}{1}{0}$
and therefore always imbeds in $\M_m(D)$ if $m$ is even.

Set $D_E=D\otimes_F E$. It is a central simple $E$-algebra and it is a division algebra if and only if $E$ does not imbed in $D$. Consequently, in the twisted linear and Galois cases, our realization of the symmetric space depends on whether or not $E$ imbeds in $D$ (henceforth-case 1 and case 2 respectively). We set up some further notation dependent on the two cases whether or not such an imbedding exists.
	
{\bf Case 1:} Assume that $E$ imbeds in $D$. Fix once and for all such an imbedding and consider $E$ as a subalgebra of $D$. Let $C=C_D(E)=C_D(\delta)$ be the centralizer of $E$ (or equivalently, of $\delta$) in $D$. It is a central division $E$-algebra of degree $\frac{d}2$. 

By the Skolem-Noether theorem, the $E/F$-Galois action is realized by restriction to $E$ of an inner involution of $D$. That is, there exists $\ve\in D^\times$ such that $\Ad(\ve)(x):=\ve x \ve^{-1}$ is the Galois conjugate of $x$ for every $x\in E$. Note that this implies that $\ve^2\in C^\times$.
Fix such an $\ve$ once and for all.

Note that although $\ve$ is not in $C$, the automorphism $\Ad(\ve)$ of $D$ preserves $C$.
The algebra $D_E$ naturally identifies with $\M_2(C)$ and the Galois action on $D_E$ is realized in $\M_2(C)$ by $\Ad(\sm{0}{\ve}{\ve^{-1}}{0})$ (see \cite[Lemma 3.1]{MR3719517}). We emphasize that this is not an inner involution on $\M_2(C)$.  

{\bf Case 2:} Assume that $E$ does not imbed in $D$. Then $D_E$ is a central division $E$-algebra of degree $d$. 
It identifies with the centralizer of $\sm{0}{\kappa}{1}{0}$ in $\M_2(D)$.

We also consider the case $E=F\times F$. When considering the global Galois case, it will be used to describe the set up at places of $F$ that split in $E$. When considering linear periods it will allow an analogy with twisted linear periods. 

When $E=F\times F$, $E$ imbeds in $\M_2(F)$ as the centralizer of $\upsilon^\circ=\diag(1,-1)$ and $D_E=D\times D$  is the centralizer of $\upsilon^\circ$ in $\M_2(D)$. We refer to the involution $(x,y)\mapsto (y,x)$ on $D\times D$ as the $E/F$-Galois involution. Note that in this case $E$ does not imbed in $D$. Henceforth we consider this a part of Case 2.

In order to unify notation for all cases let $E$ be a degree two \'Etale $F$-algebra. 
When $E/F$ is a field extension write $\Res_{E/F}$ for the Weil restriction of scalars from $E$ to $F$. and let $Q_E$ be the base change from $F$ to $E$ of an algebraic group $Q$ defined over $F$. When $E=F\times F$ and $Q$ is an algebraic group defined over $F$ set $\Res_{E/F}(Q_E)=Q\times Q$.

\subsubsection{Explicit families of symmetric spaces for inner forms of $\GL$}\label{sss the families}
For the data $F,E, D$ where $E$ is a degree two \'Etale $F$-algebra and $D$ a central division $F$-algebra of degree $d$ we attach six families of triples 
\[
(G_m,H_m,\theta_m)=(G_m,H_m,\theta_m)_\x
\] for $m\in \BN$ and $\x\in \{(\mathbf{Lin}), (\mathbf{TL1}),((\mathbf{TL2})),(\mathbf{Gal1}),(\mathbf{Gal2}),(\mathbf{Grp})\}$. The data $F,E,D$ is suppressed from the notation. 

In all cases $G_m$ is a reductive algebraic group defined over $F$ with an involution $\theta_m$ such that $H_m=G_m^{\theta_m}$ is its group of fixed points. 
We refer to \cite{MR3430877}, \cite{MR3719517}, and \cite{MR3958071} for more details about the set up. 

For $a,k\in \BN$ and an $a\times a$ matrix $g$  let $[g]_k$ be the $ak\times ak$ matrix 
\[
[g]_k=\diag(g,\dots,g).
\]

\begin{description}
    \item[\itemname{lin}{Lin}]
{\bf Linear periods.} Let $G_m=G_D(2m)$ and 
\[
\theta_m=\Ad(\upsilon) \ \ \ \text{where}\ \ \  \upsilon=[\upsilon^\circ]_m\ \ \ \text{and}\ \ \ \upsilon^\circ=\diag(1,-1).
\]
Note that 
\[	
H_m=s_m\diag(G_D(m),G_D(m))s_m^{-1}
\]  
where $s_m$ is the permutation matrix corresponding to the permutation on $[1,2m]$ sending $k\in [1,m]$ to $2k-1$ and $m+j\in[m+1,2m]$ to $2j$.

In this case set $E=F\times F$ and recall that $D_E$ identifies with the centralizer of $\upsilon^\circ$ in $\M_2(D)$. Consequently,
\[
H_m(F)=\GL_m(D_E)\subseteq GL_{2m}(D)=G_m(F).
\]
    \item[\itemname{twlin1}{TL1}]
    {\bf Twisted linear periods-case 1.} Let $G_m=G_D(m)$ and 
\[
\theta_m=\Ad(\upsilon) \ \ \ \text{where}\ \ \  \upsilon=[\upsilon^\circ]_m\ \ \ \text{and}\ \ \ \upsilon^\circ=\delta.
\]
Note that
\[	
H_m(F)=\GL_m(C)\subseteq \GL_m(D)=G_m(F),
\]  
that is, $H_m=R_{E/F}(G_{C}(m))$. 
    \item[\itemname{gal2}{Gal2}]
  {\bf Galois periods-case 2.} Let $G_m=\Res_{E/F}(G_D(m)_E)$ and let $\theta_m$ be the Galois involution on $G_m$ so that $H_m=G_D(m)$. Note that $G_D(m)_E=G_{D_E}(m)$ and
\[
H_m(F)=\GL_m(D)\subseteq \GL_m(D_E)=G_m(F).
\]

    \item[\itemname{twlin2}{TL2}]
{\bf Twisted linear periods-case 2.} Let $G_m=G_D(2m)$ and 
\[
\theta_m=\Ad(\upsilon) \ \ \ \text{where}\ \ \  \upsilon=[\upsilon^\circ]_m\ \ \ \text{and}\ \ \ \upsilon^\circ=\begin{pmatrix}0 & \kappa \\ 1 & 0\end{pmatrix}.
\]
Recall that $D_E$ is identified with the centralizer of $\upsilon^\circ$ in $\M_2(D)$ and consequently
\[	
H_m(F)=\GL_m(D_E)\subseteq \GL_{2m}(D)=G_m(F).
\]

\item[\itemname{gal1}{Gal1}]
 {\bf Galois periods-case 1.} Recall that in this case $D_E$ is identified with $\M_2(C)$. 
This gives rise to the identification 
\[
G_D(m)_E\simeq G_C(2m).
\]
Let $G_m=\Res_{E/F}(G_C(2m))\simeq \Res_{E/F}(G_D(m)_E)$
and
\[
\theta_m=\Ad(\upsilon) \ \ \ \text{where}\ \ \  \upsilon=[\upsilon^\circ]_m\ \ \ \text{and}\ \ \ \upsilon^\circ=\begin{pmatrix}0 & \ve \\ \ve^{-1} & 0\end{pmatrix}.
\]
Since $D$ identified with the centralizer of $\upsilon^\circ$ in $\M_2(C)$, $\theta_m$ realizes on $G_m$ the Galois involution on $\Res_{E/F}(G_D(m)_E)$. With this identification $H_m=G_D(m)$. Thus,
\[
H_m(F)=\GL_m(D)\subseteq \GL_{2m}(C)=G_m(F).
\]

\item[\itemname{gp}{Grp}]
 {\bf The group case.} Let $G_m=G_D(m)\times G_D(m)$ and $\theta_m(x,y)=(y,x)$ so that $H_m$ is the diagonal imbedding of $G_D(m)$ in $G_m$. Note that for $E=F\times F$, $D$ imbeds diagonally in $D_E=D\times D$ and identifying $D$ with its image in $D_E$ in this way we have
\[
H_m(F)=\GL_m(D)\subseteq \GL_m(D_E)=G_m(F).
\]

\end{description}

\begin{remark}\label{rem arithmetic vs geometric}
While, from an arithmetic point of view the Galois case consists of cases \namelink{gal1} and \namelink{gal2} and the cases of twisted linear periods consist of cases \namelink{twlin1} and \namelink{twlin2}, from a geometric point of view, the structure of parabolic orbits on $G_m/H_m$ in case \namelink{twlin1} is closer to \namelink{gal2} and in case \namelink{twlin2} is closer to \namelink{gal1}. Note further that in the linear and twisted linear cases, $\theta_m$ is an inner involution.
\end{remark}
For the sake of uniform notation set
\begin{equation}\label{eq a&D}
a=\begin{cases} m & \text{in cases  \namelink{twlin1},\namelink{gal2}, \namelink{gp}} \\
2m & \text{in cases \namelink{lin}, \namelink{twlin2}, \namelink{gal1}}
\end{cases} \ \ \ \text{and}\ \ \ \D=\begin{cases} D & \text{in cases \namelink{lin}, \namelink{twlin1}, \namelink{twlin2}}\\ D_E & \text{in cases \namelink{gal2}, \namelink{gp}} \\ C & \text{in case \namelink{gal1}} \end{cases}
\end{equation}
so that in all cases $G_m(F)=\GL_a(\D)$. Also set
\begin{equation}\label{eq f'}
F'=\begin{cases} E&  \text{in cases  \namelink{gal1}, \namelink{gal2}, \namelink{gp}} \\
F & \text{in cases \namelink{lin},\namelink{twlin1}, \namelink{twlin2}}.
\end{cases} 
\end{equation}
Note that $\D$ is a central division $F'$-algebra except in the group case \namelink{gp} where $\D=D\times D$. 
In the $p$-adic setting let $\O_\D$ be the ring of integers of the $F'$-division algebra $\D$ except in case \namelink{gp} where we set $\O_\D=\O_D\times \O_D$.

\subsubsection{Local triples associated with a global triple}\label{sss local-global}\label{sss loc-glob triples}
There is some mixing of the different cases $\x\in \{\namelink{lin},\namelink{twlin1},\namelink{gal2},\namelink{twlin2},\namelink{gal1},\namelink{gp}\}$ when looking at local triples associated with a global triple as defined above. We now explain this relation.

Let $E/F$ be a quadratic extension of number fields and $D$ a central division $F$-algebra of degree $d$. Fix $\x\in \{\namelink{lin},\namelink{twlin1},\namelink{gal2},\namelink{twlin2},\namelink{gal1}\}$ and set 
\[
(G,H,\theta)=
(G_m,H_m,\theta_m)_\x. 
\]
Fix a place $v$ of $F$ and let $D_v=D\otimes_F F_v$. Then $D_v$ is a central simple $F_v$-algebra and there exists a divisor $k=k_v$ of $d$ and a central division $F_v$-algebra $R_v$ such that $D_v\simeq \M_k(R_v)$. We fix once and for all an identification $D_v = \M_k(R_v)$.

Let $G_v=G_{F_v}$ and let $\theta_v$ be the involution on $G_v$ induced from $\theta$ so that $H_v:=H_{F_v}=G_v^{\theta_v}$. Next, we explicate how $(G_v,H_v,\theta_v)$ is \emph{essentially} a triple of the form 
\[
(G_n,H_n,\theta_n)_{\x_v}
\]
associated with the data $F_v, E_v, R_v$ for some $n=n_v\in \N$ such that $m\mid n$ and a prescribed case $\x_v\in \{\namelink{lin},\namelink{twlin1},\namelink{twlin2},\namelink{gal1},\namelink{gal2},\namelink{gp}\}$. 
More precisely, there exists $g_v\in G(F_v)$ such that 
\begin{equation}\label{eq shifted triple}
(G_v,\Ad(g_v)(H_v), \Ad(g_v\theta_v(g_v)^{-1})\circ \theta_v)=(G_n,H_n,\theta_n)_{\x_v}
\end{equation} 
where
\[
\x_v=\begin{cases}
\namelink{lin} & \x=\namelink{lin} \text{ or } (\x\in \{(\namelink{twlin1},\namelink{twlin2}\} \text{ and }v\text{ splits in }E) \\
\namelink{twlin1} &\x\in \{(\namelink{twlin1},\namelink{twlin2}\}, v \text{ is inert in }E\text{ and } E_v\text{ imbeds in }R_v \\
\namelink{twlin2} &\x\in \{(\namelink{twlin1},\namelink{twlin2}\}, v \text{ is inert in }E\text{ and } E_v\text{ does not imbed in }R_v \\
\namelink{gp} & \x\in \{(\namelink{gal1},\namelink{gal2}\} \text{ and }v\text{ splits in }E \\
\namelink{gal1} &\x\in \{(\namelink{gal1},\namelink{gal2}\}, v \text{ is inert in }E\text{ and } E_v\text{ imbeds in }R_v \\
\namelink{gal2} &\x\in \{(\namelink{gal1},\namelink{gal2}\}, v \text{ is inert in }E\text{ and } E_v\text{ does not imbed in }R_v \end{cases}
\]
and $n=km$ unless either $\x=\namelink{twlin1}$ and $\x_v\in \{\namelink{lin},\namelink{twlin2}\}$ in which case $k$ is even and $n=mk/2$ or $\x=\namelink{twlin2}$ and $\x_v=\namelink{twlin1}$ in which case $n=2km$.

Let $a,\D$ and $F'$ be defined as in \eqref{eq a&D} and \eqref{eq f'} after Remark \ref{rem arithmetic vs geometric} with respect to the triple $(G_n,H_n,\theta_n)_{\x_v}$. That is, $G_n(F_v)=\GL_a(\D)$ where $\D$ is a division $F'$- algebra except in case \namelink{gp} where $\D=R_v\times R_v$.

There exists a finite set of places $T$ of $F$ containing the archimedean places such that we can choose $g_v\in \GL_a(\O_\D)$ for all $v\not\in T$.

\subsubsection{Maximal compact subgroup}
Let $(G,H,\theta)=(G_m,H_m,\theta_m)_\x$ for 
\[
\x\in \{\namelink{lin},\namelink{twlin1},\namelink{gal2},\namelink{twlin2},\namelink{gal1},\namelink{gp}\}
\] 
(see Section \ref{sss the families}).
We define a maximal compact subgroup $K$ of $G(F)$ in the local set-up and of $G(\A)$ in the global set-up as follows.

When $F$ is archimedean there exists a Cartan involution of $G$ commuting with $\theta$. We fix such an involution and let $K$ be its fixed point subgroup. Then $K$ is a $\theta$-stable maximal compact subgroup of $G(F)$. Explicitly, writing $G(F)=\GL_a(\D)$ where $\D\in \{\R,\C,\mathbb{H}\}$ we take the Cartan involution to be $g\mapsto {}^t \bar g^{-1}$ where the bar is induced from the standard conjugation on $\mathbb{H}$ restricted to $\D$ except in  case \namelink{gp} where $G(F)=\GL_m(D)\times \GL_m(D)$ with $D\in \{\R,\C,\mathbb{H}\}$ and we take the above involution in each coordinate.  

When $F$ is $p$-adic, a $\theta$-stable maximal compact subgroup does not always exist. Instead, set $K=\GL_a(\O_\D)$ where $a,\D$ are defined by \eqref{eq a&D} after Remark \ref{rem arithmetic vs geometric}. 

In the local set-up we note now that with these choices $K\cap H(F)$ is a maximal compact subgroup of $H(F)$.

When $F$ is a number field set $K=\prod_v K_v$ to be the product over all places $v$ of $F$ of a maximal compact subgroup $K_v$ of $G(F_v)$ chosen as follows. 
We follow the discussion in Section \ref{sss loc-glob triples} as well as its notation.  Applying the identification \eqref{eq shifted triple} let $K_v'$ be the maximal compact associated above to the local triple $(G_n,H_n,\theta_n)_{\x_v}$ and set $K_v=g_v^{-1}K_v' g_v$. 

Note that $K\cap H(\A)$ is a maximal compact subgroup of $H(\A)$. Furthermore, for all but finitely many $v$, in light of the last part of Section \ref{sss loc-glob triples} and in its notation we have that $K_v=\GL_a(\O_\D)$.

\subsubsection{Parabolic subgroups}
Let $(G,H,\theta)=(G_m,H_m,\theta_m)_\x$ for 
\[
\x\in \{\namelink{lin}, \namelink{twlin1},  \namelink{twlin2}, \namelink{gal1},\namelink{gal2},,\namelink{gp}\}
\] 
(see Section \ref{sss the families}).
Set \begin{align}\label{formula::defn--w-star}
		w_{\star}  = \begin{cases}
			e \quad &\x\in \{ \namelink{lin}, \namelink{twlin1}, \namelink{gal2}, \namelink{gp}\}\\
						[\left( \begin{smallmatrix}
				0& 1 \\
				1 & 0
			\end{smallmatrix} \right)]_m& \x\in \{\namelink{twlin2},\namelink{gal1}\}.
		\end{cases}
	\end{align}

For a subgroup $Q$ of $G$ let $Q_H=Q^\theta=Q\cap H$.
Let $A_0$ be the maximal $F$-split torus consisting of diagonal matrices in $G$ with diagonal entries in $F^{\times}$ (in the group case diagonal matrices means diagonal in both coordinates). Fix the minimal parabolic subgroup $P_0 = M_0 \ltimes U_0$ of $G$ containing $A_0$ with Levi subgroup $M_0$, the subgroup of diagonal matrices, and unipotent radical $U_0$, the subgroup of unipotent upper-triangular matrices in $G$. Note that in all cases $A_0$ is $\theta$-stable, $A_0^{\theta}$ is a maximal split torus in $H$ and $\theta (P_0) = w_{\star} P_0 w_{\star}^{-1}$. Furthermore, $P_0^\theta=M_0^\theta U_0^\theta$ is a minimal parabolic subgroup of $H$ and the map $P \mapsto P_H$ is a bijection between the set of $\theta$-stable standard (with respect to $P_0$) parabolic subgroups of $G$ and standard (with respect to $P_0^\theta$) parabolic subgroups of $H$. 
Note further that the automorphism $\theta'$ defined in Section \ref{ss sym} stabilizes any standard parabolic subgroup of $G$ and acts trivially on $\fra_0^*$ and on $W$.

Let $a$ and $\D$ be defined by \eqref{eq a&D} after Remark \eqref{rem arithmetic vs geometric}. For a composition $\alpha = (m_1,\cdots,m_t)$ of $a$, let $P_{\alpha}=M_{\alpha}U_{\alpha}$ be the standard parabolic subgroup of $G$ consisting of block upper triangular matrices with unipotent radical $U_{\alpha}$ and so that 
\[
M_{\alpha}(F) =\{\diag(g_1,\dots,g_t):g_i\in \GL_{m_i}(\D),\ i\in [1,t]\}.
\]
We also say that $P_{\alpha}$ is the parabolic subgroup of $G$ of type $\alpha$.

\subsubsection{Explication of parabolic induction}
Let $\D$ be a central division $F$-algebra and ${\mathbf{G}}=G_\D(a)$.
Set $\nu = \lvert \cdot \rvert \circ \nrd$ where $\nrd$ is the reduced norm on $G$. For a representation $\pi$ of $G$ and $s \in \C$, set $\pi[s] = \nu^s \otimes \pi$ to be the twist of $\pi$ by the charcter $\nu^s$. 
	
For a composition $\alpha=(a_1,\cdots,a_r)$ of $a$, let ${\mathbf{P}}={\mathbf{MU}}$ be the standard parabolic subgroup of type $\alpha$ with its standard Levi decomposition so that ${\mathbf{M}}={\mathbf{M_1} \times \cdots\times {\mathbf{M_r}}}$ with ${\mathbf{M_i}}=G_D(a_i)$, $i\in [1,r]$. Let $\sigma_i$ be a representation of $M_i$, automorphic if $F$ is global and set $\sigma = \sigma_1 \otimes \cdots \otimes \sigma_r$ for the corresponding (automorphic when $F$ is global) representation of $M$. We set 
	\begin{align*}
		\sigma_1 \times \cdots \times \sigma_r = I_P^G(\sigma).
	\end{align*}

\subsubsection{An auxiliary involution}\label{sss auxi} Note that in the cases \namelink{lin}, \namelink{twlin2} and \namelink{gal1} the involution $\theta_m$ is defined on $\GL_{2m}(\D)$ where $\D$ is defined by \eqref{eq a&D} after Remark \ref{rem arithmetic vs geometric}. In all cases, we define a related automorphism $\iota=\iota_k$ on $\GL_k(\D)$ for every $k\in \N$ as follows:
\[
\iota(g)=\begin{cases}
\gamma_kg\gamma_k^{-1} & \text{in case }\namelink{lin} \\
 \theta_k(g)& \text{in cases } \namelink{twlin1}\text{ and }\namelink{gal2}\\
g & \text{in case }\namelink{twlin2} \\
	\epsilon g\epsilon^{-1} & \text{in case }\namelink{gal1}.
\end{cases}
\]
Here $\gamma_k=\diag(1,-1,\dots,(-1)^{k-1})$.
For a representation $\pi$ of $\GL_k(\D)$ set 
\[
\pi^*=(\pi^\iota)^\vee.
\] 
Note that $\iota$ is an involution except in case $\namelink{gal1}$ where $\iota^2$ is an inner automorphism. Consequently, in all cases, for an irreducible representation $\pi$ of $\GL_k(\D)$ we have $(\pi^*)^*\simeq \pi$.

Given $m\in \N$ let $a$ be defined by \eqref{eq a&D} after Remark \ref{rem arithmetic vs geometric}. Then,  
in cases \namelink{lin}, \namelink{twlin1}, \namelink{gal2} we have $\iota_a=\theta_m$. Recall that $\theta_m$ is an inner automorphism in case \namelink{twlin2}. In case \namelink{gal1} we have that $\iota_a\circ \theta_m=\theta_m\circ\iota_a$ is an inner automorphism of $\GL_a(\D)$. Consequently, in cases \namelink{lin}, \namelink{twlin1}, \namelink{gal2}, \namelink{twlin2} and \namelink{gal1} we have 
\[
\pi^{\iota_a}\simeq \pi^{\theta_m}.
\]

\subsection{The Jacquet-Langlands correspondence}
	
For the local and global Jacquet-Langlands correspondence $\jl$, we refer the reader to the exposition in \cite[Section 4]{MR4308058} of the main reference \cite{MR2329758} completed by \cite{MR2684298}. 

Assume first that $F$ is local. Let $\D$ be a degree $d$ central simple $F$-algebra. Denote by $\nu$ the character of $\GL_m(\D)$ for any $m\in \BN$ defined as the absolute value of the reduced norm. We say that a representation of $\GL_m(\D)$ is \emph{essentially} $\P$ for some property $\P$, if $\nu^\alpha\pi$ satisfies the property $\P$ for some $\alpha\in \BR$. 

For an irreducible and essentially square integrable representation $\delta$ of $\GL_m(\D)$ there is a unique real number that we denote by $r(\delta)$ such that $\nu^{-r(\delta)}\delta$ is unitary. We also use the terminology \emph{discrete series} for essentially square-integrable. 

We denote by $\Pi(m,\D)$ the set of irreducible representations of $\GL_m(\D)$ and let
\[\Irr_\D=\sqcup_{m\in\BN} \Pi(m,\D).\] We further denote by $\Pi_{\esi}(m,\D)$ the subset of essentially square-integrable, by $\Pi_\si(m,\D)$ the subset of square integrable, and by $\Pi_c(m,\D)$ the subset of supercuspidal representations in $\Pi(m,\D)$. Then we set 
\[
\CC_\D=\sqcup_{m\in\BN} \Pi_c(m,\D),\ \ \ \SI_\D=\sqcup_{m\in\BN} \Pi_\si(m,\D)\ \ \ \text{and} \ \ \ \ESI_\D=\sqcup_{m\in\BN} \Pi_\esi(m,\D).
\]

The following results follow from \cite{MR771672,MR1040995,MR2329758,MR2684298}. The Jacquet-Langlands transfer is a bijection $\JL$ sending $\ESI_\D$ to $\ESI_F$ (and $\SI_D$ to $\SI_F$). In \cite{MR2329758}, Badulescu extends $\JL$ to a map from the Grothendieck group of finite length representations of $\GL_m(\D)$ to that of finite length representations of $\GL_{md}(F)$ for $m\in \N$. In particular $\JL$ sends an irreducible representation of $\GL_m(\D)$ to an element of the Grothendieck group, in fact, to a finite length representation of $\GL_{md}(F)$, which in general does not need to be irreducible (see \cite[Remarque 3.2]{MR2329758}). 

\begin{definition}
We say that $\pi\in \Irr_\D$ is generic if it is of the form $\delta_1\times \dots \times \delta_k$ with $\delta_i$ essentially square-integrable $i\in[1,k]$, and if moreover $\jl(\delta_1)\times \dots \times \jl(\delta_k)$ is irreducible and generic in the usual sense, that is, admits a Whittaker model. 
\end{definition}

Denote by $\Pi_{\gen}(m,\D)$ the subset of generic representations in $\Pi(m,\D)$, and set 
\[
\GG_\D=\sqcup_{m\in\BN} \Pi_{\gen}(m,\D).
\]

Then, as recalled in \cite[Section 2.3]{Sign}, the Jacquet-Langlands transfer between Grothendieck groups defined in \cite{MR2329758}, restricts to a map $\JL:\GG_\D\rightarrow \GG_F$ mapping $\Pi_{\gen}(m,\D)$ to $\Pi_{\gen}(dm,F)$ for every $m\in \BN$. It is simply described as follows. A representation $\pi\in \GG_\D$ is of the form $\pi=\delta_1\times \cdots \times \delta_k$ for a unique multiset $\{\delta_1,\dots,\delta_k\}$
of essentially square integrable representations in $\Irr_\D$. We then have
\begin{equation}\label{eq JLmult}
\JL(\pi)=\JL(\delta_1) \times \cdots\times \jl(\delta_k).
\end{equation}

 The generic Jacquet-Langlands transfer further satisfies the following properties.
\begin{itemize}
\item If $\D=F$ then $\JL:\GG_F\rightarrow \GG_F$ is the identity.
\item $\JL(\nu^s \pi)=\nu^s \JL(\pi)$, $\pi\in \GG_\D$.
\item $\JL(\pi^\vee)=\JL(\pi)^\vee$.
\end{itemize}

For an essentially square integrable representation $\delta$ of $\GL_m(\D)$ there is a minimal positive real number $\alpha_\delta$ such that $\delta\times \nu^{\alpha_\delta} \delta$ reduces. In fact, $\alpha_\delta\in \BN$ divides $d$. Set 
\[
\nu_\delta=\nu^{\alpha_\delta}.
\]
We recall that in the archimedean case square integrable representations of $\GL_n(\C)$ and of $\GL_n(\BH)$ exist only when $n=1$, and those of $\GL_n(\BR)$ only when $n=1,2$. 
Furthermore, $\nu_\delta=\nu$ for any $\delta\in \SI_\D$ in the archimedean case.


\begin{definition}\label{df exp 12} We denote by $\Pi_\D(-\frac12,\frac12)$ the class of representations of the form $\pi=\delta_1\times \cdots \times \delta_k$ where $\delta_i$ is irreducible, essentially square integrable and such that $\abs{r(\delta_i)}<\frac12$ for $i=1,\dots,k$. 
\end{definition} 

It is well-known that every irreducible, generic and unitary representation of $\GL_m(\D)$ is in $\Pi_\D(-\frac12,\frac12)$ and that any representation in $\Pi_\D(-\frac12,\frac12)$ is irreducible and generic. 

In the non-archimedean case we can be more explicit about the restriction of $\JL$ to $\SI_\D$. 

Recall from \cite{MR584084} that for $\rho\in \CC_F$ and $k\in \BN$ the representation 
\[
\nu^{\frac{1-k}2}\rho\times \nu^{\frac{3-k}2}\rho\times \cdots\times \nu^{\frac{k-1}2}\rho
\] 
admits a unique irreducible quotient that we denote by $\St_k(\rho)$. It is essentially square integrable and any essentially square integrable representation in $\Irr_F$ is obtained in this way for a unique $(\rho,k)$. Furthermore, $\St_k(\rho)\in \SI_F$ if and only if $\rho$ has a unitary central character. 

Let $\rho\in \CC_\D$. Since $\JL(\rho)$ is essentially square integrable, there exists a unique $k\in \BN$ and $\rho'\in \CC_F$ such that $\JL(\rho)=\St_k(\rho')$. In fact, it is known that $k=\alpha_\rho$.
For $a,b \in \BR$ with $t=b+1- a \in \BZ_{\ge 0}$, set
	\begin{align*}
		\Delta(\rho,a,b) = \{ \nu_{\rho}^a \rho, \nu_{\rho}^{a+1}\rho, \cdots, \nu_{\rho}^b \rho \}.
	\end{align*}
	Such a set is called a (cuspidal) segment. The representation 
	\begin{align*}
		\nu_{\rho}^a \rho \times \nu_{\rho}^{a+1}\rho \times \cdots \times \nu_{\rho}^b \rho
	\end{align*}
	has a unique irreducible quotient that we denote by $L(\Delta(\rho,a,b))$. This is an essentially square-integrable representation of 
	$\GL_{tm}(\D)$. Set
 \[
 \steinberg_n(\rho)=L(\Delta(\rho, \frac{1-n}2,\frac{n-1}2)).
 \]
Every essentially square-integrable representation of $\GL_k(\D)$ for some $k\in \N$ is of the form $\St_n(\rho)$ for a unique pair $(\rho,n)$ as above. We have $\St_n(\rho)\in \SI_D$ if and only if $\rho$ has a unitary central character.  Alternatively, it is also of the form 
$L(\Delta(\rho,a,b))$ as above for a unique triple $(\rho,a,b)$ with $\rho$ unitary.
Furthermore, 
\[
\jl(\St_n(\rho))=\St_{kn}(\rho') \ \ \ \text{and}\ \ \ \nu_\delta=\nu_\rho \ \ \ \text{for}\ \ \ \delta=\St_n(\rho).
\]

Assume now that $F$ is a number field. Let $\D$ be a central $F$-division algebra of degree $d$, $\si(\GL_a(\D_\A))$ the set of isomorphism classes of irreducible, square-integrable automorphic representations of $\GL_a(\D_\A)$ and $\CC(\GL_a(\D_\A))$ be the subset of classes of cuspidal representations.
	The Jacquet langlands correspondence is a transfer  
	\[\jl:\si(\GL_a(\D_\A))\rightarrow \si(\GL_{ad}(\BA)).\]  
	Set 
	\[\CC^*(\GL_a(\D_\A))=\jl^{-1}(\CC(\GL_{ad}(\BA))).\]
	It is a subset of $\CC(\GL_a(\D_\A))$. 
	For $\pi\in \mathcal{C}^*(\GL_a(\D_\A))$ and for each place $v$ of $F$, the local component $\pi_v$ is unitary and generic.
	We have
	\begin{equation}\label{eq JLglobaldecomp}
	\jl(\pi)_v=\jl(\pi_v)=\JL(\delta_1) \times \cdots\times \jl(\delta_k)
	\end{equation}
	if we write $\pi_v=\delta_1\times \cdots\times \delta_k$ as in \eqref{eq JLmult}.
	
\subsection{On normalized intertwining operators}\label{ss normint}	
Let $F$ be a local field and $\D$ a central Division $F$-algebra.
Let $G=\GL_a(\D)$ and $P=MU=P_{(a_1,\dots,a_k)}$ be the standard parabolic subgroup of $G$ of type $(a_1,\dots,a_k)$ a composition of $a$. The set $W(M)$ naturally identifies with $\mathfrak{S}_k$ viewed as the group of permutations of the blocks of $M$. For $w\in W(M)$ write $\inv(w)=\{(i,j): 1\le i<j\le k,\ w(i)>w(j)\}$. For irreducible representations $\pi_i$ of $\GL_{a_i}(\D)$, $i=1,\dots,k$ set $\pi=\pi_1\otimes \cdots\otimes \pi_k$ for the corresponding irreducible representation of $M$. For $\lambda=(\lambda_1,\dots,\lambda_k)\in \C^k\simeq \fra_{M,\C}^*$ let
\[
r(w,\pi,\lambda)=\prod_{(i,j)\in\inv(w)}\frac{L(\lambda_i-\lambda_j,\JL(\pi_i),\JL(\pi_j)^\vee)}{\epsilon(\lambda_i-\lambda_j,\JL(\pi_i),\JL(\pi_i),\JL(\pi_j)^\vee,\psi)L(1+\lambda_i-\lambda_j,\JL(\pi_i),\JL(\pi_j)^\vee)}
\]
as in \cite[Chapter 2, (2.1), (2.2) and (2.3)]{MR1007299} and consider the normalized intertwining operators
\[
N(w,\pi,\lambda)=r(w,\pi,\lambda)^{-1}M(w,\pi,\lambda).
\]
It follows from \cite[Chapter 2, Lemma 2.1]{MR1007299} that these intertwining operators satisfy the properties (R1)-(R8) of \cite[Theorem 2.1]{MR999488}. Consequently, the results of \cite[I.1 and I.2]{MR1026752} stated there for the case $\D=F$ are in fact valid, with the same proofs, for our more general context of inner forms of general linear groups. 
As a consequence we have the following results that will be useful for us.

\begin{lemma}\label{lem hol int op}
For $\pi_1,\pi_2\in \Pi_\D(-\frac12,\frac12)$ let $a_1, a_2\in \N$ be such that $\pi_i$ is a representation of $\GL_{a_i}(\D)$ and $M=M_{(a_1,a_2)}$.  Then $N(w_M,\pi_1\otimes \pi_2,(s,-s))$ is holomorphic at $s=0$ and $N(w_M,\pi_1\otimes \pi_2,0)$ is an isomorphism. 
\end{lemma}
\begin{proof}
If $\pi_1$ and $\pi_2$ are essentially square integrable this follows from the irreducibility of $\pi_1\times \pi_2$ and the lemma in \cite[I.2]{MR1026752} (that is valid in our context as pointed out above). For the general case, write $\pi_1=\delta_1\times \cdots\times \delta_k$ and $\pi_2=\delta_{k+1}\times \cdots\times \delta_{k+l}$ where $\delta_i$ is essentially square integrable for $i=1,\dots,k+l$ and let $L$ be the standard Levi subgroup of $M$ such that $\delta=\delta_1\otimes \cdots\otimes \delta_{k+l}$ is a representation of $L$. Then we can decompose $w=w_t\cdots w_1$ as a product of $t=kl$ elementary symmetries so that $w_1\in W(L)$ and $w_i\in W(vLv^{-1})$ where $v=w_{i-1}\cdots w_1$ for $i>1$ and by the property (R1) of \cite[Theorem 2.1]{MR999488} we have
\[
N(w_M,\pi_1\otimes\pi_2,(s,-s))=N(w_t,w_{t-1}\cdots w_1\delta,w_{t-1}\cdots w_1\underline{s})\circ\cdots\circ N(w_1,\delta,\underline{s})
\]
where $\underline{s}=(\overbrace{\mathop{s,\dots,s}}\limits^k,\overbrace{\mathop{-s,\dots,-s}}\limits^l)$. Each of the $kl$ factors on the right hand side is holomorphic at $s=0$ and its value at $s=0$ is an isomorphism by the special case of the lemma already proved. The same therefore holds for the left hand side and the lemma follows.
\end{proof}

\subsection{Distinction and compatibility}
Let ${\mathbf{G}}$ be a reductive algebraic group defined over $F$ and ${\mathbf{H}}$ a subgroup. Let $G$ be ${\mathbf{G}}(F)$ if $F$ is local and ${\mathbf{G}}(\A)$ if $F$ is a number field and let $H$ be defined similarly.

\begin{definition}
Let $F$ be a local field and $\pi$ be a representation of $G$. We say that $\pi$ is $H$-distinguished if the space $\Hom_H(\pi,\C)$ of $H$-invariant linear forms on $\pi$ is non-zero. 
\end{definition}

Consider the global set-up. Let $Z_G$ be the center of $G$. We denote by 
\[
\P_{H}:\varphi\to \int_{(Z_G\cap H) {\mathbf{H}}(F)\backslash H} \varphi(h)dh
\] 
the $H$-period integral on the space of cuspidal automorphic representations of $G$ with central character trivial on $Z_G\cap H$, where $dh$ is the unique, up to scaling, right invariant measure on the quotient. It is convergent by \cite[Proposition 1]{MR1233493}.
\begin{definition}
Let $\pi$ be a cuspidal automorphic representation of $G$.
\begin{itemize}
\item We say that $\pi$ is $H$-distinguished if its central character is trivial on $Z_G\cap H$, and moreover if 
the period integral $\P_{H}$ does not identically vanish on the space of $\pi$. 
\item If $\pi$ is isomorphic to a restricted tensor product $\otimes'_v \pi_v$ over all places $v$ of $F$ of representations $\pi_v$ of ${\mathbf{G}(F_v)}$  then we say that $\pi$ is locally $H$-distinguished if $\pi_v$ is ${\mathbf{H}}(F_v)$-distinguished for every place $v$ of $F$.  
\end{itemize}
\end{definition}

{\bf A convention:}
We follow the following convention throughout the paper.
If $({\mathbf{G}},{\mathbf{H}},\theta)$ is defined by one of the cases \namelink{lin}, \namelink{twlin1},  \namelink{twlin2}, \namelink{gal1},\namelink{gal2}, so that $H$ is clear from the context, for a representation $\pi$ of $G$ that is $H$-distinguished we simply say that $\pi$ is distinguished. Furthermore, by convention, in cases \namelink{lin}, \namelink{twlin2} and \namelink{gal1} for $k$ odd no representation of $\GL_k(\D)$ is distinguished where $\D$ is defined by \eqref{eq a&D} after Remark \ref{rem arithmetic vs geometric} (in those cases the involution $\theta$ is only defined on $\GL_k(\D)$ for $k$ even).

\begin{definition}\label{def compatible}
Let $F$ be a local field and $(G,H,\theta)=(G_m(F),H_m(F),\theta_m)$ be defined by one of the cases \namelink{lin}, \namelink{twlin1},  \namelink{twlin2}, \namelink{gal1},\namelink{gal2}. Let $\pi$ be an irreducible and generic representation of $G$ and write $\pi\simeq \delta_1\times \cdots \times \delta_k$ where $\delta_i$ is essentially square integrable. We say that $\pi$ is $H$-incompatible if there exists $i$ such that $\delta_i$ is not distinguished but $\JL(\delta_i)$ is a representation of $\GL_k(F)$ for some even $k\in \N$ that is $\GL_{k/2}(F)\times \GL_{k/2}(F)$-distinguished in cases \namelink{lin}, \namelink{twlin1} and \namelink{twlin2}, respectively a representation of $\GL_k(E)$ for some $k\in \N$ that is $\GL_k(F)$-distinguished in cases \namelink{gal1} and \namelink{gal2}. Otherwise, we say that $\pi$ is $H$-compatible.
\end{definition}

\begin{lemma}\label{lem compatible}
In the notation of the above definition in case \namelink{gal2} and if $d$ is odd also in case \namelink{lin} every irreducible and generic representation of $G$ is $H$-compatible. 
\end{lemma}
\begin{proof}
Let $\delta$ be an essentially square integrable representation of $\GL_t(D_E)$ in case \namelink{gal2} (respectively, of $\GL_t(D)$ in case \namelink{lin}) for some $t\in \N$ and note that $\JL(\delta)$ is a representation of $\GL_{td}(E)$ (respectively of $\GL_{td}(F)$). Consequently, it suffices to show that if $\JL(\delta)$ is $\GL_{td}(F)$-distinguished (respectively, $t$ is even and $\JL(\delta)$ is $\GL_{td/2}(F)\times \GL_{td/2}(F)$-distinguished) then $\delta$ is also distinguished. 

If $F$ is non-archimedean this follows from \cite[Theorem 1]{MR3858402} in case \namelink{gal2} and from the combination of \cite[Theorem 3.20]{Sign}  and \cite[Theorem 6.1]{MR3953435} in case \namelink{lin}.
In the archimedean case, since $\JL(\delta)$ is square integrable, in both cases we must have $\JL(\delta)=\delta$.
The lemma follows.
\end{proof}

\begin{remark}\label{rem compatibility of discrete series}
In the notation of Definition \ref{def compatible}, let $\D$ be given by \eqref{eq a&D} after Remark \ref{rem arithmetic vs geometric}, and $\delta$ be an irreducible essentially square integrable representation of $\GL_k(\D)$. The following sheds more light on the notion of compatibility. 
\begin{itemize}
\item In the Galois cases assume further that $\JL(\delta)$ is distinguished. Then $\delta$ is distinguished except if we are in case \namelink{gal1} and $k$ is odd. In other words $\pi$ as in Definition \ref{def compatible} is not compatible if and only if we are in case \namelink{gal1}, and one of the representations $\delta_i$ belongs to $\Pi_\esi(k_i,\D)$ for some odd $k_i$, but $\JL(\delta_i)$ is distinguished.
\item In the linear case assume further that $\JL(\delta)$ is distinguished (in particular, $kd$ is even). Then $\delta$ is distinguished except if $k$ is odd. In particular, if $d$ is odd then $\delta$ is distinguished. In other words $\pi$ as in Definition \ref{def compatible} is not compatible if and only if $d$ is even, and one of the representations $\delta_i$ belongs to $\Pi_\esi(k_i,\D)$ for some odd $k_i$, but $\JL(\delta_i)$ is distinguished.
\item The twisted linear cases do not afford such rigidity due to the $\epsilon$-dichotomy.
\end{itemize}
\end{remark}

\section{Local $L$-factors and global $L$-functions}

We begin this section by introducing the local $L$-factors that show up in this work and recalling relevant properties.

	\subsection{The local $L$-functions}\label{sec local L fct}
	
	Consider the local case and denote by $\WD_F$ the Weil-Deligne group. In the archimedean case it is simply the Weil group attached to $F$.
	
	As $F^\times$ is naturally a quotient of the Weil group (hence of $\WD_F$), for any character $\chi$ of $F^\times$ we still denote by $\chi$ the corresponding character of $\WD_F$ also attached by local class field theory.

Attached to an $F$-parameter $\phi$, that is, a finite dimensional semi-simple representation of $\WD_F$, there is an Artin $L$-factor $L(s,\phi)$ defined by Artin in the non-archimedean case. We refer to \cite[Appendix]{MR2533003} for the description of this $L$-factor in the archimedean case. 
In both cases, $L(s,\phi \oplus \phi')=L(s,\phi)L(s,\phi')$ and the description reduces to that for irreducible parameters.

Let $n\in \BN$ and let $\pi$ be an irreducible representation of $\GL_n(F)$. In the archimedean case the Langlands parameter $\phi_\pi$ is attached to $\pi$ in \cite{MR1011897}. In the non-archimedean case, the parameter $\phi_\pi$ is attached by the local Langlands correspondence obtained independently by \cite{MR1876802}, \cite{MR1738446} and \cite{MR3049932}. 
	
For any finite-dimensional representation $\mathbf{r}$ of $\GL_n(\C)$, the connected component of the Langlands dual group,  $\mathbf{r}(\phi_\pi)=\mathbf{r}\circ \phi_\pi$ is another $F$-parameter and this allows us to define the $L$-factor $L(s,\pi,\mathbf{r})=L(s,\mathbf{r}(\phi_\pi))$. In this work we apply this construction directly in the following three cases:
\begin{itemize}
\item the standard $L$-factor $L(s,\pi)=L(s,\pi,\mathrm{Std})$;
\item the exterior square $L$-factor $L(s,\pi,\wedge^2)$;
\item and the symmetric square $L$-factor $L(s,\pi, \Sym^2)$.
\end{itemize}
Here $\mathrm{Std}$ is the standard n-dimensional representation, $\wedge^2$ the exterior square $(n(n-1)/2)$-dimensional and $\Sym^2$ the symmetric square $(n(n+1)/2)$-dimensional representation of $\GL_n(\C)$.
We further consider other Artin $L$-factors in this work. For two irreducible representations $\pi_i$ of $\GL_{n_i}(F)$, $i=1,2$ the tensor product $\phi_{\pi_1}\otimes \phi_{\pi_2}$ is another $F$-parameter. Set
\begin{itemize}
\item the $L$-factor of pairs $L(s,\pi_1,\pi_2)=L(s,\phi_{\pi_1}\otimes \phi_{\pi_2})$.
\end{itemize}
Clearly, we have
\[
L(s,\pi_1,\pi_2)=L(s,\pi_2,\pi_1).
\]
It is also straightforward that the standard $L$-factor of an irreducible representation $\pi$ of $\GL_n(F)$ equals an $L$-factor for pairs, namely, we have
\[
L(s,\pi)=L(s,\pi,\triv_{F^\times})
\]
where $\triv_{F^\times}$ is the trivial character of $F^\times=\GL_1(F)$. Another simple observation is that
\begin{equation}\label{eq rs decomp}
L(s,\pi,\pi)=L (s,\pi, \wedge^2)L (s,\pi, \Sym^2).
\end{equation}

Let $E/F$ be a quadratic field extension. For any $E$-parameter $\phi$ let $\As^{+}(\phi)$ and $\As^{-}(\phi)$ be the $F$-parameters constructed respectively as the odd and even Asai lift of $\phi$ following \cite[p.26-27]{MR3202556}. As these authors point out, we have 
\[
\As^-(\phi)=\eta_{E/F}\otimes \As^{+}(\phi)
\] 
where $\eta_{E/F}$ is the quadratic character of $F^\times$ (considered as a character of $\WD_F$) attached to $E/F$ by local class field theory. 
For an irreducible representation $\pi$ of $\GL_n(E)$, set
\begin{itemize}
\item the Asai $L$-factors $L(s,\pi,\As^{\mathfrak{e}})=L(s,\As^{\mathfrak{e}}(\phi_\pi))$, $\mathfrak{e}\in\{+,-\}$.
\end{itemize}
We have
\[
L(s,\pi,\As^-)=L(s,\eta\otimes \pi,\As^+)
\]
for any extension $\eta$ of $\eta_{E/F}$ to a character of $E^\times$. In analogy with \eqref{eq rs decomp} we have
\begin{equation}\label{eq asai rs decomp}
L(s,\pi,\pi^\theta)=L (s,\pi, \As^+)L (s,\pi, \As^-)
\end{equation}
where $\theta$ is the $E/F$ Galois action.

There are two other standard methods to define local $L$-factors more directly without reference to the local Langlands correspondence. Shahidi $L$-factors are defined by the Langlands-Shahidi method
and Rankin-Selberg $L$-factors are defined as the ``greatest common divisors" of a certain family of zeta integrals.

In the $p$-adic case the Shahidi $L$-factors for pairs, exterior square, symmetric square and Asai $L$-factors are defined as a part of the general construction in \cite{MR1070599}. It is a consequence of \cite{MR1738446} for $L$-factors for pairs and of \cite{MR2595008} for exterior square, symmetric square and Asai $L$-factors that the Shahidi $L$-factors coincide with the corresponding Artin $L$-factors. We therefore do not use different notation for them and we use this fact throughout the paper to apply well-known properties of Shahidi $L$-factors to the corresponding Artin $L$-factors and vice-versa. 
Similarly, in the archimedean case Shahidi proved in \cite{MR816396} that the Artin $L$-factors are the correct factors that emerge in the Langlands-Shahidi method. 

In many cases that we consider, it is also known that the $L$-factors defined above coincide with the corresponding Rankin-Selberg $L$-factors. 
Again we will not give Rankin-Selberg $L$-factors special notation. Instead, we point out bellow when it is known that they coincide with the other type of $L$-factors and freely use this fact in the sequel, particularly, for square-integrable representations.

The Rankin-Selberg type $L$-factors for pairs are defined in \cite{MR701565} (see \cite{MR2533003} for the archimedean case). In the non-archimedean case Shahidi proved in \cite{MR729755} that they coincide with the Shahidi $L$-factors. In the archimedean case the results of \cite{MR2533003} and \cite{MR1159102} show that they coincide with Artin $L$-factors for pairs.

We only consider properties of Rankin-Selberg type Exterior squre, Symmetric square and Asai $L$-factors in the non-archimedean case. 

For exterior square $L$-factors, Jacquet and Shalika suggested in \cite{MR1044830} a family of Rankin-Selberg type integrals. It is proved in \cite{MR3008415} for square-integrable representations and later by \cite[Theorem 5.14]{MR4074055} for any irreducible representation that the Jacquet and Shalika type Rankin-Selberg $L$-factors for exterior square coincide with the corresponding Artin L-factors. Another family of Rankin-Selberg type integrals for the exterior square $L$-factor is considered in \cite{MR3430877} based on the global work of Bump and Friedberg. 

Rankin-Selberg type $L$-factors for the Symmetric square are defined in \cite{MR3582406} for any irreducible generic representation via integrals considered by Bump and Ginzburg. Yamana proves that for a squre-integrable irreducible representation, the Rankin-Selberg type symmetric square $L$-factor coincides with the corresponding Artin $L$-factor. See also \cite{MR3722261}.

When $E/F$ is a quadratic field extension, we refer the reader to \cite[Section 6.1]{MR4308058} for an introduction of Flicker's Rankin-Selberg theory of local Asai $L$-factors $L (s,\pi,\As^{\pm})$ associated to an irreducible generic representation $\pi$ of $\GL_n(E)$. In this case they coincide with Artin $L$-factors.

		\begin{cv}
	The local factor of the global Asai $L$-function at a place of $F$ that splits in $E$ is the $L$-factor for pairs. We therefore adopt the following standard convention. When $E=F\times F$ is the split 2-dimensional \'etale $F$-algebra and $\pi=\pi_1\otimes \pi_2$ where $\pi_i$ is an irreducible generic representation of 
	$\GL_n(F)$, $i=1,2$ we set 
	\[L (s,\pi,\As^{\mathfrak{\mathfrak{e}}})= L(s,\pi_1,\pi_2),\ \ \ \mathfrak{e}\in\{+,-\}.\]
	\end{cv}
	
	All the different possible definitions of the above $L$-factors come together with gamma factors attached to them. In each case, 
	the Langlands-Shahidi method also attaches an epsilon factor to each L-function considered above. More precisely, for $\psi$ a non-trivial character of $F$, the epsilon factor 
	$\epsilon(s,\pi, \star,\psi)$ is a unit of $\C[q^{\pm 1}]$ when $F$ is non-archimedean with residual field of size $q$, whereas it is a constant in $\C^\times$ when $F$ is archimedean.  Here, we are in one of the following three situations
 \begin{itemize} 
 \item $\star$ is either $\As^{+}$ or $\As^{-}$ and $\pi$ is an irreducible generic representation of $\GL_n(E)$ if $E$ is a degree two \'etale $F$-algebra (this includes the split case);
 \item or $\star$ is either $\wedge^2$ or $\Sym^2$ and $\pi$ is an irreducible generic representation of $\GL_n(F)$.
\item or $\star=\mathrm{Std}$ is the standard representation and $\pi$ is an irreducible generic representation of $\GL_n(F)$. In this case we simply omit $\mathrm{Std}$ from the notation. 
\end{itemize}

	One then sets 
	\[\gamma(s,\pi, \star, \psi):=\frac{\epsilon(\pi, \star,\psi)L(1-s,\pi^\vee, \star)}{L(s,\pi, \star)}.\]
	For most arguments in this paper, $\epsilon$-factors play no role. Hence instead of $\gamma$-factors we consider the corresponding quotient of $L$-factors. We set	
 \[
 \gamma_0(s,\pi, \star):=\frac{L(1-s,\pi^\vee, \star)}{L(s,\pi, \star)}.
 \] 
 
	The following properties, namely multiplicativity and relation to distinction in the case of discrete series representations allows one to precisely analyze the order of the poles at $s=0$ of these local factors.  
	
	The multiplicativity relation of such factors refer to their behaviour under parabolic induction, and we explicate it in the list below.
	
	\begin{theorem}\label{thm mult rel L fct} Let $F$ be a local field of characteristic zero.
	For  irreducible and generic representations $\pi$ of $\GL_n(F)$ and $\pi_i$ of $\GL_{n_i}(F)$, $i=1,2$ we have
	\begin{itemize}
	\item $L (s,\pi,\pi_1\times \pi_2) = L (s,\pi,\pi_1) L(s,\pi,\pi_2)$ and in particular $L (s,\pi_1\times \pi_2) = L (s,\pi_1) L(s,\pi_2)$.
	\item $L (s,\pi_1\times \pi_2,\star)  = L (s,\pi_1,\star) L(s,\pi_2,\star)L(s,\pi_1,\pi_2)$ where $\star$ is either $\wedge^2$ or $\Sym^2$.
	\end{itemize}
	Let $E/F$ be a quadratic field extension. For irreducible and generic representations $\pi_i$ of $\GL_{n_i}(E)$, $i=1,2$ we have
	\begin{itemize}
	\item $L (s,\pi_1\times \pi_2,\star)  = L (s,\pi_1,\star) L(s,\pi_2,\star)L(s,\pi_1,\pi_2^\vartheta)$ where $\star$ is either $\As^+$ or $\As^-$ and $\vartheta$ is the $E/F$-Galois action.
	\end{itemize}
	\end{theorem}
\begin{proof}
The parameter of $\pi_1\times \pi_2$ is a direct sum $\phi_{\pi_1\times \pi_2}=\phi_{\pi_1}\oplus \phi_{\pi_2}$. Consequently, 
\[
\phi_\pi\otimes \phi_{\pi_1\times \pi_2}=(\phi_\pi\otimes \phi_{\pi_1})\oplus (\phi_\pi\otimes \phi_{\pi_2})
\]
and 
\[
\star(\phi_{\pi_1\times \pi_2})= \begin{cases} \star(\phi_{\pi_1})\oplus \star(\phi_{\pi_2})\oplus (\phi_{\pi_1}\otimes \phi_{\pi_2} )& \star=\wedge^2 \text{ or }\Sym^2\\
\star(\phi_{\pi_1})\oplus \star(\phi_{\pi_2})\oplus(\phi_{\pi_1}\otimes \phi_{\pi_2^\vartheta}) & \star=\As^+ \text{ or }\As^-.\end{cases}\]
The multiplicative relations are therefore immediate from the definition of Artin $L$-factors of direct sums. 
\end{proof}	
For the following lemmas in the archimedean case it will be helpful to recall some explicit information about $L$-factors for square-integrable representations.

The following discussion can be read off the appendix of \cite{MR2533003}.
Every unitary character $\mu$ of $\BR^\times$ is of the form $\mu=\mu_{t,\epsilon}$ where $\epsilon=0,1$, $t\in \BR$ and $\mu(x)=\eta(x)^\epsilon\abs{x}^{it}$ where $\eta$ is the sign character. We have
\[
L(s,\mu_{t,\epsilon})=\pi^{-\frac{s+it+\epsilon}2}\Gamma(\frac{s+it+\epsilon}2).
\]

Every unitary character $\mu$ of $\C^\times$ is either of the form $\mu=\mu_{t,m}$ or $\mu=\mu_{t,m}^\theta$ where $m\in\BZ_{\ge 0}$, $t\in \BR$ and $\mu_{t,m}(z)=z^m(z\bar z)^{it-\frac{m}2}$. Here $\mu^\theta(z)=\mu(\bar z)$. We have
\[
L(s,\mu_{t,m})=L(s,\mu_{t,m}^\theta)=2(2\pi)^{-s-it-\frac{m}2}\Gamma(s+it+\frac{m}2).
\]

Every irreducible square-integrable representation $\pi$ of $\GL_2(\BR)$ has Langlands parameter $\phi_\pi=\Ind_{W_\C}^{W_\BR}(\Omega)$ for some unitary character $\Omega$ of $\C^\times\simeq W_\C$ such that $\Omega^\theta\ne \Omega$.  In this case
\[
L(s,\pi)=L(s,\Omega).
\]
Furthermore, the following are equivalent:
\begin{itemize}
\item $\pi$ has trivial central character;
\item $\Omega=\mu_{0,m}$ or $\Omega=\mu_{0,m}^\theta$ for some odd integer $m\ge 1$;
\item $\phi_\pi$ is symplectic (has its image in $\SL_2(\C)$).
\end{itemize}

\begin{lemma}\label{lem std}
Let $\delta_i$ be an irreducible essentially square integrable representation of $\GL_{n_i}(F)$, $i=1,\dots,k$ such that $r(\delta_i)>-\frac12$ for all $i$ and $\pi=\delta_1\times \cdots\times \delta_k$ is an irreducible representation of $\GL_n(F)$. Then
$L(s,\pi)$ is holomorphic at $s=\frac12$. 
\end{lemma}
\begin{proof}
Applying the multiplicativity of $L$-factors (Theorem  \ref{thm mult rel L fct}) it suffices to show that $L(s,\pi)$ is holomorphic for $\Re(s)>0$ and $\pi$ square-integrable. This follows from the proof of \cite[Theorem 8.2]{MR701565} in the non-archimedean case.  In the archimedean case it follows directly from the definition, the above description of $L$-factors of discrete series for $\GL_1(F)$, $F=\BR$ or $\C$ and for $\GL_2(\BR)$ and holomorphicity of $\Gamma(s)$ for $\Re(s)>0$. 
\end{proof}

\begin{lemma}\label{lem cusprs}
Consider the non-archimedean case. Let $\pi_i$ be an irreducible cuspidal and unitary representation of $\GL_{n_i}(F)$, $i=1,2$. Then $L(s,\pi_1,\pi_2)$ is holomorphic at $s=s_0$ whenever $\Re(s_0)\ne 0$. 
\end{lemma}
\begin{proof}
This is an immediate consequence of \cite[Proposition 8.1]{MR701565}.
\end{proof}

\begin{lemma}\label{lem pairs}
Let $\pi$ be an irreducible and essentially square-integrable representation of $\GL_n(F)$. Then
$L(s,\pi,\pi^\vee)$ has a simple pole at $s=0$. 
\end{lemma}
\begin{proof}
In the non-archimedean case this follows from \cite[Theorem 8.2 and Proposition 8.1]{MR701565}. In the archimedean case, if $\pi$ is a character of $F^\times$ for $F=\R$ or $F=\C$ then this is immediate from the fact that $\Gamma(s)$ has a simple pole at $s=0$. If $\pi$ is on $\GL_2(\R)$ then it is a simple observation that $\phi_\pi\otimes \phi_{\pi^\vee}$ contains the trivial representation with multiplicity one and the lemma follows.
\end{proof}
\begin{lemma}\label{lem extasai}
Let $F'=F$ (resp. $E$) and let $\pi$ be an irreducible and essentially square integrable representation of $\GL_n(F')$. Let $\pi'=\pi$ (resp. the $E/F$-Galois conjugate of $\pi$). If $0<\abs{r(\pi)}<\frac12$ then $L(s,\pi,\pi')$ is holomorphic at $s=0$ and in particular so are $L(s,\pi,\wedge^2)$ and $L(s,\pi,\Sym^2)$ (resp. $L(s,\pi,\As^+)$ and $L(s,\pi,\As^-)$).
\end{lemma}
\begin{proof}
Consider the non-arcimedean case and write $\pi=\St_k(\rho)$ where $\rho$ is cuspidal. Note that $r(\pi)=r(\rho)=r(\pi')$. It follows from Lemma \ref{lem cusprs} and \cite[Theorem 8.2]{MR701565} that $L(s,\pi,\pi')$ is holomorphic at $s=0$ unless $2r(\rho)+k$ is an integer in $[1,k]$. However, by assumption $-1<2r(\rho)\ne 0$. 

In the archimedean case, $\pi$ is either a character of $\GL_1(F')$ or an essentially square integrable representation of $\GL_2(\R)$. Every irreducible component of $\phi_\pi\otimes \phi_{\pi'}$ is of the form $\phi=\phi_{\nu^{2r(\pi)}\tau}$ where $\tau$ is irreducible and square-integrable. It follows that $L(s,\phi)=L(s+2r(\pi),\tau)$ and since $0<\abs{2r(\pi)}<1$ it follows from the explicit description of $L$-factors for square-integrable representations and the location of poles of $\Gamma(s)$ that $L(s,\pi,\pi')$ is holomorphic at $s=0$.

The rest of the lemma follows from \eqref{eq rs decomp} (resp. \eqref{eq asai rs decomp}).

\end{proof}

\begin{lemma}\label{lem dqr int@1rs}
For $\pi_1,\pi_2\in \Pi_F(-\frac12,\frac12)$ we have that $L(s,\pi_1,\pi_2)$ is holomorphic at $s=1$.
\end{lemma}
\begin{proof}
By the multiplicativity, Theorem \ref{thm mult rel L fct}, it suffices to consider the case where  $\pi_1,\pi_2\in \Pi_F(-\frac12,\frac12)$ are essentially square integrable. Since $L(s,\nu^{\alpha_1}\pi_1, \nu^{\alpha_2}\pi_2)=L(s+\alpha_1+\alpha_2,\pi_1,\pi_2)$ this case follows from the fact that if $\delta_1$ and $\delta_2$ are (unitary) square-integrable then $L(s,\delta_1,\delta_2)$ is holomorphic whenever $\Re(s)>0$. In the non-archimedian case this is the displayed statement (6) in the proof of  \cite[Theorem 8.2]{MR701565}. In the archimedean case it is a simple observation that $\phi_{\delta_1}\otimes \phi_{\delta_2}$ decomposes as a direct sum of parameters of unitary square integrable representations. The lemma follows from the direct description of the associated $L$-factors and holomorphicity of $\Gamma(s)$ for $\Re(s)>0$.
\end{proof}

\begin{lemma}\label{lem asai L}
Let $\pi$ be an irreducible and essentially square integrable representation of $\GL_n(E)$ such that $\abs{r(\pi)}<\frac12$.
Then $\pi$ is $\GL_n(F)$-distinguished if and only if $L(s,\pi,\As^+)$ has a pole at $s=0$ and in this case the pole is simple.
\end{lemma}
\begin{proof}
In the archimedean case, $\pi$ is a character of $\C^\times$ and $L(s,\pi,\As^+)=L(s,\chi)$ where $\chi$ is the restriction of $\pi$ to $\BR^\times$. Note that $\pi$ is $\BR^\times$-distinguished if and only if $\chi=\triv_{\BR^\times}$ and $L(s,\chi)$ has a pole at $s=0$ if and only if $\chi=\triv_{\BR^\times}$ in which case the pole is simple. The statement follows.

Consider the non-archimedean case. If $\pi$ is unitary then the equivalence of conditions is \cite[Corollary 1.5]{MR2063106}. Otherwise, $\pi$ is clearly not distinguished and it follows from Lemma \ref{lem extasai} that $L(s,\pi,\As^+)$ is holomorphic at $s=0$. The simplicity of the pole follows from \eqref{eq asai rs decomp}, Lemma \ref{lem pairs} and \cite[Proposition 12]{MR1111204}.
\end{proof}
\begin{lemma}\label{lem lin period}
Let $\pi$ be an irreducible and essentially square integrable representation of $\GL_{k}(F)$ such that $\abs{r(\pi)}<\frac12$. 
\begin{enumerate}
\item\label{part odd} If $k$ is odd then $L(s,\pi,\wedge^2)$ is holomorphic at $s=0$.
\item\label{part even} If $k=2n$ is even then $\pi$ is $\GL_n(F)\times \GL_n(F)$-distinguished if and only if $L(s,\pi,\wedge^2)$ has a pole at $s=0$ and in this case the pole is simple.
\end{enumerate}
\end{lemma}
\begin{proof}
Consider first the case where $k$ is odd. If $k=1$ then $L(s,\pi,\wedge^2)=1$. This takes care, in particular, of the archimedean case. Consider the non-archimedean case. If  $\pi$ is cuspidal then $L(s,\pi,\wedge^2)=1$ by \cite[Theorem 3.6]{MR4074055}. Otherwise, write $\pi=\St_{2t+1}(\rho)$ for $t\in \N$ and $\rho$ cuspidal. Applying in addition \cite[Theorem 5.12]{MR4074055} we have
\[
L(s,\pi,\wedge^2)=\prod_{j=0}^t L(s+2j,\rho,\wedge^2) \cdot \prod_{j=1}^t L(s+2j-1,\rho,\Sym^2)=\prod_{j=1}^t L(s+2j-1,\rho,\Sym^2)
\]
which is holomorphic at $s=0$ by \eqref{eq rs decomp} and Lemma \ref{lem cusprs}. For the rest of the proof we assume that $k$ is even. 

In the archimedean case $\pi$ is a representation of $\GL_2(\R)$. Let $\Omega$ be the character of $W_\C=\C^\times$ such that $\phi_\pi=\Ind_{W_\R}^{W_\C}(\Omega)$. Then $L(s,\pi,\wedge^2)=L(s,\chi)$ where $\chi$ is the character of $\BR^\times$ given by $\chi(t)=\Omega(t^{\frac12})$, $t>0$ and $\chi(-1)=-\Omega(-1)$. Thus $L(s,\pi,\wedge^2)$ has a pole at $s=0$ if and only if $\chi=\triv_{\R^\times}$ if and only if $\pi$ has a trivial central character and when this is the case the pole is simple. It is a simple observation that $\pi$ is $\GL_1(\R)\times \GL_1(\R)$-distinguished if and only if it has a trivial central character. (If $\pi$ has a trivial central character, integration over the torus $\diag(a,1)$, $a\in \R^\times$ converges on the Whittaker model of $\pi$ and defines a non-zero $\GL_1(\R)\times \GL_1(\R)$-invariant linear form.) The lemma follows in this case. 

Consider the non-archimedean case. If $\pi$ is unitary the equivalence of the conditions is \cite[Proposition 3.4]{MR3619910}  (see also \cite[Theorem 6.2]{MR3953435}). If $\pi$ is not unitary then it is clearly not $\GL_n(F)\times \GL_n(F)$-distinguished and $L(s,\pi,\wedge^2)$ is holomorphic at $s=0$ by Lemma  \ref{lem extasai}.
The simplicity of the pole follows from \eqref{eq rs decomp}, Lemma \ref{lem pairs} and \cite[Theorem 1.1]{MR1394521}. 
\end{proof}

	\subsection{The global $L$-functions}
	
 Here $F$ is a number field and $G=G_m(\A)$ where $G_m$ is one of the groups defined in cases \namelink{lin}, \namelink{twlin1},  \namelink{twlin2}, \namelink{gal1},\namelink{gal2}. The global $L$-functions under consideration in this paper are by definition the completed $L$-functions obtained by meromorphic continuation of the product over all places of $F$ of the local, previously defined, $L$-factors which is known to converge in some right half plane of $\C$. For $\pi\in \mathcal{C}^*(G)$ we consider the:
\begin{itemize}
	\item Standard $L$-function: \[L (s,\JL(\pi))  = \prod_{v} L (s,\JL(\pi_{v}));\]
	\item Asai $L$-functions: \[L (s,\JL(\pi),\as^{\mathfrak{e}})  =\prod_{v} L (s,\JL(\pi_{v}),\as^{\mathfrak{e}}),  \ \ \  \mathfrak{e}\in\{+,-\}\]  in the Galois cases \namelink{gal1} and \namelink{gal2}; 
	\item Exterior-square $L$-function: \[L (s,\JL(\pi), \wedge^2) =\prod_{v} L (s,\JL(\pi_{v}),\wedge^2);\]
	\item Symmetric-square $L$-function: \[L (s,\JL(\pi), \Sym^2)  = \prod_{v} L (s,\JL(\pi_{v}),\Sym^2).\]
	\end{itemize}

	We will also use the following global functional equations, which are available thanks to the Langlands-Shahidi method (see \cite[Theorem 7.7]{MR1070599}).
	
	\begin{theorem}\label{thm gfe LS}
	With the above notation, there are nowhere vanisihing entire functions on $\C$ denoted by the letter $\epsilon$ below, such that: 
	\begin{itemize}
	\item For standard L functions \[L (1-s,\JL(\pi)^\vee)=\epsilon(s,\JL(\pi))L (s,\JL(\pi))\] 
	\item  in the Galois cases \namelink{gal1} and \namelink{gal2} we have 
	\[L (1-s,\JL(\pi)^\vee, \as^\mathfrak{e})=\epsilon(s,\JL(\pi), \as^\mathfrak{e})L (s,\JL(\pi), \as^\mathfrak{e}),\ \ \ \mathfrak{e}\in\{+,-\}\] 
	\item in cases \namelink{lin}, \namelink{twlin1} and \namelink{twlin2} we have
	\[L (1-s,\JL(\pi)^\vee, \wedge^2)=\epsilon(s,\JL(\pi), \wedge^2)L (s,\JL(\pi), \wedge^2),\] and 
	\[L (1-s,\JL(\pi), \Sym^2) = \epsilon(s,\JL(\pi), \Sym^2) L (s,\JL(\pi), \Sym^2) .\]
	\end{itemize}
	\end{theorem}
	
We will also use the functional equation of the corresponding partial L-functions. For this we fix a non-trivial automorphic character $\psi=\otimes'_v \psi_v$ of $\A$, for $v$ varying in the set of places of $F$. For $S$ a finite set of places of $F$ and $\pi$ as above, we set 
\[L^S(s,\JL(\pi),\star):=\prod_{v\notin S} L (s,\JL(\pi_{v}),\star).\] The following version of the global functional equation follows from the equality  
\[\epsilon(s,\JL(\pi),\star)= \prod_{v} \epsilon(s,\JL(\pi_v),\star, \psi_v)\] and the triviality of $\epsilon(s,\JL(\pi_v),\star, \psi_v)$ for $v$ outside a large enough set of places of $F$ containing all archimedean ones.

\begin{corollary}\label{cor gfe LS}
With notation as in Theorem \ref{thm gfe LS}, there exists a finite set of places $S_0$ of $F$ containing all archimedean places, such that for any finite set $S_0\subseteq S$: 
\[L^S(s,\JL(\pi),\star)=L^S(1-s,\JL(\pi)^\vee,\star)\prod_{v\in S} \gamma (s,\JL(\pi_{v}),\star,\psi_v) .\]
\end{corollary}

Let $G=G_m(\A)$ be defined by one of the cases \namelink{lin}, \namelink{twlin1},  \namelink{twlin2}, \namelink{gal1},\namelink{gal2} with the extra assumption that $D=F$. That is, $G=\GL_a(\A_E)$ in the Galois cases \namelink{gal1} and \namelink{gal2} and $G=\GL_a(\A)$ otherwise where $a$ is defined by \eqref{eq a&D} after Remark \ref{rem arithmetic vs geometric}.
	\begin{theorem}\label{thm non vanishing s=1}
	Assume that $D=F$ and let $\pi\in \CC(G)$. Then:
	\begin{itemize}
	\item in the Galois cases \namelink{gal1} and \namelink{gal2} we have  
	\[0\leq \Ord_{s=0}(L (s, \pi , \as^\pm ))\ \ \ \text{and}\ \ \ \Ord_{s=0}(\ L (s, \pi , \as^+))+\Ord_{s=0}(L (s, \pi , \as^-))\leq 1,\] where the second inequality is an equality if and only if 
	$\pi^\theta\simeq \pi^\vee$; 
	\item in cases \namelink{lin}, \namelink{twlin1} and \namelink{twlin2} we have 
	\[
	0\leq \Ord_{s=0}(L (s, \pi , \wedge^2 )), \ 0\leq \Ord_{s=0}(L (s, \pi , \Sym^2 ))\] and
	\[ \Ord_{s=0}(L (s, \pi , \wedge^2))+\Ord_{s=0}(L (s, \pi , \Sym^2))\leq 1,\] where the last inequality is an equality if and only if 
	$\pi\simeq \pi^\vee$. 
	\item If $a\ne 1$ then the standard $L$-function $L(s, \pi )$ is entire on $\C$. 
	\end{itemize}
	\end{theorem}
	\begin{proof}
	The third point follows from \cite{MR0342495}. 
	For the first two points, applying the functional equations of Theorem \ref{thm gfe LS} we may replace $\Ord_{s=0}$ by $\Ord_{s=1}$ throughout.
	The non-vanishing of each of the $L$-functions (the inequalities $\ge 0$) is an immediate consequence of \cite[Theorem 1]{MR2230919}. The remaining two inequalities with the criteria for equality follow from the factorizations \eqref{eq rs decomp} and \eqref{eq asai rs decomp}
together with the analytic properties of $L$-functions of pairs established in \cite{MR623137} and \cite{MR618323}. 
	\end{proof}

\subsection{The $L$-factors attached to the symmetric pair}\label{ss Lfactors for pairs}

Let $(G,H,\theta)=(G_m,H_m,\theta_m)_\x$ for $\x\in \{\namelink{lin},\namelink{twlin1},\namelink{gal2},\namelink{twlin2},\namelink{gal1},\namelink{gp}\}$ (in the global set up we never consider the case $\x=\namelink{gp}$). For an irreducible representation $\pi$ of $G(F)$ if $F$ is local and in $\mathcal{C}^*(G(\A))$ in the global case, we consider the product of $L$-functions $\L(s,\pi,\theta)$ defined as follows.
	\begin{itemize}
	\item The Bump-Friedberg $L$-function
	\[\L(s,\pi,\theta)=L(s+\frac12,\JL(\pi))^2L(2s,\JL(\pi),\wedge^2)\]
	in the linear period case \namelink{lin};
	\item  The Guo-Jacquet $L$-function \[\L(s,\pi,\theta)=L(s+\frac12,\BC_{F}^{E}(\JL(\pi)))L(2s,\JL(\pi),\wedge^2)\]
	in the twisted linear period cases \namelink{twlin1} and \namelink{twlin2};
	\item  The even Asai $L$-function
	\[\L(s,\pi,\theta)=L(2s,\JL(\pi),\As^+)\]
	in the Galois cases \namelink{gal1}, \namelink{gal2}, \namelink{gp}. 
	\end{itemize}
In the twisted linear cases $\BC_{F}^{E}$ stands for quadratic Base-Change and 
\[L(s,\BC_{F}^{E}(\JL(\pi))=L(s,\JL(\pi))L(s,\eta_{E/F}\otimes \JL(\pi))\] where 
$\eta_{E/F}$ is the quadratic character attached to the quadratic extension $E/F$ by class field theory (local or global). 

We further attach to the data $(\pi,\theta)$ the auxiliary $L$-function
\[
\DL(s,\pi,\theta)=\begin{cases}
L(2s+1,\JL(\pi),\Sym^2) & x\in\{\namelink{lin},\namelink{twlin1},,\namelink{twlin2}\}\\
L(2s+1,\JL(\pi),\As^-) & x\in\{\namelink{gal1}, \namelink{gal2}, \namelink{gp} \}.
\end{cases}
\]

For uniformity of notation we often follow the following convention, we write
\[
L^+(s,\Pi)=\begin{cases} L(s,\Pi,\wedge^2) &\text{ in cases }\namelink{lin}, \namelink{twlin1} \text{ and }\namelink{twlin2}\\ L(s,\Pi,\As^+) &\text{ in cases }\namelink{gal1},  \namelink{gal2} \text{ and } \namelink{gp}\end{cases} 
\]
and
\[
L^-(s,\Pi)=\begin{cases} L(s,\Pi,\Sym^2) &\text{ in cases }\namelink{lin}, \namelink{twlin1} \text{ and }\namelink{twlin2}\\ L(s,\pi,\As^-) &\text{ in cases }\namelink{gal1},  \namelink{gal2} \text{ and } \namelink{gp} \end{cases}
\]
in the local (resp. global) set-up, for an irreducible (resp. an irreducible cuspidal automorphic) representation $\Pi$ of $G(F)$ (resp. $G(\A)$) for some $n\in \N$.  

Using the above mentioned properties of local $L$-factors, one can deduce the exact order of pole at $s=0$ of $\L(s,\pi,\theta)$ for a distinguished representation $\pi\in \Pi_\D(-\frac12,\frac12)$.	
	\begin{theorem}\label{thm order pole local L}
	Assume $F$ is local and let $(G,H,\theta)=(G_m(F),H_m(F),\theta_m)_\x$, 
	\[
	\x\in \{\namelink{lin},\namelink{twlin1},\namelink{twlin2},\namelink{gal1},\namelink{gal2},\namelink{gp}\}.
	\] 
	Let $\D$ be defined by \eqref{eq a&D} after . Write $\vartheta$ for the identity automorphism of $F$ in cases \namelink{lin}, \namelink{twlin1} and \namelink{twlin2} and for the $E/F$-Galois involution in cases \namelink{gal1} and \namelink{gal2}.
	
\begin{enumerate}
\item\label{part disc} For an irreducible, square integrable and distinguished representation $\pi$ of $G$ we have 
\[
\Ord_{s=0}(\L(s,\pi,\theta))=1.
\]
\item\label{part basic} For an irreducible representation $\pi$ of $G$ of the form $\pi=\tau\times \tau^*$ where $\tau$ is essentially square integrable, not distinguished and such that $\abs{r(\tau)}<\frac12$ we have
\[
\Ord_{s=0}(\L(s,\pi,\theta))=\begin{cases} 1 & \pi \text{ is }H\text{-compatible (See Definition \ref{def compatible})}\\
3 & \text{otherwise}.
\end{cases}
\]
\item\label{part multiplicative} Let $\pi=\pi_1\times \pi_2$ be a representation of $G$ and $\pi_1,\pi_2\in \Pi_\D(-\frac12,\frac12)$ with $\pi_i$ a representation of $G_{m_i}(F)$ (so that $\L(s,\pi_i,\theta)$ is defined), $i=1,2$. Then
\[
\Ord_{s=0}(\L(s,\pi,\theta))=k_1+k_2+k
\]
where
$k_i=\Ord_{s=0}(\L(s,\pi_i,\theta)),\ i=1,2\ \ \ \text{and} \ \ \ k=\Ord_{s=0}(L(s,\JL(\pi_1),\JL(\pi_2)^\vartheta)$.
\item\label{part denominatorL} For $\pi\in \Pi_\D(-\frac12,\frac12)$, $\Ord_{s=0}(\DL(s,\pi,\theta))=0$.
\end{enumerate}
	\end{theorem}
	\begin{proof}
The last part of the theorem follows from Lemma \ref{lem dqr int@1rs} and the factorizations
\eqref{eq rs decomp} and \eqref{eq asai rs decomp}. We proceed with the first three parts.	
	
Since both $\JL$ and quadratic base-change are equivariant with respect to unramified twists, Lemma \ref{lem std} applies to both $\JL(\pi)$ and $\BC_F^E(\JL(\pi))$ and we conclude that $L(s,\JL(\pi))$ (respectively, $L(s,\BC_F^E(\JL(\pi)))$) is holomorphic at $s=\frac12$ in case \namelink{lin} (respectively, in cases \namelink{twlin1} and \namelink{twlin2}).
We conclude that
in the first three parts of the theorem 
\[
\Ord_{s=0}(\L(s,\pi,\theta))=\Ord_{s=0}(L^+(s,\JL(\pi))).
\]

Consider first the case that $\pi$ is square-integrable and distinguished. Then $\JL(\pi)$ is also square integrable. 
In the Galois cases \namelink{gal1} and \namelink{gal2} we claim that $\JL(\pi)$ is a representation of $\GL_{md}(E)$ that is $\GL_{md}(F)$-distinguished. The theorem then follows from Lemma \ref{lem asai L} in this case. In the remaining cases we claim that $\JL(\pi)$ is a representation of $\GL_{2k}(F)$ that is $\GL_k(F)\times \GL_k(F)$-distinguished for some $k$. The theorem then follows from Lemma \ref{lem lin period}\eqref{part even} in this case.

Consider the Galois cases first. If $F$ is archimedean then $\pi$ is a character of $\C^\times$ that is trivial on $\R^\times$ and $\JL(\pi)=\pi$ is therefore distinguished. If $F$ is non-archimedean it follows from \cite[Theorem 1]{MR3858402} that $\JL(\pi)$ is distinguished.  

Consider now the twisted linear cases \namelink{twlin1} and \namelink{twlin2}. If $F$ is archimedean, $(G,H)$ is either $(\BH^\times,\C^\times)$ or $(\GL_2(\BR),\C^\times)$. It follows from \cite[Theorem 1.1]{MR4659866}  that $\phi_{\JL(\pi)}$ is symplectic. By the equivalent conditions preceding Lemma \ref{lem std}, $\JL(\pi)$ is a representation in the discrete series of $\GL_2(\R)$ with a trivial central character. It is therefore $\R^\times \times \R^\times$-distinguished (see the proof of Lemma \ref{lem lin period}). 
 If $F$ is non-archimedean it follows from \cite[Theorem 1.1]{MR4226986} that $\JL(\pi)$ is a representation of $\GL_{2k}(F)$ that is $\GL_k(F)\times \GL_k(F)$-distinguished where $k=dm/2$ in case  \namelink{twlin1} and $k=dm$ in case \namelink{twlin2}. 
 
 Finally, consider the linear case \namelink{lin}. If $F$ is archimedean then $G=\GL_2(\R)$ and $\pi=\JL(\pi)$ is distinguished by assumption. In the non-archimedean case it follows from 
 \cite[Theorem 3.20]{Sign} that $\pi$ has a Shalika model, therefore from \cite[Theorem 6.1]{MR3953435} that $\JL(\pi)$ has a Shalika model and consequently from \cite[Theorem 3.20]{Sign} that $\JL(\pi)$ is $\GL_k(F)\times \GL_k(F)$-distinguished.

This completes the case that $\pi$ is square integrable. We now consider the case where $\pi=\tau\times \tau^*$. We claim that 
\begin{equation}\label{eq iota jl sym}
\JL(\tau^*)^\vartheta\simeq \JL(\tau)^\vee.
\end{equation}
Indeed, in all cases except case \namelink{gal1} this is straightforward from the definition of $\iota$ and $\tau^*$  in Section \ref{sss auxi} and the fact that $\JL$ commutes with taking contragradient and Galois conjugation (the latter follows from the character identity in \cite[Th\'eor\`eme 2.3]{MR2329758} characterizing $\JL$). In case \namelink{gal1} $G=\GL_{2m}(C)$ and $\tau$ is a representations of $\GL_m(C)$. Let $\Pi=\tau \times \tau$. It is an irreducible representation of $G$ and as pointed out in Section \ref{sss auxi} it satisfies 
\[
\Pi^\iota\simeq \Pi^\theta.
\]
It follows from the character identity \cite[Th\'eor\`eme 3.1(a)]{MR2329758} characterizing $\JL$ that
\[
\JL(\Pi^\iota)\simeq \JL(\Pi^\theta)\simeq \JL(\Pi)^\vartheta.
\]
That is
\[
\JL(\tau^\iota)\times \JL(\tau^\iota)\simeq \JL(\tau)^\vartheta \times \JL(\tau)^\vartheta
\]
and since $\JL(\Pi^\iota)$ is generic it uniquely determines the multiset of essentially square integrable representations that it is induced from. That is, we conclude that  
\[
\JL(\tau^\iota)\simeq \JL(\tau)^\vartheta.
\]
Dualizing and applying $\vartheta$ we deduce \eqref{eq iota jl sym}.
It now follows from Theorem \ref{thm mult rel L fct} that
\[
L^+(s,\JL(\pi))=L^+(s,\JL(\tau))L^+(s,\JL(\tau^*))L(s,\JL(\tau),\JL(\tau)^\vee).
\]
It follows from Lemma \ref{lem pairs} that
\[
\Ord_{s=0}(L(s,\JL(\tau),\JL(\tau)^\vee))=1.
\]
If $\JL(\tau)$ is a representation of $\GL_{2k}(F)$ in the linear and twisted linear cases (respectively, of $\GL_t(E)$ in the Galois cases), in light of \eqref{eq iota jl sym}, $\JL(\tau)$ is distinguished with respect to $GL_k(F)\times \GL_k(F)$ (respectively, $GL_t(F)$) if and only if the same holds for $\JL(\tau^*)$. Consequently, the theorem in this case now follows from 
 Lemma \ref{lem lin period} (respectively, Lemma \ref{lem asai L}).
 
 Finally, consider the last part of the theorem where $\pi=\pi_1\times \pi_2$. Since $\JL(\pi)=\JL(\pi_1)\times \JL(\pi_2)$ the theorem in this case is immediate from the multiplicativity property of $L^+(s,\pi)$ in  Theorem \ref{thm mult rel L fct}.

\end{proof}

	\section{Distinguished standard modules}\label{sec::Dist}

Let $F$ be a local field. Let $(G,H,\theta)=(G_m,H_m,\theta_m)$ be defined by one of the cases \namelink{lin}, \namelink{twlin1},  \namelink{twlin2}, \namelink{gal1},\namelink{gal2} of \S\ref{sec: symmetric pair} and by abuse of notation write $\theta(g)$ for $\theta_k(g)$ for $g\in G_k(F)$ for any $k\in \BN$. Let $a$ and $\D$ be defined by \eqref{eq a&D} so that $G=\GL_a(\D)$. 

We use the notation introduced in Section \ref{ss sym}. In particular, $X=\{x\in G: x=\theta(x)^{-1}\}$ is a $G$-space and the map $gH\mapsto g\theta(g)^{-1}$ identifies $G/H$ with $G\cdot e$. We point out that in the cases that we consider, the auxiliary involution $\theta'$ is the identity on $W$.

\subsection{Parabolic orbits and stabilzers for $G/H$}\label{ss orbits}

Let $P=M U$ be a standard parabolic subgroup of $G$. Recall that the double coset space $P\bs G/H$ is in bijection with the $P$-orbits in $G\cdot e\subseteq X$.
A complete set of $P$-good representatives for $P$-orbits in $G\cdot e$  as well as explication of the stabilizer $M_x$ for each orbit representative $x$ can be deduced from the analysis in \cite{MR3719517} in cases \namelink{gal1} and \namelink{gal2}, from \cite{Sign} in case \namelink{lin}, and from \cite{MR3958071} in cases \namelink{twlin1} and \namelink{twlin2}. The analysis in all of the above references is written in the non-archimedean case, however, its archimedean analog holds similarly. We formulate the results for our general local set-up.


For the explicit description of orbit representatives we introduce some further notation. 
Let $\alpha=(a_1,\dots,a_r)$ be the composition of $a$ associated with $P$ and let $I(\alpha)$ be the set of $r\times r$ symmetric matrices $s=(a_{i,j})$ with entries in $\BZ_{\ge 0}$ such that $\sum_{j=1}^r a_{i,j}=a_i$ for all $i=1,\dots,r$. 
By ordering the indices $i,j$ in lexicographic order, we associate to $s$ a composition of $a$ into at most $r^2$ parts (we omit the zero entries of $s$) which is a refinement of $\alpha$ and denote by $P_s=M_sU_s$ the corresponding parabolic subgroup of $G$ of type $s$ contained in $P$. 
Write elements of $M_s$ as $\diag(g_{i,j})$ where $g_{i,j}\in \GL_{a_{i,j}}(\D)$.

Identifying $W$ with permutation matrices in $G$, let $w_s\in {}_MW_M$ be the involution characterized by 
\[
w_s\diag(g_{i,j})w_s=\diag(g_{j,i}), \ \ \ \diag(g_{i,j})\in M_s.
\]
Note that $M_s=M(w_s)$ (see Section \ref{ss paraborb}).
For $s=(a_{i,j})\in I(\alpha)$ let 
\[
L(s)=\{(a_{i,i}^+,a_{i,i}^-)_{i=1}^r\in ((\BZ_{\ge 0})^2)^r:\ \sum_{i=1}^r a_{i,i}^+=\sum_{i=1}^r a_{i,i}^-, \  a_{i,i}^++a_{i,i}^-=a_{i,i}, i=1,\dots,r\}
\] 
and also set
\[
K(\alpha)=\{s=(a_{i,j})\in I(\alpha),\ a_{i,i} \ \text{is even for} \ i=1,\dots,r\}.
\]

The set of $P$-orbits in $G\cdot e$ is parametrized by a 
set $J(\alpha)$ defined as follows:
\begin{itemize}
\item in case \namelink{lin} $J(\alpha)=\{(\frak{s},l),\ \frak{s}\in  I(\alpha), \ l\in L(\frak{s})\}$;
\item in cases \namelink{twlin1} and \namelink{gal2} $J(\alpha)=I(\alpha)$;
\item and in cases \namelink{twlin2} and \namelink{gal1} $J(\alpha)=K(\alpha)$.
\end{itemize}

For $s\in J(\alpha)$, we write $\frak{s}=(a_{i,j})$ for its first coordinate, i.e. $s=\frak{s}$ except in case \namelink{lin}, where  $s=(\frak{s},l)$.
Set $M_s=M_{\frak{s}}$ and $w_s=w_{\frak{s}}$. 

A set of representatives $\{x_s\}_{s\in J(\alpha)}$ for the $P$-orbits in $G\cdot e$ may be chosen so that $x_s\in M_s w_sw_\star$ ($P$-good) and the stabilizer $M_{x_s}$ consists of $\diag(g_{i,j})\in M_s$ such that:
\begin{enumerate}
	\item $g_{j,i} =  g_{i,j}$ for all $i \ne j$ in cases \namelink{lin} and \namelink{twlin2};
	\item $g_{j,i} =  \theta(g_{i,j})$ for all $i\ne j$ in cases \namelink{twlin1} and \namelink{gal2};
	\item $g_{j,i} =  \epsilon g_{i,j} \epsilon^{-1}$ for all $i \ne  j$ in case \namelink{gal1};
	\item $\vartheta_i(g_{i,i})=g_{i,i}$ for all $i$,  where  $\vartheta_i=\theta$ in all cases except \namelink{lin} where $\vartheta_i=\Ad(\diag(I_{a_{i,i}^+},-I_{a_{i,i}^-}))$ and $s=(\frak{s},l)$ with $l=(a_{i,i}^+,a_{i,i}^-)_{i=1}^r\in L(\frak{s})$. 
	\end{enumerate}

\subsection{Some consequences on distinction of parabolic induction}

We continue to use the notation of Section \ref{ss orbits}.
First, recall the following direct application of the geometric lemma of Bernstein and Zelevinsky (see \cite[Proposition 4.1]{MR3541705}).
\begin{lemma}\label{lem gemlem}
Consider the non-archimedean case. Let $\rho$ be a representation of $M$. 
If $I_P^G(\rho)$ is $H$-distinguished then there exists $s\in J(\alpha)$ such that
\[
\Hom_{M_{x_s}}(r_{M_s,M}(\rho),\delta_{P_{x_s}}\delta_P^{-1/2})\ne 0
\] 
where $r_{M_s,M}$ is the normalized Jacquet functor. \qed
\end{lemma}

We further recall the following identity of modulus functions. In the references bellow the identity is formulated in the non-archimedean case but holds similarly in the archimedean case.
\begin{lemma}\label{lem mod}
Let $s\in J(\alpha)$. In case \namelink{lin}, assume furthermore that $s=(\frak{s},l)$ with $\frak{s}\in K(\alpha)$ and  $l=(a_{i,i}^+,a_{i,i}^-)_{i=1}^r$ with $a_{i,i}^+=a_{i,i}^-$ for all $i=1,\dots,r$.
Then $\delta_{P_{x_s}}\delta_P^{-1/2}$ is identically one on $M_{x_s}$.
\end{lemma}
\begin{proof}
The lemma follows from \cite[Proposition 4.3.2]{MR2010737} in case \namelink{gal2} and from \cite[Corollary 6.9]{MR3541705} or \cite[Proposition 3.3]{MR3719517} in case \namelink{gal1}. In cases \namelink{twlin1} and \namelink{twlin2} it is \cite[(5.3)]{MR4338222} (see also Remark 5.4 in ibid.). Finally, in case \namelink{lin} it follows from \cite[Lemma 3.7 (a)]{Sign}.
\end{proof}

\begin{corollary}\label{cor orbit cont}
Consider cases \namelink{twlin1}, \namelink{twlin2}, \namelink{gal1},  \namelink{gal2}. Let $s=(a_{i,j})\in J(\alpha)$ and let $\rho_{i,j}$ be an irreducible essentially square-integrable representation of $\GL_{a_{i,j}}(\D)$ for every $i,j$. For the representation $\rho=\otimes \rho_{i,j}$ of $M_s$ the following are equivalent 
\begin{enumerate}
\item $\Hom_{M_{x_s}}(\rho,\delta_{P_{x_s}}\delta_P^{-1/2})\ne 0$
\item $\rho$ is $M_{x_s}$-distinguished
\item $\rho_{j,i}^\theta\simeq \rho_{i,j}^\vee$ for all $i\ne j$ and $\rho_{i,i}$ is distinguished for all $i$.
\end{enumerate}
\end{corollary}
\begin{proof}
This is an immediate consequence of Lemma \ref{lem mod} and the explication of $M_{x_s}$ in Section \ref{ss orbits}.
\end{proof}
For linear periods the analog of the Corollary is more complicated to formulate. Since the results we need in this case have been established in \cite{Sign} we only recall that a crucial role is played by the modulus computation \cite[Lemma 3.7]{Sign} and the following result. 
\begin{lemma}\label{lem lin sqrint}
Let $k\in \N$ and $\pi$ be an irreducible essentially square-integrable representation of $\GL_k(\D)$. If 
\[
\Hom_{\GL_c(\D)\times \GL_{k-c}(\D)}(\pi,\chi)\ne 0
\]
for some $0\le c\le m$ and character $\chi$ of $\GL_c(\D)\times \GL_{k-c}(\D)$ then either $k=2c$ or $k=1$ and $\pi=\chi$.
\end{lemma}
\begin{proof}
In the archimedean case we must have either $k=1$ or $k=2$ and $\D=\R$ and the statement is straightforward since characters of $\GL_2(\R)$ are not essentially square-integrable. The non-archimedean case is \cite[Theorem 3.8]{Sign}.
\end{proof}

The next Theorem is a characterization of distinguished standard modules on $G$. 

A standard module on $G$ is a representation $S=\delta_1\times \cdots\times \delta_t$ where $\delta_i$ is an essentially square-integrable representation of $\GL_{a_i}(\D)$ for $i=1,\dots,t$ and $r(\delta_1)\ge \cdots\ge r(\delta_t)$. 
The representation $S$ admits a unique irreducible quotient $\pi$.  The multi-set $\{\delta_1,\dots,\delta_t\}$ is uniquely determined by $\pi$. This gives a bijection, Langlands classification, from multi-sets of essentially square-integrable representations to irreducible representations. We point out, however, that $S\simeq \delta_{\sigma(1)} \times \cdots \times \delta_{\sigma(t)}$ for potentially many $\sigma\in \mathfrak{S}_t$. (For all $\sigma$ if $\pi$ is generic.) The realization of $S$ matters in the application we have in mind.

In this section, in the non-archimedean case, a cuspidal  segment $\Delta(\rho,x,y)$ will always be presented with $\rho$ unitary. We write $L(\Delta(\rho,x,y)) \preccurlyeq L(\Delta(\rho',x',y')) $ if $y'\ge y$ and if equality holds then also $x'\ge x$.

We say that an ordered multi-set of essentially square-integrable representations $(\delta_1,\dots,\delta_t)$ is right aligned if $\delta_j \preccurlyeq \delta_i$ whenever $i<j$.
When this is the case the induced representation $\delta_1\times\cdots \times \delta_t$ is isomorphic to the standard module associated to the multi-set $\{\delta_1,\dots,\delta_t\}$. 

The explication of the Jacquet module of an essentially square integrable representation is key to the following lemma. We recall it here. Let $\rho$ be an irreducible cuspidal (unitary) representation of $\GL_m(\D)$, $\delta=L(\Delta(\rho,x,y))$ be an essentially square-integrable representation of $\GL_k(\D)$, $\alpha=(c_1,\dots,c_t)$ a composition of $k$ and $M_\alpha=\GL_{c_1}(\D)\times \cdots\times \GL_{c_t}(\D)$ the associated Levi subgroup of $\GL_k(\D)$. Then $r_{M_\alpha,\GL_k(\D)}=0$ unless $m\mid c_i$ for all $i=1,\dots,t$ in which case
\begin{multline}\label{eq jm}
r_{M_\alpha,\GL_k(\D)}(\delta)=\delta_1\otimes \cdots \otimes \delta_t, \ \ \ \text{where} \ \ \ \delta_i=L(\Delta(\rho, d_i+1,d_{i-1})),\\ \text{with}\ \ \ d_0=y, \ d_i=d_{i-1}-\frac{c_i}m,\,i=1,\dots,t.
\end{multline}

We further freely use that
\[
L(\Delta(\rho,a,b))^\vee=L(\Delta(\rho^\vee,-b,-a))\ \ \ \text{and}\ \ \ L(\Delta(\rho,a,b))^\iota=L(\Delta(\rho^\iota,a,b))
\]
(see Section \ref{sss auxi}) so that 
\[
L(\Delta(\rho,a,b))^*=L(\Delta(\rho^*,-b,-a)
\]

and that if $L(\Delta(\rho,a,b))$ is a distinguished representation of $G$ then it is unitary, that is, $a+b=0$.
\begin{lemma}\label{lem admiss contribute}
Assume that $F$ is non-archimedean. Let $(\delta_1,\dots,\delta_r)$ be a right aligned multi-set of essentially square-integrable representations so that $\delta=\delta_1\otimes \cdots \otimes \delta_r$ is an essentially square-integrable representation of the Levi $M$ of $G$. If  $s\in J(\alpha)$ is such that
\[
\Hom_{M_{x_s}}(r_{M_s,M}(\delta),\delta_{P_{x_s}}\delta_P^{-1/2})\ne 0
\] 
then $M_s=M$, that is, $\frak{s}$ is a monimial matrix and in case \namelink{lin} furthermore, $s=(\frak{s},l)$ with $\frak{s}\in K(\alpha)$ and  $l=(a_{i,i}^+,a_{i,i}^-)_{i=1}^r$ with $a_{i,i}^+=a_{i,i}^-$ for all $i=1,\dots,r$. In particular, if $\frak{s}=(a_{i,j})$ and $\sigma\in S_r$ is the involution such that $a_{i,\sigma(i)}=a_i$ then 
\begin{itemize}
\item $\delta_{\sigma(i)}^\theta\simeq \delta_i^\vee$ whenever $i\ne \sigma(i)$ and
\item $\delta_i$ is distinguished whenever $i=\sigma(i)$.
\end{itemize}
\end{lemma}
\begin{proof}
We show that $M_s=M$. The remainder of the lemma follows from the proof of \cite[Theorem 3.12]{Sign} in case \namelink{lin} and is immediate from Corollary \ref{cor orbit cont} in the other cases.  

We freely use the explicit description of $M_{x_s}$ in Section \ref{ss orbits}. 
Write $\delta_i=L(\Delta(\rho_i,b_i,e_i))$, $i=1,\dots,r$ and for each $i$ let $r_{M_{\alpha_i},\GL_{a_i}(\D)}(\delta_i)=\delta_{i,1}\otimes \cdots\otimes \delta_{i,r}$ for the composition $\alpha_i=(a_{i,1},\dots,a_{i,r})$ of $a_i$ associated to the $i$th row of $\frak{s}$. We proceed by induction.  Assume by contradiction that $a_{1,j}\ne 0$ and $a_{1,j'}\ne 0$ for some $j<j'$. Then there exist $b_1\le \alpha \le \beta<\gamma\le \delta\le e_1$ such that $\delta_{1,j}=L(\Delta(\rho_1,\gamma,\delta))$ and $\delta_{i,j'}=L(\Delta(\rho_i,\alpha,\beta))$. Applying Corollary \ref{cor orbit cont} and the remarks preceding this lemma it follows that $e_j=-\gamma$ and $e_{j'}=-\alpha$. Since $-\gamma<-\alpha$ this contradicts the right aligned assumption.
Consequently, $a_{1,j}\ne 0$ for a unique $j$. If $j=1$ the lemma follows by induction on the right aligned multiset $(\delta_2,\dots,\delta_r)$. Assume that $j>1$ so that $a_1=a_{1,j}=a_{j,1}\le a_j$. If $a_1=a_j$ then similarly the lemma follows by induction on $(\delta_2,\dots,\delta_{j-1},\delta_{j+1},\dots,\delta_r)$.
Assume by contradiction that $a_1<a_j$ and note that since $\delta_1^*=L(\Delta(\rho_1^*,-e_1,-b_1))\simeq \delta_{j,1}$ if $k>1$ is such that $a_{j,k}\ne 0$ then we must have $e_k>e_1$, once again, a contradiction to the right aligned assumption. The lemma follows.
\end{proof}

The following theorem is the classification of distinguished standard modules. In most cases it is already known. More precisely, in the non-archimedean case all cases are already written except case \namelink{gal1}: \cite[Proposition 10.3]{MR4308058} in case \namelink{gal2},  \cite[Theorem 1.3]{MR4157701} in cases \namelink{twlin1} and \namelink{twlin2}, and \cite[Theorem 3.12]{Sign} in case \namelink{lin}. Most archimedean cases are also essentially proved already, and we refer to the proof below for the precise references. We recall the steps for convenience of the reader.

	\begin{theorem}\label{thm classif dist S}
	Let $S$ be a standard module of $G$ associated to the multiset $\{\delta_1,\dots,\delta_r\}$ of essentially square integrable representations. Then $S$ is distinguished if and only if there exists an involution $p \in \mathfrak{S}_r$ such that $\delta_i^* \cong \delta_{p(i)}$ for all $i$, and $\delta_i$ is distinguished if $p(i) = i$.
	\end{theorem}
	\begin{proof}
First we establish that if $\delta$ is a distinguished, irreducible and essentially square-integrable representation then $\delta\simeq \delta^*$. In the non-Galois cases \namelink{lin}, \namelink{twlin1} and \namelink{twlin2} this is in fact known for all irreducible representations (\cite[Corollary 5.8 and Theorem 6.7]{MR4338222} in cases \namelink{twlin1} and \namelink{twlin2} and  \cite[Appendix A]{Sign} in case \namelink{lin}). Consider the Galois cases \namelink{gal1} and \namelink{gal2}. In the archimedean case, $\delta$ is a one-dimensional character of $\C^\times$ and the statement is straightforward. In the non-archimedean case, if $G=\GL_k(E)$ and $H=\GL_k(F)$ for some $k\in \BN$ then this is \cite[Proposition 12]{MR1111204}. For the general case, it follows from \cite[Theorem 1]{MR3858402} that $\JL(\delta)$ is also distinguished and therefore $\JL(\delta)\simeq \JL(\delta)^*$. The injectivity of the Jacquet-Langlands correspondence implies that $\delta\simeq \delta^*$.

Note that the existence of the involution $p$ implies the existence of $s\in J(\alpha)$ such that $M_s=M$ and $\delta_1 \otimes \cdots\otimes \delta_r$ is an $M_{x_s}$-distinguished representation of $M$.
It now follows from \cite[Theorem 5.4]{MR4679384} that $S$ is distinguished.

For the `if' part of the theorem, consider first the archimedean case.
It follows from \cite[Theorem D3-appendix]{Sign} in case \namelink{lin}, from \cite[Theorem 1.2]{MR4659866} in cases \namelink{twlin1} and \namelink{twlin2} and from \cite[Theorem 1.2]{MR3358053} in case \namelink{gal2} but a similar proof holds up to obvious modifications in case \namelink{gal1} in view of the double coset and stabilizers comptutations in \cite{MR3719517}. This is being written up by the second named author's students, Alan Hou and Tudor Popescu.
In the non-archimedean case it is immediate from Lemmas \ref{lem gemlem} and \ref{lem admiss contribute}.
	\end{proof}
	 
	\begin{corollary}\label{cor distinction-transfer-generic}
Let $\pi$ be an irreducible, generic representation of $G$. 
\begin{enumerate}
\item If $\pi$ is essentially square-integrable then in cases \namelink{gal1}, \namelink{gal2} and \namelink{lin} we have that $\pi$ is distinguished if and only if $\JL(\pi)$ is distinguished.
\item In case \namelink{gal2}, and if $d$ is odd also in case \namelink{lin} we have that $\pi$ is distinguished if and only if $\JL(\pi)$ is distinguished.
\item In case \namelink{gal1}, and if $d$ is even also in case \namelink{lin} we have that if $\pi$ is distinguished then $\JL(\pi)$ is distinguished.
\end{enumerate}
	\end{corollary}
	\begin{proof}
Consider first the case where $\pi$ is essentially square-integrable. In the archimedean case we must have $\pi=\JL(\pi)$ and the result is trivial. Consider the non-archimedean case. In the Galois cases the result follows from \cite[Theorem 1]{MR3858402}. In case \namelink{lin} it is a consequence of the combination of \cite[Theorem 1.2]{MR3953435} and  \cite[Corollary 3.4]{Sign}. The first of the two references shows that $\pi$ has a Shalika model if and only if $\JL(\pi)$ does. The second implies that $\pi$ has a Shalika model if and only if it has a linear model and similarly for $\JL(\pi)$.

For $\pi$ generic the corollary now follows from \eqref{eq JLmult}, Theorem \ref{thm classif dist S} and the fact that $\JL(\delta^*)=\JL(\delta)^*$ for $\delta$ essentially square-integrable. 
	\end{proof}
	
\begin{remark}
\begin{enumerate}
\item
In case \namelink{gal1} and if $d$ is even in case \namelink{lin} the converse implication of Corollary \ref{cor distinction-transfer-generic} may fail for the following reason, if $\pi=\delta_1\times \cdots\times \delta_k$ where some $\delta_i$ is a representation of $\GL_t(\D)$ for $t$ odd and yet $\JL(\delta_i)$ is distinguished. See Remark \ref{rmk counter example dist} for such an example.
\item In cases \namelink{twlin1} and \namelink{twlin2}) the analog of Corollary \ref{cor distinction-transfer-generic} is well-known to be false due to the epsilon dichotomy phenomenon.
\end{enumerate}
\end{remark} 

For our local global principle, we are in fact only interested in compatible generic representations (Definition \ref{def compatible}) which are distinguished. Thanks to Theorem \ref{thm classif dist S}, they can be described in a very precise manner as follows. Before, we recall that Remark \ref{rem compatibility of discrete series} explains even further when discrete series occuring in the essentially square-integrable support of a generic representation can contribute to its $H$-incompatibility. We also recall the notation $\mathcal{D}$ from Equation \eqref{eq a&D} after Remark \ref{rem arithmetic vs geometric}, and that $d$ is the degree of $D$.

\begin{corollary}\label{cor dist comp}
Let $\pi$ be an irreducible generic and distinguished representation of $G$. According to Theorem \ref{thm classif dist S},  
\[
\pi\simeq \delta_1\times \delta_1^* \cdots \times \delta_k \times \delta_k^* \times \tau_1 \times \dots \times \tau_\ell,
\] 
for some integers $k,\ell\ge 0$ and $\delta_i,\tau_j\in \ESI_\D$, $i\in [1,k]$, $j\in [1,\ell]$ and furthermore each $\tau_j$ is distinguished whereas no $\delta_i$ is distinguished. If each $\delta_i\in \Pi_{\ESI}(m_i,\D)$ then $\pi$ is $H$-compatible if and only if  for $i\in[1,k]$ the representation $\JL(\delta_i)$ is not $\GL_{dm_i/2}(F)\times \GL_{dm_i/2}(F)$-distinguished in cases \namelink{lin}, \namelink{twlin1} and \namelink{twlin2}, respectively not $\GL_{d'm_i}(F)$-distinguished in cases \namelink{gal1} and  \namelink{gal2} where $d'$ equals $d/2$ in case \namelink{gal1} and equals $d$ in case \namelink{gal2} is the degree of $\D$ over $E$.
 
\end{corollary} 

\begin{remark}\label{rem dist comp}
In cases \namelink{lin}, \namelink{twlin1} and \namelink{twlin2}, the condition 
\[
\JL(\delta_i) \   \text{is not} \ \GL_{dm_i/2}(F)\times \GL_{dm_i/2}(F)-\text{distinguished}
\] 
can be restated as 
\[
\text{the Langlands parameter of}\ \JL(\delta_i) \ \text{is not symplectic},
 \]
 whereas in cases \namelink{gal1} and \namelink{gal2}, the condition 
 \[
 \JL(\delta_i) \ \text{is not}\ \GL_{d'm_i}(F)-\text{distinguished}
 \] 
 can be restated as 
 \[
 \text{the Langlands parameter of}\ \JL(\delta_i)\ \text{is not conjugate-orthogonal}. 
 \]
\end{remark}

Another consequence of Theorem \ref{thm classif dist S} is the completion of the following result.

\begin{theorem}\label{thm mult 1 and ssduality}
Let $\pi$ be an irreducible distinguished representation of $G$. Then
\begin{enumerate}
\item $\dim(\Hom_H(\pi,\C))= 1$ and
\item $\pi\simeq \pi^*$. 
\end{enumerate}
\end{theorem}
\begin{proof}
This is \cite[Corollary 5.8 and Theorem 6.7]{MR4338222} in cases \namelink{twlin1} and \namelink{twlin2} and  \cite[Appendix A]{Sign} in case \namelink{lin}. 

Consider the Galois cases \namelink{gal1} and \namelink{gal2}. 
We first prove that $\pi\simeq \pi^*$. Suppose $\pi$ is the unique ireducible quotient of the standard module $S=\delta_1\times \cdots\times \delta_r$ associated to the multiset $\{\delta_1,\dots,\delta_r\}$ of essentially square-integrable representations. Then $S$ is also distinguished and it therefore follows from Theorem \ref{thm classif dist S} that $\{\delta_1,\dots,\delta_r\}$ is stable under $\delta\mapsto \delta^*$. Since $\pi^*$ is the unique irreducible quotient of the standard module $\delta_r^*\times \cdots\times \delta_1^*$ it follows that $\pi^*\simeq \pi$. 

We therefore have
\[
\dim(\Hom_{H}(\pi,\C))=\dim(\Hom_{H}(\pi^*,\C))=\dim(\Hom_{H}(\pi^\vee,\C)).
\]
(The second equality is since $H=\theta(H)$ and $\pi^*\simeq (\pi^\vee)^\theta$).
In order to prove the multiplicity one result, it therefore suffices to show that $(G,H,\theta)$ is a GP2-pair (see \cite[Definition 8.1.2]{MR2553879}).
By \cite[Corollary 8.1.6]{MR2553879} it suffices to show that it is a GK-pair (see \cite[Definition 7.1.8]{MR2553879}). 
By \cite[Theorem 7.6.2]{MR2553879} any Galois symmetric space is tame (see \cite[Definition 7.3.1]{MR2553879}).
By \cite[Remark 7.3.2 and Theorem 8.1.5]{MR2553879} it suffices to show that $(G,H,\theta)$ is good (see \cite[Definition 7.1.6]{MR2553879}). This follows from  \cite[Corollary A2(1)]{MR1277452} which is formulated in the non-archimedean case, but its proof is valid verbatim in the archimedean case.
\end{proof}

	\begin{remark}\label{rmk counter example dist}
	In case \namelink{lin} and when the degree $d$ of $D$ is even, it can happen that $\jl(\pi)$ is distinguished for generic $\pi$ without $\pi$ being distinguished. The most trivial example is to take $\chi$ a quadratic character of $D^\times$. Then $\jl(\chi)=\st_d(\chi)$ which is known to be $\GL_{d/2}(F)\times \GL_{d/2}(F)$-distinguished (see \cite{MR3168918}). A slightly more elaborate example, where distinction for $\pi$ would make sense, is as follows. As observed in \cite{MR2654780}, for two different quadratic characters $\chi$ and $\chi'$ of $D^\times$ the induced representation $\pi=\chi\times \chi'$ of $G=\GL_2(D)$ is not $D^\times \times D^\times$-distinguished whereas $\jl(\pi)=\st_d(\chi)\times \st_d(\chi')$ is $\GL_d(F)\times \GL_d(F)$-distinguished. 
	\end{remark}

For the next proposition we use the following convention. In cases \namelink{lin}, \namelink{twlin2} and \namelink{gal1} when $a=2m$, for an odd integer $k$ we recall that every representation of $\GL_k(\D)$ is not distinguished.

	\begin{proposition}\label{prop open contribution two discrete} Assume that $F$ is non-archimedean.
	Let $a=a_1+a_2$ and $\delta_i$ be an essentially square-integrable representation of $\GL_{a_i}(\D)$, $i=1,2$ and assume that $r(\delta_1)\ge r(\delta_2)$ and that at least one of $\delta_1,\delta_2$ is not distinguished. Let $S=\delta_1\times \delta_2$ be the corresponding standard module of $G$
and $S^\circ$ the $H$-invariant subspace of sections in $S$ supported on the unique open $(P_{(a_1,a_2)},H)$-double coset. Then $ \Hom_H (S^\circ , \C)$ is one dimensional if $a_1=a_2$ and $\delta_2\simeq \delta_1^*$ and zero otherwise and the restriction map $\ell\mapsto \ell|_{S^\circ}: \Hom_H (S , \C) \rightarrow  \Hom_H (S^\circ , \C) $ is a bijection.
		\end{proposition}
\begin{proof}
Let $M$ be the Levi subgroup of $G$ of type $(a_1,a_2)$ so that $\delta=\delta_1\otimes \delta_2$ is a representation of $M$. It follows from Lemma \ref{lem admiss contribute} that 
\[
\Hom_{M_{x_s}}(r_{M_s,M}(\delta),\delta_{P_{x_s}}\delta_P^{-1/2})=0
\]
unless $a_1=a_2$, $\delta_2\simeq \delta_1^*$ and $\frak{s}=\begin{pmatrix} 0 & a_1 \\ a_1 & 0\end{pmatrix}$ in which case $P\cdot x_s$ is the unique open $P$-orbit in $G\cdot e$ and by irreducibility of $\delta_1$,  $\Hom_{M_{x_s}}(r_{M_s,M}(\delta),\delta_{P_{x_s}}\delta_P^{-1/2})=\Hom_{M_{x_s}}(\delta,\C)$ is one dimensional. (Here $M_{x_s}=\{(g,g^\iota): g\in \GL_{a_1}(\D)\}$, see Section \ref{sss auxi}.)  As a consequence of the geometric lemma \cite[Proposition 4.1]{MR3541705}  (see \eqref{eq geom lem}), restriction to $S^\circ$ defines an imbedding  $\Hom_H(S,\C) \hookrightarrow \Hom_H(S^\circ,\C)\simeq \Hom_{M_{x_s}}(\delta,\C)$.
The lemma follows.
\end{proof}

\subsection{An application of the archimedean geometric lemma}\label{ss archprop}

For the special case where $\delta_2\simeq \delta_1^*$ we require an archimedean analog of Proposition \ref{prop open contribution two discrete}.
We apply an archimedean analog of the geometric lemma for Nash groups in the language introduced by Chen and Sun in \cite{MR4211018}.

\subsubsection{Induced representations as vector bundles}

We introduce some of the language from \cite{MR4211018} specialized to real reductive groups. Let $G$  be the group of $\BR$-points of a real reductive group, and $P$ be a closed algebraic subgroup of $G$. Let $(\sigma,V_{\sigma})$ be a smooth Frechet representation of $P$. 
Consider the right $P$-action on $G\times V_\sigma$ given by
\[(g,v)\cdot p= (gp,\sigma(p^{-1})v).\] 
It commutes with the left $G$-action $x\cdot (g,v)=(xg,v)$, $x,g\in G$, $v\in V_\sigma$.
Denote by 
\[G\times^P  V_{\sigma}\] the space of $P$-orbits in $G\times V_{\sigma}$. It is a $V_{\sigma}$-bundle over $G/P$ with associated projection 
defined by 
\[(g,v)P\mapsto gP: G\times^P  V_{\sigma} \to G/P.\] 
It is proved in  \cite[Section 3.3]{MR4211018}  that this is a tempered bundle (see \cite[Definition 2.14]{MR4211018}). The space \[\mathcal{S}(G/P,G\times^P  V_{\sigma})\] of Schwartz sections is defined in \cite[Section 6.1]{MR4211018}. It consists of sections \[s:G/P\to G\times^P  V_{\sigma}\] of the bundle satisfying certain regularity properties. Following \cite[Proposition 6.3]{MR4211018} there is an action of $G$ on $\mathcal{S}(G/P,G\times^P  V_{\sigma})$
given by $\rho(g)s=g\cdot s(g^{-1} \ \cdot )$, which confers $\mathcal{S}(G/P,G\times^P  V_{\sigma})$ a $G$-module structure. The content of 
\cite[Proposition 6.7]{MR4211018}  is that the map 
\[f\to [gP\to (g,f(g^{-1}))P]\] identifies the smooth Frechet representation of $G$ induced from $(P,\sigma)$ with non-normalized induction with the $G$-module $\mathcal{S}(G/P,G\times^P  V_{\sigma})$, that is, \[\Ind_P^G(\sigma)\simeq \mathcal{S}(G/P,G\times^P  V_{\delta_P^{1/2}\sigma})\] 
where $\Ind_P^G$ stands for normalized induction.
\subsubsection{The filtration}\label{sss filtration}

Assume here that $H$ is a symmetric subgroup of $G$ and $P$ a parabolic subgroup of $G$ such that there is a unique open $(P,H)$-double coset in $G$ and let $u\in G$ be such that $U_0=PuH$ is this unique $(P,H)$-double coset. We enumerate the other double cosets $Pu_1H,\dots, Pu_rH$ and assume that they are ordered so that
\[U_i:=U_0 \cup_{j=1}^i Hu_jP\]
is open for $i=1,\dots,r$.

As in \cite[Section 1.6 (13)]{MR4211018} one has a closed linear embedding given by extension by zero 
\[\mathcal{S}(U_0/P,(G\times^P  V_{\delta_{P}^{1/2}\sigma})_{|U_0})\hookrightarrow\mathcal{S}(G/P,G\times^P  V_{\delta_{P}^{1/2}\sigma}).\]

Denote by $V_0$ the $H$-invariant subspace of $\Ind_P^G(\sigma)$ of sections supported in $U_0$. Then under the $G$-module isomorphism 
\[\mathcal{S}(G/P,G\times^P  V_{\delta_{P}^{1/2}\sigma})\simeq \Ind_P^G(\sigma),\] we get an $H$-module isomorphism 
\[\mathcal{S}(U_0/P,(G\times^P V_{\delta_{P}^{1/2}\sigma})_{|U_0})\simeq V_0.\]

As in \cite{MR4659866} the quotient 
\[Q:=\mathcal{S}(G/P,G\times^P V_{\delta_{P}^{1/2}\sigma})/\mathcal{S}(U_0/P,(G\times^P  V_{\delta_{P}^{1/2}\sigma})_{|U_0})\] is equipped with a filtration. More precisely, as in \cite[Introduction]{MR4659866}, $Q$ first has a finite filtration with respective subquotients $Q_1, \cdots , Q_r$ where each $Q_i$ corresponds to the representative $u_i$. Explicitly: 
\[Q_i=\mathcal{S}(U_i/P,G\times^P  \delta_{P}^{1/2}V_{\sigma})_{|U_i})/\mathcal{S}(U_{i-1}/P,(G\times^P  \delta_{P}^{1/2}V_{\sigma})_{|U_{i-1}}).\]
Moreover each $Q_i$ admits an infinite filtration with consecutive subquotients $Q_{i,k}$ for $k\in \BN$, where explicitly 
\[Q_{i,k}=\mathcal{S}(H/P_i,H\times^{P_i} V_{\delta_P^{1/2} \sigma}\otimes S_{i,k}).\] 
Here 
\[P_i=P\cap u_iHu_i^{-1}\ \ \ \text{and} \ \ \ S_{i,k}=Sym^k((\mathfrak{g}/\mathfrak{h}+\Ad(u_i)\mathfrak{p})_{\C}^\vee)\]  
where $\mathfrak{g}$, $\mathfrak{h}$ and $\mathfrak{p}$ are the respective Lie algebras of $G$, $H$ and $P$.
Observe that \[Q_{i,0}=\mathcal{S}(H/P_i,H\times^{P_i}  V_{\delta_P^{1/2}\sigma}) .\]

\subsubsection{Consequences in special cases relevant to us}

Assume that $G=G_m(F)$ and $H=H_m(F)$ where $(G_m,H_m)$ are defined by one of the cases \namelink{lin}, \namelink{twlin1},  \namelink{twlin2}, \namelink{gal1},\namelink{gal2} of \S\ref{sec: symmetric pair} and let $a$ and $\D$ be as in \eqref{eq a&D} after Remark \ref{rem arithmetic vs geometric}. Let $\pi$ be a standard module of $G$ of the form $\pi=\delta\times \delta^*$ where $\delta$ is irreducible essentially square integrable and $r(\delta)\ge 0$. In particular, either $a=2$ or $\D=\R$ and $a=4$.

Let $P=MU$ be the parabolic subgroup of $G$ of type $(\frac{a}2,\frac{a}2)$ and $\sigma=\delta\otimes \delta^*$ so that $\pi=\Ind_P^G(\sigma)$. 
It easily follows from Section \ref{ss orbits} that there is a unique open $(P,H)$-double coset in $G$.
In the next proposition and its proof we freely apply the notation of Section \ref{sss filtration}. 

Note that $m\in\{1,2\}$. If $m=1$, by convention, we say that any representation of $\GL_{\frac{a}2}(\D)$ is not distinguished. 
\begin{proposition}\label{prop arch geom lem}
With the above notation if $\delta$ is not distinguished then the restriction map $\ell\mapsto \ell|_{V_0}: \Hom_H (\pi , \C) \rightarrow  \Hom_H (V_0 , \C) $ is injective.
\end{proposition}
\begin{proof}
If either $(G,H)=(\GL_2(\R),\C^\times)$ or $(G,H)=(\GL_2(\C),\BH^\times)$ then $G=PH$ so that $\pi=V_0$ and the proposition is straightforward. This concludes case \namelink{gal1} and in case \namelink{twlin2} leaves only the case where $(G,H)=(\GL_4(\R),\GL_2(\C))$. 

Consider either this case or cases \namelink{lin}, \namelink{twlin1} or \namelink{gal2}. We observe that $P$ is $\theta$-stable. Furthermore, call a $(P,H)$-double coset admissible if it contains a representative $v$ such that $v\cdot e$ is $P$-admissible. One can explicitly show that if $Pu_i H$ is admissible then $u_i$ can be chosen so that it normailzes $H$ so that $P_i=P\cap H$ is independent of the admissible double coset. 

We claim that the Schwartz homology spaces satisfy
\begin{equation}\label{eq schwartz homology}
H_0(H,Q_i)=0,\ i=1,\dots,r.
\end{equation}

We first treat the remaining cases except \namelink{lin}. Note that in those cases, the assumption on $\delta$ implies that $\sigma$ is not $M\cap H$-distinguished and furthermore, all $(P,H)$-double cosets are admissible. 
The conditions (A), (B), (C), (D) of \cite[Section 5.3]{MR4659866} are satisfied and applying  \cite[Theorem 5.8]{MR4659866} we conclude that 
\[H_0(H,Q_{i,k})=\{0\}, \ \forall k>0.\]
Furthermore, by Shapiro's lemma \cite[Lemma 3.7]{MR4659866} and Lemma \ref{lem mod}, we have \[H_0(H,Q_{i,0})=H_0(P\cap H,\sigma),\] and by \cite[Lemma 3.8]{MR4659866}, the continuous dual of 
$H_0(P\cap H,\sigma)$ is \[H_0(P\cap H,\sigma)^\vee\simeq \Hom_{P\cap H}(\sigma,\C)=0.\] 
Hence by the proof of \cite[Lemma 5.2]{MR4659866}, this implies that for each $i=1,\dots,r$
\[\dim H_0(H,Q_i)\leq \dim H_0(H,Q_{i,0})=0,\] and therefore \eqref{eq schwartz homology} follows.

Next we consider the linear cases \namelink{lin}. Condition (A) of \cite[Section 5.3]{MR4659866} is no longer satisfied, however, 
 it follows from the proof of \cite[Appendix, Lemma D.4]{Sign} that 
all homology spaces $H_0(H,Q_{i,k})$ are equal to zero except if $k=0$ and $Pu_iH$ is admissible. 
For the case $(G,H)=(\GL_4(\R),\GL_2(\R)\times \GL_2(\R))$ the representatives of admissible orbits may be chosen to normalize $H$ and the assumption on $\delta$ implies that $\sigma$ is not $M\cap H$-distinguished. In this case \eqref{eq schwartz homology} follows now in the same way as in the previous cases. The remaining cases are $(G,H)=(\GL_2(\D), \D^\times\times \D^\times)$ for $\D\in \{\R,\C,\BH\}$ where $\delta$ is any irreducible representation such that $r(\delta)\ge 0$. In these cases $r=2$ and the two non-open double cosets are closed and $P_i=H$, $i=1,2$. Shapiro's lemma \cite[Lemma 3.7]{MR4659866} now gives
\[H_0(H,Q_{i,0})=H_0(H,\delta_P^{\frac12}\sigma),\ i=1,2.\] 
and by \cite[Lemma 3.8]{MR4659866}, the continuous dual of $H_0(H,\delta_P^{\frac12}\sigma)$ is
\[H_0(H,\delta_P^{\frac12}\sigma)^\vee\simeq \Hom_{H}(\delta_P^{\frac12}\sigma,\C)=0.\] 
The vanishing follows since $r(\delta)\ge 0$. 

This establishes \eqref{eq schwartz homology} in all cases and we conclude that $H_0(H,Q)=0$.
Applying  \cite[Lemma 3.8]{MR4659866} once more we conclude that $\Hom_H(Q,\C)=0$ and the proposition readily follows. 

\end{proof}

\section{Local intertwining periods: preliminaries}
	
Here we go back to the notation of Section \ref{ss sym} for a general symmetric space in the local set-up.
A systematic study of local intertwining periods has been carried out in \cite{MR4679384}. Here we recall and slightly extend some of our results. 

Let $P=MU$ be a parabolic subgroup of $G$. An element $x\in X$ is $P$-\emph{admissible} if $\theta_x(M)=M$. In this case $\theta_x$ acts as an involution on $\fra_{M,\C}^*$ and we denote by $(\fra_{M,\C}^*)_x^\pm$ its $\pm1$-eigenspace.
We say that a $P$-admissible $x$ satisfies the \textit{modulus assumption} if the following holds:
	
	\[\delta_{P_x}=\delta_P^{1/2}|_{M_x}.\]
	
	For a $P$-admissible $x \in G\cdot e$ satisfying the modulus assumption, take $u \in G$ such that $u \theta(u)^{-1} = x$, $\ell \in \Hom_{M_x} (\sigma, \C)$ and $\lambda \in (\fra_{M,\C}^{\ast})_x^- $. The \emph{intertwining period}, attached to $x,\ell,
	\sigma$ and $\lambda$, is a linear form on $I_P^G(\sigma)$ defined by the meromorphic continuation of the integral
	
	\begin{align}\label{formula::defn--intt periods}
		J_P^G(\varphi;x,\ell,\sigma,\lambda) = \int_{ u^{-1}Pu\cap H \backslash H} \ell (\varphi_{\lambda} (uh)) dh = \int_{P_x \backslash G_x} \ell (\varphi_{\lambda}(g u)) dg.
	\end{align}
Note that the definition does not depend on the choice of $u$. It is easy to check that the integral is formally well-defined thanks to the modulus assumption satisfied by $x$. By \cite[Theorem 5.3]{MR4679384}, the above integral is absolutely convergent when $\Re(\lambda)$ is in a certain cone in $(\fra_{M}^{\ast})^-_x$ and admits a meromorphic continuation to $\lambda \in (\fra_{M,\C}^{\ast})^-_x$. For $m\in M$, let $x' = m \cdot x$. Note that $u'\theta(u')^{-1}=x'$ for $u'=mu \in G$ and $\ell \circ \sigma(m) \in \Hom_{M_{x'}} (\sigma,\C)$. Then by definition one has
	\begin{align}\label{formula::intt-period--change-in-orbit}
		J_P^G(\varphi;x,\ell,\sigma,\lambda)  = e^{\langle \lambda + \rho_P, H_M(m) \rangle} J_P^G (\varphi; x',\ell \circ \sigma(m), \sigma,\lambda).
	\end{align}

	Singularities of intertwining periods will play an essential role in our local and global results. We say that $J_P^G(x,\ell,\sigma,\lambda)$ is holomorphic at $\lambda=\lambda_0$ if $J_P^G(\varphi;x,\ell,\sigma,\lambda)$ is holomorphic at $\lambda=\lambda_0$ for every $\varphi \in I_P^G(\sigma)$. Otherwise, we say that $J_P^G(x,\ell,\sigma,\lambda)$ has a singularity at $\lambda=\lambda_0$.
		
We begin with the following simple observation in the non-archimedean case. 
\begin{lemma}\label{lem::holomorphy}
Suppose that $F$ is $p$-adic. Let $\sigma$ be a representation of $M$ so that every $P$-orbit in $G\cdot e$ that is relevant to $\sigma$ is open in $G\cdot e$. Then for any $x\in G\cdot e$ that is $P$-admissible and satisfies the modulus assumption and for $\ell\in \Hom_{M_x} (\sigma, \C)$ we have that $J_P^G(x,\ell,\sigma,\lambda)$ is holomorphic at $\lambda=0$.
\end{lemma}
	\begin{proof}
		For any $\varphi\in I_P^G(\sigma)$ with support contained in the union of open $(P,H)$-double cosets, the integrand in \eqref{formula::defn--intt periods} vanishes outside a compact domain and the integral is therefore absolutely convergent and hence holomorphic at any $\lambda\in (\fra_{M,\C}^{\ast})_x^- $. If $J_P^G(x,\ell,\sigma,\lambda)$ is not holomorphic at $\lambda =0$, then its leading term (along any line through $0$) is an $H$-invariant linear form on $I_P^G(\sigma)$ that vanishes on the $H$-subspace of functions supported on the open $P$-orbits. The assumption now contradicts \eqref{eq geom lem}.
	\end{proof}

	We recall the functional equations satisfied by intertwining periods as well as their compatibility with transitivity of parabolic induction. We introduce a directed, labeled graph $\mathfrak{G}$ as in \cite{MR3541705} which is a close variant of the graph considered in \cite{MR2010737}. The vertices of $\mathfrak{G}$ are the pairs $(M,x)$, where $M$ is a standard Levi subgroup of $G$ and $x \in X$ is $P$-admissible. The edges of $\mathfrak{G}$ are given by
	\begin{align}\label{eq edge}
		(M,x) \stackrel{n}{\searrow} (M_1,x_1) 
	\end{align}
	if there is $\alpha \in \Delta_P$ with $- \alpha \ne \theta_x(\alpha) < 0$ such that $n \in s_{\alpha} M$ where $s_{\alpha} \in W(M)$ is the elementary symmetry associated to $\alpha$, $M_1 = nMn^{-1}$ and $x_1 = n \cdot x$. Note that $(M_1)_{x_1} = n M_x n^{-1}$ and that
	\[\delta_{P_x}\delta_P^{-1/2}(m)=\delta_{(P_1)_{x_1}}\delta_{P_1}^{-1/2}(nmn^{-1})\] for all $m\in M_x$ by \cite[Corollary 6.5]{MR3541705}. In particular the modulus assumption is satisfied by $x$ if and only if it is satisfied by $x_1$, or in other words the modulus assumption is satisfied by one vertex of the graph $\mathfrak{G}$ if and only if it is satsified by its connected component. 
	
There are two types of extreme vertices that we now describe. We say that a vertex $(M,x)$ is minimal if there exists a a standard Levi subgroup $L \supset M$ such that $\iota_P(P\cdot x)= w_M^L$ and $w_M^L(\alpha) = - \alpha $ for all $\alpha \in \Delta_P^Q$ (see  \cite{MR2010737}); on the other hand, we say that a vertex $(M,x)$ is maximal if there exists a standard Levi subgroup $L \supset M$ such that $\iota_P(P\cdot x) = w_L^G$ and $  w_L^G (\alpha) = \alpha$ for all $\alpha \in \Delta_P^Q$ (see \cite{MR4679384}). In both cases, the standard Levi subgroup $L$ is uniquely determined by the vertex. We refer the reader to the references above for a more detailed study of these two notions. 
	\begin{example}
		In case \namelink{gal2}, let $\alpha = (a_1,\cdots,a_r)$, $s \in J(\alpha)$ and $w_s \in G$ be as in section \ref{ss orbits}. Note that, in fact, $w_s\in X$. When $s$ is monomial, it can be naturally viewed as an involution of $\mathfrak{S}_r$ (permuting the blocks of $M_\alpha$), which we denote by $p(s)$. Then $(M,w_s)$ is minimal if and only if $s$ is monomial and $p(s)$ is a product of disjoint transpositions of the form $(j,j+1)$; $(M,w_s)$ is maximal if and only if $s$ is monomial and $p(s)$ is of the form $(1,r)(2,r-1)\cdots(k,r+1-k)$ for some $0 \leqslant k \leqslant r/2$.
	\end{example}

	Now we can state a compatibility property of transitivity of parabolic induction with intertwining periods. 
	
	\begin{proposition}\label{prop::intt-period++PI}
		Let $P = MU$. Let $(M,x)$ be a vertex on the graph $\mathfrak{G}$ such that $x$ satisfies the modulus assumption, and $\sigma$ a finite length representation of $M$. Let $u\in G$ be such that $u \theta(u)^{-1} = x$ and $\ell \in \Hom_{M_x}(\sigma,\C)$.
\begin{enumerate}
		\item\label{part closed fe}Suppose that there is a parabolic subgroup $ Q = LV$ containing $P$ such that $L$ and $P \cap L$ are $\theta_x$-stable and that $\theta_x(Q) = Q^-$, the parabolic subgroup opposite to $Q$. Define $\Lambda_{\ell} \in \Hom_{L_x} (I_{P \cap L}^L \sigma,\C)$ by
		\begin{align}\label{formula::defn--Lambda-ell}
			\Lambda_{\ell} (f) = \int_{(P \cap L)_x \backslash L_x} \ell ( f (h)) dh.
		\end{align}
		Then, for all $\lambda \in (\fra_{L,\C}^{\ast})^-_x$, 
		\begin{align}\label{formula::intt-period++PI-maximal}
			J_P^G(\varphi;x,\ell,\sigma,\lambda) = J_Q^G(F_{\varphi};x,\Lambda_{\ell},I_{P \cap L}^L \sigma,\lambda).
		\end{align}
		In particular, if $(M,x)$ is maximal, such a $Q$ can be taken as the parabolic subgroup with Levi $L$ in the definition of maximality.
		
		\item\label{part open fe} Suppose that there is a parabolic subgroup $Q = LV$ containing $P$ such that $Q$ and $L$ are $\theta_x$-stable and $\theta_x( P \cap L) = (P \cap L)^-$, the parabolic subgroup of $L$ opposite to $ P \cap L$. Then, for all $\lambda \in (\fra_{M,\C}^{\ast})^-_x$,
		\begin{align}\label{formula::intt-period++PI-minimal-I}
			J_P^G (\varphi;x,\ell,\sigma,\lambda) = \int_{Q_x \backslash G_x} \int_{M_x \backslash L_x} \ell (((I_P^G(gu,\sigma,\lambda)\varphi)[e])_{\lambda}(l)) dl dg.
		\end{align}
		In particular, if $(M,x)$ is minimal, such a $Q$ can be taken as the parabolic subgroup with Levi $L$ in the definition of minimality.
\end{enumerate}	\end{proposition}
	\begin{proof}
		When $(M,x)$ is a maximal vertex, the proof of \eqref{formula::intt-period++PI-maximal} is given in the proof of \cite[Theorem 5.3]{MR4679384}. The proof can be carried over without modification to the general situation. When $(M,x)$ is minimal, \eqref{formula::intt-period++PI-minimal-I} is proved in the proof of \cite[Corollary 1]{MR4508342}\footnote{The paper is in the p-adic setup but the proof carries verbatim to the archimedean case}. The proof also carries over to the general situation.
	\end{proof}
	
	By integration in stages we have the following functional equation which relates the intertwining periods attached to two adjacent vertices in the graph $\mathfrak{G}$. 
	
	\begin{proposition}\label{prop::intt-period++FE+Adjacent}
		Let $P = M\ltimes U$ and $P_1 = M_1\ltimes U_1$ be two parabolic subgroups of $G$. Assume that $(M,x) \stackrel{n}{\searrow} (M_1,x_1)$ is an edge on the graph $\mathfrak{G}$ and $\alpha \in \Delta_M$ is such that $n \in s_{\alpha}M$. Let $\sigma$ be a representation of $M$ and $\ell \in \Hom_{M_x}(\sigma,\delta_x)$. Then for $\varphi \in I_P^G(\sigma)$ and $\lambda \in (\fra_{M,\C}^{\ast})^-_x$ we have
		\begin{align}
			J_P^G(\varphi; x, \ell,\sigma,\lambda)  =  J_{P_1}^G ( M(n,\sigma,\lambda)\varphi;x_1, \ell,s_{\alpha}\sigma,s_{\alpha}\lambda).
		\end{align} 
	\end{proposition}
	\begin{proof}
		This is proved in \cite[Theorem 5.3]{MR4679384}.
	\end{proof}
	
\section{Singularities of local intertwining periods}\label{sec local intper}

Let $F$ be a local field. Let $(G',H')=(G_m(F),H_m(F))$ where $(G_m,H_m,\theta_m)_\x$ is defined by one of the cases $\x\in \{\namelink{lin}, \namelink{twlin1},\namelink{twlin2},\namelink{gal1},\namelink{gal2},\namelink{gp}\}$. We continue to write $\theta$ for $\theta_k$ for all $k\in \N$. In this section we double the setup by letting $(G,H)=(G_{2m}(F),H_{2m}(F))$ and study intertwining periods with respect to the parabolic $P=MU$ of $G$ with Levi part $M=G'\times G'$ and choice of $x\in G\cdot e$ such that $\iota_P(P\cdot x)=w_M^G$. 

Let $a$ and $\D$ be defined by \eqref{eq a&D} after Remark \ref{rem arithmetic vs geometric}, so that $G'=\GL_a(\D)$, $G=\GL_{2a}(\D)$ and $M$ is the Levi subgroup of $G$ of type $(a,a)$. 
Let 
\[
w=\begin{pmatrix} & I_a \\ I_a & \end{pmatrix}\in G
\]
represent $w_M^G$ (in case \namelink{gp} $w\in H=\GL_a(\D)$ is embedded diagonally in $G=H\times H$). It is a simple computation that in all cases $w\in G\cdot e$ and $\theta(w)=w$. Furthermore, we explicate the stabilizer
\begin{equation}\label{eq mw}
P_w=M_w=\left\{\diag(g,\theta(g)):g\in G'\right\}.
\end{equation}
In particular, an irreducible representation $\sigma$ of $M$ is $M_w$-distinguished if and only if $\sigma=\sigma_1\otimes \sigma_2$ where $\sigma_1$ is an irreducible representation of $G'$ and $\sigma_2\simeq\sigma_1^*$ and in this case $\Hom_{M_w}(\sigma,\C)$ is one dimensional.

For a representation $\sigma$ of $M$ that is $M_w$-distinguished and $\ell\in \Hom_{M_w}(\sigma,\C)$ we study in this chapter the intertwining period $J_P^G(w,\ell,\sigma,\lambda)$, $\lambda\in (\fra_{M,\C}^*)_w^-$. Since $(\fra_{M,\C}^*)_w^-$ is one dimensional, it will be more convenient to identify it with $\C$. Let $\varpi\in (\fra_M^*)_w^-$ be such that
\[
e^{\sprod{\varpi}{H_M(\diag(g_1,g_2))}}=\nu(g_1g_2^{-1}),\ \ \ g_1,g_2\in G'
\]
and identify $\C$ with $(\fra_{M,\C}^*)_w^-$ via $s\mapsto s\varpi$. Throughout this section we often write $s$ for $s\varpi$ so that $J_P^G(w,\ell,\sigma,s)=J_P^G(w,\ell,\sigma,s\varpi)$. Furthermore, when $\sigma$ is irreducible, we choose once and for all a non-zero $\ell_\sigma$ in the one dimensional space $\Hom_{M_w}(\sigma,\C)$ and write
\[
\J_\sigma(s)=J_P^G(w,\ell_\sigma,\sigma,s).
\]
By varying $\ell_\sigma$ we only rescale $\J_\sigma(s)$ by a non-zero scalar. Unless otherwise specified, our results will be independent of the choice of $\ell_\sigma$.
It is straightforward that if $\ell_{\sigma[t]}=\ell_\sigma$ then
\begin{equation}\label{eq shiftJ}
\J_{\sigma[t]}(\varphi_t,s)=\J_\sigma(\varphi,s+t),\ s,t\in \C.
\end{equation}

Recall that by Theorem \ref{thm mult 1 and ssduality}, for an irreducible, distinguished representation $\pi$ of $G'$ we have $\pi\simeq \pi^*$ and therefore $\J_{\pi\otimes \pi}(s)$ makes sense.
We point out a useful observation.  
	It relies on the explication of stabilizers $M_w$ for the open $P$-orbit $P\cdot w$ (see \eqref{eq mw}) and 
\[
M_e=H' \times H'
\]	
for the closed $P$-orbit $P\cdot e$.
	\begin{lemma}\label{lem::intt-period-pole}
		Let $\pi$ be an irreducible, distinguished representation of $G'$ such that $\pi\times \pi$ is an irreducible representation of $G$.
Then $\J_{\pi\otimes \pi}(s)$ has a pole at $s=0$. Furthermore, let $k=\Ord_{s=0}(\J_\sigma(s))\in \N$. Then there exists $0\ne L\in \Hom_{M_e} (\pi \otimes \pi, \C)$ such that
		\begin{align*}
			\lim_{s \ra 0} s^k \J_{\pi\otimes \pi}(s) = J_P^G(e,L,\pi \otimes \pi,0).
		\end{align*}
	\end{lemma}
	\begin{proof}
		In case \namelink{gal2} when $F$ is $p$-adic, this is \cite[Proposition 10.9]{MR4308058}. The same proof holds in all cases and we recall it for convenience.
It follows from Theorem \ref{thm mult 1 and ssduality} that $\Hom_H(\pi\times \pi,\C)$ is one-dimensional. Furthermore, the closed orbit intertwining period $J_P^G(e,L,\pi \otimes \pi,0)$ is a non- zero element of $\Hom_H(\pi\times \pi,\C)$ for $0\ne L\in \Hom_{M_e} (\pi \otimes \pi, \C)$ that vanishes on sections supported on the unique open $(P,H)$-double coset in $G$. Since $\J_{\pi\otimes \pi}(s)$ restricted to the $H$-invariant subspace of $I_P^G(\pi\otimes\pi,s)$ of such sections is holomorphic and non-zero we conclude that $\J_{\pi\otimes \pi}(s)$ has a pole at $s=0$. Since the leading term $\lim_{s \ra 0} s^k \J_{\pi\otimes \pi}(s)$ is a non-zero element of $\Hom_H(\pi\times \pi,\C)$ the lemma follows. 
	\end{proof}

In the rest of this section we study an explicit functional equation satisfied by the linear form $\J_\sigma(s)$ as well as the order of its pole at $s=0$.
For the sake of some of our arguments, we consider $\sigma$ that is parabolically induced and similar intertwining periods for the inducing data. For this purpose it is more convenient to choose a different representative $x\in M\cdot w$.

Let 
\begin{equation}\label{eq def good x}
x=u\cdot w=\begin{pmatrix} & \gamma^{-1} \\ \gamma & \end{pmatrix},\ \ \ u=\diag(I_a,\gamma)\in M
\end{equation}
where we set $\gamma=I_a$ in cases \namelink{lin}, \namelink{twlin1}, \namelink{gal2} and \namelink{gp}, $\gamma=[\upsilon^\circ]_m$ in case \namelink{twlin2} and $\gamma=\epsilon[\upsilon^\circ]_m$ in case \namelink{gal1} (for $\upsilon^\circ$ see Section \ref{sss the families}). We observe that
\begin{equation}\label{eq mwx}
P_x=M_x=\left\{\diag(g,\iota(g)):g\in G'\right\}
\end{equation}
(for $\iota$ see Section \ref{sss auxi}).
Let $L'$ be a Levi subgroup of $G'$ of type $(a_1,\dots,a_k)$ and $L=L'\times L'$ the corresponding Levi subgroup of $G$. Then $\theta_x(L)=L$ and 
\[
L_x=\{\diag(g_1,\dots,g_k,\iota(g_1),\dots,\iota(g_k)): g_i\in \GL_{a_i}(\D),\ i\in [1,k]\}.
\]
An irreducible representation of $L$ is $L_x$-distinguished if and only if it has the form 
\[
\sigma_1\otimes \cdots\otimes \sigma_k\otimes \sigma_1^*\otimes \cdots\otimes \sigma_k^*
\]
where $\sigma_i$ is an irreducible representation of $\GL_{a_i}(\D)$, $i\in [1,k]$. In this case $\Hom_{L_x}(\sigma,\C)$ is one dimensional.

\subsection{Explicit functional equations}
It is often the case that intertwining periods satisfy functional equations that are not accounted for by Proposition \ref{prop::intt-period++FE+Adjacent}.
This is the case at hand for $\J_\sigma(s)$. In fact, we need the functional equation in a slightly more general set-up. Let $Q'=L'V'$ be a parabolic subgroup of $G'$ of type $(a_1,\dots,a_k)$  and $Q=LV$ be the parabolic of $G$ with Levi $L=L'\times L'$. Note that $(\fra_{L,\C}^*)^-_x=\{(\lambda,-\lambda):\lambda\in \fra_{L',\C}^*\}$.	
	\begin{proposition}\label{prop::FE--MultiOne}
		Let $\sigma$ be an irreducible representation of $L$ that is $L_x$-distinguished. Then there exists a meromorphic function $\alpha_\sigma'(\lambda)$ on $(\fra_{L,\C}^*)^-_x$ such that for $\ell\in \Hom_{L_x}(\sigma,\C)$ we have
\[
\alpha_\sigma'(\lambda)J_Q^G(x,\ell,\sigma,\lambda)=J_Q^G(x,\ell,w\sigma,-\lambda)\circ M(w,\sigma,\lambda).
\]
In particular, when $L'=G'$ (i.e. $Q=P$) there is a meromorphic function $\alpha_\sigma(s)\underset{\C^\times}{\sim} \alpha_\sigma'(s)$ such that
\[
\alpha_\sigma(s)\J_\sigma(s)=\J_{w\sigma}(-s)\circ M(w,\sigma,s).
\]

	\end{proposition}
\begin{proof}
The representation $I_Q^G(\sigma,\lambda)$ is irreducible for a generic $\lambda$ and distinguished for all $\lambda$ in $(\fra_{L,\C}^*)^-_x$ by \cite[Theorem 5.4]{MR4679384}. Since both sides are $H$-invariant linear forms on $I_Q^G(\sigma,\lambda)$ for a generic $\lambda$ the functional eqution follows from Theorem \ref{thm mult 1 and ssduality}. The last part of the proposition further applies \eqref{formula::intt-period--change-in-orbit}.
\end{proof}


The main goal of this subsection is to explicitly relate the proportionality factor $\alpha_\sigma(s)$ to local $L$-factors when $Q=P$. The argument is global. 
First, we generalize \cite[Lemma 4.1]{MR4308058} on globalization of characters.

\begin{lemma}\label{lm globalizing char}
Let $k$ be a global field with adele ring $\BA_k$, and let $S$ be a finite set of places of $k$. We set $k_S:=\prod_{v\in S} k_v$, where $k_v$ is the completion of $k$ at $v$. For any character $\chi_S$ of $k_S^\times$ there exists an automorphic character $\mu=\prod_v \mu_v$ of $\BA_k^\times$ such that $\mu_S^{-1}\chi_S$ is unramified (i.e. trivial on the maximal compact subgroup of $k_S^\times$) where $\mu_S=\prod_{v\in S}\mu_v$. 
\end{lemma}	
\begin{proof}
Denote by $k_S^0$ the maximal compact subgroup of $k_S^\times$. Then the natural map from $k_S^0$ to $\BA_k^\times/k^\times$ identifies 
 $k_S^0$ with a compact subgroup of $\BA_k^\times/k^\times$, hence $(\chi_S)_{|k_S^0}$ extends to a character $\mu$ of $\BA_k^\times/k^\times$ by Pontryagin duality, and the result follows. 
\end{proof}

As in \cite{MR4308058}, the unramified formula for local open intertwining periods plays a crucial role in the argument. Luckily it has already been proved in all cases that we consider. 

\begin{proposition}\label{prop ur comp}
Let $\D=F$ be a local field of characteristic zero, and $E/F$ be an $F$-\'Etale algebra of dimension $2$. We allow $E$ to be a field only when $F$ is $p$-adic, in which case $E/F$ is assumed to be unramified. Hence when $F$ is archimedean, we are in cases \namelink{lin} or \namelink{gp}, and when $F$ is non-archimedean we are in cases \namelink{lin}, \namelink{gal2}, \namelink{twlin2} or \namelink{gp}. In case \namelink{twlin2} assume further that $\delta\in \O_E^\times$. Let $\pi$ be an irreducible, generic, unramified representation of $G_a(F)$, 
and set $\sigma:=\pi \otimes \pi^*$. Then, for $\varphi_0$ the normalized spherical function in $\pi\times \pi^*$, we have  
\[\J_{\sigma}(\varphi_0,s)=\frac{\L(s,\pi,\theta)}{\DL(s,\pi,\theta)}.\]
\end{proposition}
\begin{proof}
In view of the compatibility of intertwining periods with transitivity of parabolic induction (\cite[Proposition 3.7]{LuMatringe}), the statement has been proved in \cite[Theorem 36]{MR1625060}, \cite[8.8.3]{MR4308058} and \cite[8.8.2]{MR4308058} in cases \namelink{gal2} and \namelink{gp}, and in \cite[Proposition 4.5]{MR4721777} in cases \namelink{lin} and \namelink{twlin2} (\cite[Proposition 4.5]{MR4721777} is actually a translation of the results in \cite{MR2060496} and \cite{MR3776281}). 
\end{proof}

The Jacquet-Langlands correspondence extends locally and globally in the most obvious manner to products of general linear groups over division algebras. We need the following result on globalization of disrcete series, slightly generalizing \cite[Corollary 4.1]{MR4308058}.

\begin{lemma}\label{lm globalizing disc}
Let $F$ be a local field of characteristic zero, and let $\mathcal{D}$ be a finite dimensional division algebra with center $F$. Let $a_1,\dots,a_r$ be positive integers, and let $\delta$ be an irreducible, essentially square-integrable representation of 
$L':=\GL_{a_1}(\mathcal{D})\times \dots \times \GL_{a_r}(\mathcal{D}).$ Let $k$ be a number field with adele ring $\BA_k$, and let $v_0$ be a place of $k$ such that $k_{v_0}=F$. Finally let $D$ be a division algebra with center $k$ such that $D_{v_0}=\mathcal{D}$ and split at all archimedean places different from $v_0$. Then there exists an irreducible, cuspidal automorphic representation $\Delta$ of 
$L'_{\BA_k}:=\GL_{a_1}(D_{\BA_k})\times \dots \times \GL_{a_r}(D_{\BA_k})$ such that:
\begin{enumerate}
\item $\Delta_{v_0}=\delta$.
\item $\Delta_v$ is an unramified generic principal series for all archimedean places $v$ of $k$ such that $v\neq v_0$.
\item $\JL(\Delta)$ is cuspidal. 
\end{enumerate}
In particular, $\JL(\Delta)_{v}=\JL(\Delta_{v})$ for any place $v$ of $k$, hence $\JL(\Delta)_{v_0}=\JL(\delta)$. 
\end{lemma}	
\begin{proof}
We will use the results of \cite{MR3004076}, as in \cite[Section 4]{MR4308058}. Let $d$ be the square root of $[\mathcal{D}:F]$. Since Shin's work is written for semi-simple groups, we restrict $\delta_0:=\JL(\delta)$ to the derived subgroup \[L^1:=\SL_{da_1}(F)\times \dots \times \SL_{da_r}(F)\] of \[L:=\GL_{d a_1}(F)\times \dots \times \GL_{d a_r}(F)\] and pick $\delta_0^1$ an irreducible component of this restriction. For a finite set of places $S$ of $F$ let
\[L_S^1=\SL_{da_1}(k_S)\times \dots \times \SL_{da_r}(k_S)\ \ \ \text{where} \ \ \ k_{S}=\prod_{v\in S}k_v.\]
Let $S_1$ be the set of archimedean places $v$ of $k$ such that $v\neq v_0$, and let $S_2$ be the set of finite places $v$ of $k$ such that $v\neq v_0$ and $D_v$ is split. In particular the sets $S_1$, $S_2$ and $\{v_0\}$ are disjoint by our assumption that $D_v$ is split whenever $v\neq v_0$ is archimedean. Set $S:=S_1\sqcup S_2$. 
First we fix an irreducible, square-integrable representation $\delta_{S_2}^1$ of $L_{S_2}^1$.
Now we denote by $B_{S_1}^1$ the upper triangular Borel subgroup of $L_{S_1}^1$ and by $T_{S_1}^1$ its diagonal torus. 
Then we fix an irreducible, tempered, unramified representation $\tau^1_{S}$ of $L_S^1$ and write it under the form 
\[\tau^1_{S_1}=I_{B_{S_1}^1}^{L^1_{S_1}}(\mu_{S_1}),\] where $\mu_{S_1}$ is a unitary unramified character of $T_{S_1}^1$. Now, as explained in the proof of \cite[Proposition 4.1]{MR4308058} (see in particular the discussion of the assumptions in \cite[Section 4]{MR3004076} there), one can apply \cite[Theorem 5.13]{MR3004076} to claim that for any non empty open and bounded subset 
\[O_{S_1}\subseteq \{(z_{v,k,l_k})_{v\in S_1,\ k=1,\dots,r, \ l_k=1,\dots,a_k} \in (\prod_{k=1}^r i{\BR}^{a_k})^{|S_1|},\ \forall v\in S_1, \ \forall k=1,\dots, r,\ \sum_{l_k=1}^{a_k} 
z_{v,k,l_k}=0 \},\] 
there exists $\underline{u_0}\in O_{S_1}$ and an irreducible cuspidal automorphic representation $\Delta_{0,\underline{u_0}}^1$ of 
\[L_{\BA_k}^1=\SL_{da_1}(\BA_k)\times \dots \times \SL_{da_r}(\BA_k),\] such that \[\Delta^1_{0,\underline{u_0},v_0} =\delta_0^1,\] 
\[\Delta^1_{0,\underline{u_0},S_2}=\delta_{S_2}^1,\] and such that \[\Delta^1_{0,\underline{u_0},S_1}=\tau^1_{S_1}[\underline{u_0}].\] Precisely, the set $S$ in \cite[Theorem 5.13]{MR3004076}, with its notation, is our set $S\cup \{v_0\}$, we take $\widehat{U}$ in \cite[Theorem 5.13]{MR3004076} to be $\{\tau^1_{S_1}[\underline{u}]\otimes \delta_{S_2}^1 \otimes \delta_0^1,\ \underline{u}\in \overline{O_{S_1}}\}$, and we take $v_1$ and $v_2$ in \cite[Theorem 5.13]{MR3004076} to be two random places outside of $S\cup \{v_0\}$.

We now fix such a pair $(O_{S_1},\underline{u_0})$, and set $\Delta_0^1:=\Delta^1_{0,\underline{u_0}}$. By \cite[Chapter 4]{MR2918491}, the cuspidal automorphic representation $\Delta_0^1$ occurs in the restriction of an irreducible cuspidal automorphic representation $\Delta_0$ of 
$L_{\BA_k}=\GL_{da_1}(\BA_k)\times \dots \times \GL_{da_r}(\BA_k).$

Note that $\Delta_{0,S \cup  \{v_0\}}:=\otimes_{v\in S\cup \{v_0\}}\Delta_{0,v}$ contains $\tau_{0,S_1}^1[\underline{u_0}]\otimes \delta_{S_2}^1\otimes \delta_0^1$ in its restriction to $L_S^1$, where the tensor product is taken to be the completed one between archimedean representations. Observe as well that $\tau_{0,S_1}^1[\underline{u_0}]$ obviously extends to a tempered unramified representation $\tau_{S_1}$ of $L_{S_1}=\GL_{da_1}(k_{S_1})\times \dots \times \GL_{da_r}(k_{S_1})$ (just extend the inducing character $\mu_{S_1}$ to an unramified unitary character of the diagonal torus of $L_{S_1}$), and that 
$\delta_{S_2}^1$ extends as well to an essentially square-integrable representation $\delta_{S_2}$ of $L_{S_2}$. Hence we deduce from \cite[Chapter 2]{MR2918491} that $\Delta_{0,S\cup \{v_0\}}$ is of the form 
$\chi_{S\cup \{v_0\}}\otimes (\tau_{S_1}\otimes \delta_{S_2} \otimes \delta_0)$ for $\chi_{S\cup \{v_0\}}$ a character of $k_{S\cup \{v_0\}}$. Using Lemma \ref{lm globalizing char} with $S$ there being $S\cup \{v_0\}$ here, we deduce that up to twisting $\Delta_0$ by an automorphic character, we may assume that $\Delta_{0,S_1}$ is generic unramified, $\Delta_{0,S_2}$ is essentially square-integrable, and that $\Delta_{0,v_0}=\mu_0\otimes \delta_0$ for $\mu_0$ an unramified character. Because local unramified characters at one place obviously extend to global automorphic unramified characters, we infer that we may actually moreover assume, replacing $\Delta_0$ with an unramified twist if necessary, that $\Delta_{0,v_0}= \delta_0$. Finally, as recalled in \cite[Proof of Corollary 4.1]{MR4308058}, because $\Delta_{0,v}$ is essentially-square integrable at any place $v$ of $k$ such that $D_v$ does not split, it follows from the results of \cite{MR2390289} and \cite{MR2684298}, that $\Delta_0$ is automatically of the form $\JL(\Delta)$ for $\Delta$ a cuspidal automorphic representation of $L_{\BA_k}$. This representation $\Delta$ satisfies all the required properties. In particular, because $\JL(\Delta)$ is cuspidal, we have $\JL(\Delta)_w=\JL(\Delta_w)$ by  \cite{MR2390289} and \cite{MR2684298} again (see 
 \cite[Section 4]{MR4308058} for more details). 
\end{proof}

The next result is inspired by \cite[Theorem 9.2]{MR4308058} and its proof is very similar, it also generalizes \cite[Proposition 4.7]{MR4721777}. It uses the factorization of global intertwining periods into local ones, and we refer to Section \ref{sec::LGp} below for a detailed discussion of this fact. 

\begin{theorem}\label{thm::constant-FE}
	Suppose that $F$ is $p$-adic with residual characteristic $p$. Let $\sigma = \sigma_1 \otimes \sigma_2$ be an irreducible, generic and $M_w$-distinguished representation of $M$ (so that $\sigma_2\simeq \sigma_1^*$). Then 
	\begin{enumerate}
	\item	in cases \namelink{gal1}, \namelink{gal2}:
	\[\alpha_{\sigma}(s)  \underset{\C[p^{\pm s}]^\times}{\sim} \gamma_0(-2s,\jl(\sigma_1),\as^+)^{-1} \gamma_0(2s,\jl(\sigma_1),\as^-)^{-1},\] 
	\item in cases \namelink{lin}, \namelink{twlin1}, \namelink{twlin2}:
	\[\alpha_{\sigma}(s)  \underset{\C[p^{\pm s}]^\times}{\sim} \frac{\gamma_0(s+1/2,\jl(\sigma_1))\gamma_0(s+1/2,\eta_0\otimes\jl(\sigma_1))}{\gamma_0(-2s,\jl(\sigma_1),\wedge^2) \gamma_0(2s,\jl(\sigma_1),\Sym^2)}\]
	where $\eta_0$ is trivial in case \namelink{lin} and  $\eta_0=\eta_{E/F}$ in cases \namelink{twlin1} and \namelink{twlin2}.
	  
	\end{enumerate}
	If moreover $\sigma_1$ is assumed to be square-integrable and to satisfy $\sigma_1=\sigma_1^*$, and $\rho$ is the cuspidal representation and $t$ is the integer such that \[\jl(\sigma_1)=\st_t(\rho),\] then:
	\begin{enumerate}
	\item	in cases \namelink{lin}, \namelink{twlin1}, \namelink{twlin2}:
	\[\alpha_{\sigma}(s) \underset{\C[p^{\pm s}]^\times}{\sim}  \frac{L(-s+\frac{t}{2},\rho)L(-s+\frac{t}{2},\eta_0\otimes\rho)}{L(s+\frac{t}{2},\rho)L(s+\frac{t}{2},\eta_0\otimes \rho)}\times \frac{L(-2s,\rho,\Sym^2)}{L(-2s + t, \rho,\Sym^2 )} \frac{L(2s,\rho,\wedge^2)}{L(2s+t,\rho,\wedge^2)},\] 
	\item in cases \namelink{gal1}, \namelink{gal2}:
	\[\alpha_{\sigma}(s)\underset{\C[p^{\pm s}]^\times}{\sim}  \frac{L(-2s,  \rho, \as^-)}{L(-2s+t,\rho,\as^-)}\frac{L(2s, \rho,\as^+)}{L(2s + t, \rho,\as^+)}. \]
	\end{enumerate}
	\end{theorem}
	\begin{proof}
The second part of the theorem follows from the first and an explication of the appropriate $L$-factors of $\St_t(\rho)$ in terms of $L$-factors of $\rho$. 
In the Galois cases \namelink{gal1} and \namelink{gal2} this is carried out in \cite[Proposition 6.3]{MR4308058}. For the other cases \namelink{lin}, \namelink{twlin1} and \namelink{twlin2} we apply \cite[Lemma 4.8]{MR4721777} together with more familiar formulas for standard $L$-factors.

The first part is proved using the globalization of Lemma \ref{lm globalizing disc} together with  the functional equation of global intertwining periods and the known unramified formula for the local intertwining periods. 

In the Galois cases \namelink{gal1} and \namelink{gal2}, we put \[\beta_{\sigma}(s)=\gamma_0(-2s,\jl(\sigma_1),\as^+)^{-1} \gamma_0(2s,\jl(\sigma_1),\as^-)^{-1},\]  whereas in the other cases \namelink{lin}, \namelink{twlin1} and \namelink{twlin2} we put \[\beta_{\sigma}(s)=\frac{\gamma_0(s+1/2,\jl(\sigma_1))\gamma_0(s+1/2,\eta_0\otimes\jl(\sigma_1))}{\gamma_0(-2s,\jl(\sigma_1),\wedge^2) \gamma_0(2s,\jl(\sigma_1),\Sym^2)}.\]

The representation $\sigma_1$ has the form \[\sigma_1=\delta_1 \times \dots \times \delta_r ,\] where $\delta_i$ is essentially square-integrable $i=1,\dots,r$. We write \[\delta=\delta_1\otimes \dots \otimes \delta_r,\] it is a representation of some Levi subgroup \[L'=\GL_{a_1}(\mathcal{D})\times \dots \times \GL_{a_r}(\mathcal{D})\] of $G'=\GL_a(\mathcal{D})$. Let us set 
\[\beta_{\delta\otimes \delta^*}(s)=\]
\[ \frac{\prod_{i=1}^r\gamma_0(-2s,\jl(\delta_i),\as^+)^{-1} \gamma_0(2s,\jl(\delta_i),\as^-)^{-1}}{\prod_{1\leq j <k\leq r} 
\gamma_0(-2s,\jl(\delta_j),\jl(\delta_k)^\theta) \gamma_0(2s,\jl(\delta_j),\jl(\delta_k)^\theta)}\] in cases 
\namelink{gal1} and \namelink{gal2} and 
\[\beta_{\delta\otimes \delta^*}(s)=\] 
\[\frac{\prod_{i=1}^r \gamma_0(s+1/2,\jl(\delta_i))\gamma_0(s+1/2,\eta_0\otimes\jl(\delta_i))}{
\prod_{j=1}^r \gamma_0(-2s,\jl(\delta_j),\wedge^2) \gamma_0(2s,\jl(\delta_j),\Sym^2) \prod_{1\leq k<l\leq r} 
\gamma_0(-2s,\jl(\delta_k),\jl(\delta_l)) \gamma_0(2s,\jl(\delta_k),\jl(\delta_l))}\]
in cases \namelink{lin}, \namelink{twlin1} and \namelink{twlin2}. 

By multiplicativity of gamma factors we have 
\begin{equation}\label{eq mult beta} \beta_{\sigma}(s)\underset{\C[p^{\pm s}]^\times}{\sim} \beta_{\delta\otimes \delta^*}(s).\end{equation}
Let $Q'$ be the standard parabolic subgroup of $G'$ with corresponding Levi component $L'$ and $Q=LV$ be the parabolic of $G$ with Levi $L=L'\times L'$. We observe that the $M_x$-invariant linear form $\ell=\ell_\sigma\circ \sigma(u)$ on $\sigma$ is induced from a unique $\L\in \Hom_{L_x}(\delta\otimes \delta^*,\C)$ in the following sense: 
\[\ell(f_1\otimes f_2)=\int_{Q'\backslash G'} \L(f_1(g')\otimes f_2(\iota(g')) dg'.\]
Applying Proposition \ref{prop::intt-period++PI} \eqref{part closed fe}, the proportionality functions of Proposition  \ref{prop::FE--MultiOne} for $\sigma$ and for $\delta\otimes \delta^*$ satisfy 
\[\alpha_{\sigma}'(s) =\alpha_{\delta\otimes \delta^*}'(s).\] 
We recall that by the same proposition 
\begin{equation}\label{eq alpha alpha}  
\alpha_\sigma(s)\underset{\C^\times}{\sim} \alpha_\sigma'(s).
\end{equation}

For the details of the following globalization process for a quadratic extension and central simple algebras, we refer to \cite[Section 9.2]{MR4308058}. First we choose a number field $k$ which has a unique place $v_0$ lying over $p$, and such that $k_{v_0}=F$. Then in cases \namelink{twlin1}, \namelink{gal2}, \namelink{twlin2} and \namelink{gal1}, we choose a quadratic extension $l/k$ that remains inert over $v_0$ and is split at infinity, and such that if $w_0$ is the place of $l$ lying above $v_0$, then $l_{w_0}/k_{v_0}=E/F$. Finally we choose a global division algebra $D$ with center $k$, such that $D_{v_0}=\mathcal{D}$ and $D$ is split at infinity. 

According to Lemma \ref{lm globalizing disc} and in its notation, there exists an irredeucible, cuspidal automorphic representation $\Delta$ of $L'_{\BA_k}$ satisfying all the requirements of Lemma \ref{lm globalizing disc}, and in particular, such that $\Delta_{v_0}=\delta$. 

The end of the proof is the same as in \cite[Theorem 9.2]{MR4308058}. Let $S$ be a finite set of finite places. Let $Q$ be the standard parabolic subgroup of $G :=\GL_{2a}(D_{\A_k})$ with standard Levi subgroup $L'_{\A_k}\times L'_{\A_k}$. For $\varphi_s$ a decomposable holomorphic section of 
$I_Q^G(\Delta\otimes \Delta^*,s)$, we write \[\J_{\Delta\otimes \Delta^*}^S(s,\varphi)=\prod_{v\notin S}\J_{\Delta_v\otimes \Delta_v^*}(s,\varphi_v),\] and \[M^S(w,{\Delta\otimes \Delta^*},s)=\otimes'_{v\notin S} M (w,{\Delta_v\otimes \Delta_v^*},s).\] We also write 
\[\J_{\Delta\otimes \Delta^*,S}(s,\varphi)=\prod_{v \in S}\J_{\Delta_v\otimes \Delta_v^*}(s,\varphi_v),\] and \[M_S(w,\Delta\otimes \Delta^*,s)=\otimes_{v\in S} M(w,\Delta_v\otimes \Delta_v^*,s).\] We define $\alpha_{\Delta\otimes \Delta^*,S}(s)$ and $\beta_{\Delta\otimes \Delta^*,S}(s)$ similarly. 

Recall that $\Delta$ is an unramified (tempered) principal series at infinity. Consider $S$ large enough to contain $v_0$ and such that $D$ splits and $\Delta$ is unramified at every $v\not\in S$. We now restrict to sections $\varphi_s$ such that $\varphi_s^S$ is the normalized spherical section. By Proposition \ref{prop ur comp}, we also recall that 
\[\J_{\Delta\otimes \Delta^*}^S(s,\varphi)=\frac{\L^S(s,\Delta,\theta)}{\DL^S(s,\Delta,\theta)}.\] 

Now, enlarging $S$ if needed, and using the functional equation of the global partial $L$-functions at stake given by Corollary \ref{cor gfe LS}, together with the Gindikin-Karpelevich formula, we deduce as in the proof of \cite[(8) p.47]{MR4308058} that 
\begin{equation}\label{eq eq} \frac{\J_{\Delta\otimes \Delta^*}^S(s,\varphi)}{\J_{\Delta^*\otimes \Delta}^S(-s,M^S(w,{\Delta\otimes \Delta^*},s)\varphi)}\underset{\C[p_1^{\pm s},\dots,p_l^{\pm s}]^\times}{\sim} \beta_{\Delta\otimes \Delta^*,S}(s),\end{equation} where $\{p_1,\dots,p_l\}$ is the finite set of prime numbers lying under the places in $S$. In particular, $p\in\{p_1,\dots,p_r\}$.

Then, by Corollary \ref{cor:FEof GIP} below\footnote{Section \ref{sec::Global} is independent of the results of this section}, which is the functional equation of global intertwining periods, we deduce that 
\[\frac{\J_{\Delta\otimes \Delta^*}^S(s,\varphi)}{\J_{\Delta^*\otimes \Delta}^S(-s,M^S(w,{\Delta\otimes \Delta^*},s)\varphi)}= 
\frac{\J_{\Delta^*\otimes \Delta,S}(-s,M_S(w,{\Delta\otimes \Delta^*},s)\varphi)}{\J_{\Delta\otimes \Delta^*,S}(s,\varphi)},\]
so that by Equation \eqref{eq eq}:
\begin{equation}\label{eq eq eq} \alpha_{\Delta\otimes \Delta^*,S}(s)=\frac{\J_{\Delta^*\otimes \Delta,S}(-s,M_S(w,{\Delta\otimes \Delta^*},s)\varphi)}{\J_{\Delta\otimes \Delta^*,S}(s,\varphi)}\underset{\C[p_1^{\pm s},\dots,p_l^{\pm s}]^\times}{\sim} \beta_{\Delta\otimes \Delta^*,S}(s).\end{equation} 
We finally determine the $p$-part of this identity by applying \cite[Lemma 9.3]{MR4308058}:

\begin{equation}\label{eq alpha beta}
\alpha_{\delta\otimes \delta^*}'(s)\underset{\C[p^{\pm s}]^\times}{\sim} \beta_{\delta\otimes \delta^*}(s)
\end{equation} 

In view of Equations \eqref{eq mult beta} and \eqref{eq alpha alpha}, we conclude that 
\begin{equation}\label{eq final eq} \alpha_{\sigma}(s)\underset{\C^\times}{\sim}\alpha_{\delta\otimes \delta^*}'(s)\underset{\C[p^{\pm s}]^\times}{\sim} \beta_{\delta\otimes \delta^*}(s)\underset{\C[p^{\pm s}]^\times}{\sim} \beta_{\sigma}(s).\end{equation}
	\end{proof}
	
\begin{remark}\label{rem more precise FE}
Let $\psi$ be a non trivial additive character of $F$. In the proof above, if we replace $\gamma_0(s,\pi,\star)$ by the Shahidi gamma factors $\gamma(s,\pi,\star,\psi)$ as in the discussion before Theorem \ref{thm mult rel L fct}, one obtains a more precise statement of Theorem \ref{thm::constant-FE}. Indeed in Equation \eqref{eq mult beta}, we can then replace $\underset{\C[p^{\pm s}]}{\sim}$ by $\underset{\C^\times}{\sim}$. Then in Equation \eqref{eq eq}, hence in Equation \eqref{eq eq eq}, we can replace $\underset{\C[p_1^{\pm s},\dots,p_l^{\pm s}]}{\sim}$ by an equality, i.e. Equation \eqref{eq eq eq} becomes \[\alpha_{\Delta\otimes \Delta^*,S}(s)=\beta_{\Delta\otimes \Delta^*,S}(s).\] Applying \cite[Lemma 9.3]{MR4308058}, we now obtain $\alpha_{\delta\otimes \delta^*}'(s)\underset{\C^\times}{\sim} \beta_{\delta\otimes \delta^*}(s)$ instead of Equation \eqref{eq alpha beta}. All in all Equation \eqref{eq final eq} becomes $\alpha_{\sigma}(s) \underset{\C^\times}{\sim} \beta_{\sigma}(s)$, which reads:
\begin{enumerate}
	\item	in cases \namelink{gal1}, \namelink{gal2}:
	\[\alpha_{\sigma}(s)  \underset{\C^\times}{\sim} \gamma(-2s,\jl(\sigma_1),\as^+,\psi)^{-1} \gamma(2s,\jl(\sigma_1),\as^-,\psi)^{-1},\] 
	\item in cases \namelink{lin}, \namelink{twlin1}, \namelink{twlin2}:
	\[\alpha_{\sigma}(s)  \underset{\C^\times}{\sim} \frac{\gamma(s+1/2,\jl(\sigma_1),\psi)\gamma(s+1/2,\eta_0\otimes\jl(\sigma_1),\psi)}{\gamma(-2s,\jl(\sigma_1),\wedge^2,\psi) \gamma(2s,\jl(\sigma_1),\Sym^2,\psi)}.\]
	\end{enumerate}
\end{remark}
	
\begin{corollary}\label{cor arch fe}
Suppose that $F$ is archimedean. Let $\sigma = \sigma_1 \otimes \sigma_2$ be an irreducible, generic and $M_w$-distinguished representation of $M$ (so that $\sigma_2\simeq \sigma_1^*$). Then 
	\begin{enumerate}
	\item	in cases \namelink{gal1}, \namelink{gal2}:
	\[\alpha_{\sigma}(s)   \sim \gamma_0(-2s,\jl(\sigma_1),\as^+)^{-1} \gamma_0(2s,\jl(\sigma_1),\as^-)^{-1},\] 
	\item in cases \namelink{lin}, \namelink{twlin1}, \namelink{twlin2}	:
	\[\alpha_{\sigma}(s) \sim \frac{\gamma_0(s+1/2,\jl(\sigma_1))\gamma_0(s+1/2,\eta_0\otimes\jl(\sigma_1))}{\gamma_0(-2s,\jl(\sigma_1),\wedge^2) \gamma_0(2s,\jl(\sigma_1),\Sym^2)}\]
	where $\eta_0$ is trivial in case \namelink{lin} and  $\eta_0=\eta_{E/F}$ is the sign character of $\R^\times$ in cases \namelink{twlin1} and \namelink{twlin2}.
  
	\end{enumerate}
\end{corollary}
\begin{proof}
The proof is the same as that of Theorem \ref{thm::constant-FE} and we do not fully repeat it. Note that $F=\R$ except in case \namelink{lin} where $F$ could either be $\R$ or $\C$. Fix the global field $k$ to be $\mathbb{Q}$ except if $F=\C$ where we set $k$ to be $\mathbb{Q}[i]$. In cases \namelink{twlin1}, \namelink{twlin2}, \namelink{gal1},  \namelink{gal2} also set $l=\mathbb{Q}[i]$. Thus, in all cases $k$ has a unique archimedean place $v_0$, $k_{v_0}=F$ and except in case \namelink{lin}, $l$ has a unique archimedean place $w_0$ and $l_{w_0}\simeq E$ so that $v_0$ is inert in $l$. 

Again, we write $\sigma_1$ as a product 
$\sigma_1=\delta_1\times \dots \times \delta_r$ of essentially square integrable representations and then we globalize $\delta_1\otimes \dots \otimes \delta_r$ as in the proof of Theorem \ref{thm::constant-FE}, thanks to Lemma \ref{lm globalizing disc}. By Theorem \ref{thm::constant-FE} together with \cite[8.8.2]{MR4308058} in case \namelink{gp}, and using compatibility of intertwining periods with transitivity of parabolic induction (Proposition \ref{prop::intt-period++PI} \eqref{part closed fe}) as in the proof of \ref{thm::constant-FE}, we know that the expected formula for $\alpha$ holds at all finite places now. 
The corollary follows. \end{proof}

\begin{remark}\label{rem more precise FE2}
In view of Remark \ref{rem more precise FE} and the above proof, the analogue of Remark \ref{rem more precise FE} similarly holds for the archimedean factors of Corollary \ref{cor arch fe}. 
\end{remark}


	\subsection{Poles of intertwining periods}\label{sec order local pole}
	In this subsection, which is the core of our local investigation of intertwining periods, we study the singularities of local intertwining periods. 
	The group case \namelink{gp} is much simpler and is treated separately in Section \ref{sec f}. The rest of this section excludes the group case.
	Our main local result is the following
\begin{theorem}\label{thm mainloc}
Let $\pi$ be an irreducible and distinguished repreresentation of $G'$ that lies in $\Pi_\D(-\frac12,\frac12)$. Then
\[
\Ord_{s=0}(\J_{\pi\otimes \pi}(s))\le \Ord_{s=0}(\L(s,\pi,\theta))
\]
and equality holds if and only if $\pi$ is $H'$-compatible (see Definition \ref{def compatible}).
\end{theorem}	
The proof of this theorem is lengthy and involved. It will occupy the rest of this section.

In order to determine the order of the pole at $s=0$ of the intertwining period above, we prove that this order satisfies a multiplicative relation, which reduces the problem to the two basic cases, one where $\pi$ is a distinguished discrete representation and the other where $\pi$ is of the form $\tau\times \tau^*$ for $\tau$ a non-distinguished essentially square integrable representation. 
We address these special cases first. 

\subsubsection{The discrete case-statement of result}
The main result here is as follows.

	\begin{proposition}\label{prop::Discrete series case}
		Let $\pi$ be an irreducible, distinguished square-integrable representation of $G'$.  Then \[\Ord_{s=0}(\J_{\pi\otimes \pi}(s))=1.\] 
	\end{proposition}
In conjunction with Theorem \ref{thm order pole local L} part \eqref{part disc}, the following corollary is immediate.
	\begin{corollary}\label{cor::pole--case-I}
		Let $\pi$ be as in Proposition \ref{prop::Discrete series case}. Then \[\Ord_{s=0}(\J_{\pi\otimes \pi}(s))=\Ord_{s=0}\L(s,\pi,\theta).\]
	\end{corollary}
In order to prove Proposition \ref{prop::Discrete series case} let $\pi$ be as in the proposition. Observe that by Lemma \ref{lem::intt-period-pole}, $\J_{\pi\otimes \pi}(s)$ has a pole at $s=0$  and  it suffices to show that the pole is at most simple, that is, that  $s\J_{\pi\otimes \pi}(s)$ is holomorphic at $s=0$. We carry this out separately in the archimedean case in Section \ref{sss archpf} and in the non-archimedean case in Section \ref{sss nonarchpf}. For the archimedean case we start with some preparation.

\subsubsection{Some auxiliary results for the archimedean case}\label{sss aux}
Assume that $F$ is archimedean. The proof in case \namelink{lin} requires an auxiliary  lemma on intertwining period on $\GL_2(\R)$. The following elementary lemma will be applied in a key step in its proof. 
\begin{lemma}\label{lem calc1}
We have the following orthogaonality relations on $L^2([0,\pi])$:
for integers $n>k\ge 0$ if $k$ is even then
\[
\int_0^\pi \sin^k(x)\cos(nx)\ dx=0
\]
and if $k$ is odd then
\[
\int_0^\pi \sin^{k}(x)\sin(nx)\ dx=0.
\]

\end{lemma}
\begin{proof}
The cosine product to sum formula
\[
2\cos(a)\cos(b)=\cos(a-b)+\cos(a+b)
\]
is key. 
If $k$ is odd, applying integration by parts we have
\[
\int_0^\pi \sin^{k}(x)\sin(nx)\ dx=\frac{k}{n}\int_0^\pi \sin^{k-1}(x)\cos(x)\cos(nx)\ dx.
\]
By the cosine product to sum formula this equals
\[
\frac{k}{2n}\int_0^\pi \sin^{k-1}(x)[\cos((n-1)x)+\cos((n+1)x)]\ dx
\]
and it therefore suffices to prove the lemma for $k$ even. Write $k=2t$.
We claim that the function $\sin^{2t}(x)$ is a linear combination of the functions $\cos(2jx)$, for $0\le j\le t$. Indeed, since 
\[
\sin^{2t}(x)=(1-\cos^2(x))^t
\] 
it suffices to show the same statement for $\cos^{2t}(x)$ and since $2\cos^2(x)=1+\cos(2x)$ this easily follows by induction from the cosine product to sum formula.
It remains to observe that by another application of this formula
\[
\int_0^\pi \cos(mx)\cos(nx)\ dx=\left[\frac{\sin((m-n)x)}{2(m-n)}+\frac{\sin((m+n)x)}{2(m+n)} \right]\Big |_{x=0}^{\pi}=0 
\]
whenever $m\ne n$. 
\end{proof}

Let $G=\GL_2(\BR)$, $K_2=O(2)$, and $B_2=A_2 N_2$ be the Borel subgroup of upper-triangular matrices in $\GL_2$ with its standard Levi decomposition. Let $\chi_0$ be the sign character and let $\pi$ be an irreducible square integrable representation of $G$ that is $A_2(\R)$-distinguished. Then, there exists $\epsilon\in \{0,1\}$ and an odd integer $k$ such that $\pi$ is the unique irreducible quotient of $\chi_0^\epsilon\abs{\cdot}^{-\frac{k}2}\times \chi_0^\epsilon\abs{\cdot}^{\frac{k}2}$.

Let $\sigma_{\epsilon,s}=I_{B_2(\R)}^G(\chi_0^\epsilon \otimes \chi_0^\epsilon,(s,-s))\simeq \chi_0^\epsilon\abs{\cdot}^s \times \chi_0^\epsilon\abs{\cdot}^{-s}$. We consider the linear form $\jmath(\epsilon,s)$ on $\sigma_{\epsilon,s}$ defined, whenever convergent by the integral
\begin{equation}\label{eq jgl2r}
\jmath(\varphi,\epsilon,s)=\int_{\BR^\times}\varphi_s(\vartheta \diag(1,a))\ d^\times a,\ \ \ \text{where}\ \ \ \vartheta=\begin{pmatrix} 1 & 1 \\ 1 & -1\end{pmatrix}.
\end{equation}
By the general theory of intertwining periods, this converges for $\Re(s)\gg 1$ and admits a meromorphic continuation in $s$. Furthermore, let $M(\epsilon,s):\sigma_{\epsilon,s}\rightarrow \sigma_{\epsilon,-s}$ be the standard intertwining operator. Since $\Hom_{A_2(\BR)}(\sigma_{\epsilon,s},\C)$ is one dimensional for generic $s$ (whenever $\sigma_{\epsilon,s}$ is irreducible) there is a meromorphic function $b_\epsilon(s)$ such that
\begin{equation}\label{eq jfegl2r}
\jmath(\epsilon,-s)=b_\epsilon(s)\jmath(\epsilon,s)\circ M(\epsilon,-s).
\end{equation}

\begin{lemma}\label{lem gl2rintper}
For  $\epsilon\in\{0,1\}$ the linear form $\jmath(\epsilon,s)$ is defined by an absolutely convergent integral and is therefore holomorphic whenever $\Re(s)>-\frac12$.
Furthermore, $\jmath(\epsilon,s)$ satisfies the following properties. Let $k$ be a positive odd integer and $\pi$ the unique irreducible quotient of $\sigma_{\epsilon,-\frac{k}2}$. 
\begin{enumerate}
\item If $k\equiv 1+2\epsilon \mod 4$ then $\jmath(\epsilon,\frac{k}2)$ vanishes on $\pi$ while $\jmath(\epsilon,-s)$ has at most a simple pole at $s=\frac{k}2$. In particular, $(\jmath(\epsilon,-s)\otimes \jmath(\epsilon,s))|_{\sigma_{\epsilon,-\frac{k}2} \otimes \pi}$ is holomorphic at $s=\frac{k}2$.
\item If $k\equiv 3-2\epsilon \mod 4$ then $\jmath(\epsilon,-s)$ is holomorphic at $s=\frac{k}2$ and therefore also $\jmath(\epsilon,-s)\otimes \jmath(\epsilon,s)$ is holomorphic at $s=\frac{k}2$.
\end{enumerate} 
\end{lemma}
\begin{proof}
In [LO, p. 42] it is observed that for $t\in \C$ with $\Re(t)>0$ the following integral converges and satisfies the equality
\begin{equation}\label{eq lo}
2\int_0^\infty \left(\frac{a}{1+a^2}\right)^t\ d^\times a=\frac{\Gamma(\frac{t}2)^2}{\Gamma(t)}.
\end{equation}
A basis to the space of $K_2$-finite vectors in $\sigma_{\epsilon,s}$ (see \cite[Section 2.5]{MR1431508}) is given by $\varphi_{2n,\epsilon,s}$, $n\in \BZ$ where
\[
\varphi_{2n,\epsilon,s}\left[\begin{pmatrix} a & x \\ & b\end{pmatrix}\begin{pmatrix} \cos \theta & \sin \theta \\-\sin \theta & \cos\theta\end{pmatrix}\right]=\chi_0^\epsilon(ab)\abs{\frac{a}{b}}^{s+\frac12}e^{i2n\theta}, \ \ \ a,b\in \BR^\times, x\in \BR,0\le \theta< 2\pi.
\]
Observe that
\[
\varphi_{2n,\epsilon,s}(g\diag(1,-1))=(-1)^\epsilon\varphi_{-2n,\epsilon,s}(g),\ \ \ g\in G
\]
and therefore, whenever convergent
\[
\jmath(\varphi_{2n,\epsilon,s},\epsilon,s)=\int_0^\infty [\varphi_{2n,\epsilon,s}+(-1)^\epsilon\varphi_{-2n,\epsilon,s}](\vartheta \diag(1,a))\ d^\times a.
\]

We observe that for $a>0$ we have
\[
\varphi_{2n,\epsilon,s}(\vartheta \diag(1,a))=(-1)^\epsilon\left(\frac{2a}{1+a^2}\right)^{s+\frac12}e^{i2n\arctan(\frac1{a})}=(-1)^{\epsilon+n}\left(\frac{2a}{1+a^2}\right)^{s+\frac12}e^{-i2n\arctan(a)}.
\]
The last equality follows from the identity $\arctan(a)+\arctan(\frac{1}{a})=\frac{\pi}2$.
Consequently,
\[
\jmath(\varphi_{2n,\epsilon,s},\epsilon,s)=(-1)^n\int_0^\infty  \left(\frac{2a}{1+a^2}\right)^{s+\frac12} [e^{i2n\arctan(a)}+(-1)^\epsilon e^{-i2n\arctan(a)} ]\ d^\times a.
\]
That is
\begin{equation}\label{eq j-int-formula0}
\jmath(\varphi_{2n,0,s},0,s)=2(-1)^n\int_0^\infty  \left(\frac{2a}{1+a^2}\right)^{s+\frac12} \cos(2n\arctan(a))\ d^\times a
\end{equation}
and
\begin{equation}\label{eq j-int-formula1}
\jmath(\varphi_{2n,1,s},1,s)=
2i(-1)^n\int_0^\infty  \left(\frac{2a}{1+a^2}\right)^{s+\frac12} \sin(2n\arctan(a))\ d^\times a.
\end{equation}
By \eqref{eq lo} it follows that $\jmath(\epsilon,s)$ converges absolutely and is therefore holomorphic for $\Re(s)>-\frac12$. 
Furthermore for $n=\epsilon=0$ the combination of \eqref{eq lo} and \eqref{eq j-int-formula0} gives
\[
\jmath(\varphi_{0,0,s},0,s)=2^{s+\frac32}\frac{\Gamma(\frac{s}2+\frac14)^2}{\Gamma(s+\frac12)}
\]
first for $\Re(s)>-\frac12$ and then, by meromorphic continuation, for any $s$. Applying this together with the Gindikin-Kapelevich formula 
\[
M(0,-s)\varphi_{0,0,-s}=\pi^{\frac12}\frac{\Gamma(-s)}{\Gamma(\frac12-s)}\varphi_{0,0,s}
\]
to the functional equation \eqref{eq jfegl2r} we deduce that
\[
b_0(s)\sim \frac{\Gamma(\frac12+s)\Gamma(\frac14-\frac{s}2)^2}{\Gamma(-s)\Gamma(\frac14+\frac{s}2)^2}.
\]
On the other hand for $\epsilon=1$ and $n=0$ we notice that $\jmath(\varphi_{0,1,s},1,s)=0$ and in order to explicate the functional equation \eqref{eq jfegl2r} we turn to $n=1$. Since
\[
\sin(2\arctan(a))=\frac{2a}{1+a^2}
\]
we deduce from \eqref{eq lo} and \eqref{eq j-int-formula1} that
\[
\jmath(\varphi_{2,1,s},1,s)=-2^{s+\frac52}i\frac{\Gamma(\frac{s}2+\frac34)^2}{\Gamma(s+\frac32)}
\]
as meromorphic functions in $s$.
It follows from \cite[Proposition 2.6.3]{MR1431508} that
\[
M(1,-s)\varphi_{2,1,-s}=-\pi^{\frac12}\frac{\Gamma(\frac12-s)\Gamma(-s)}{\Gamma(\frac32-s)\Gamma(-\frac12-s)}\varphi_{2,1,s}.
\]
Applying all this to the functional equation \eqref{eq jfegl2r} we deduce that
\[
b_1(s)\sim \frac{\Gamma(-\frac12-s)\Gamma(\frac32+s)\Gamma(\frac34-\frac{s}2)^2}{\Gamma(\frac12-s)\Gamma(-s)\Gamma(\frac34+\frac{s}2)^2}.
\]
We conclude that
\[
\Ord_{s=\frac{k}2}(b_\epsilon(s))=\begin{cases} 2 & k\equiv 1+2\epsilon \mod 4\\ 0 &  k\equiv 3-2\epsilon \mod 4.\end{cases}
\]
Note that $M(\epsilon,-s)$ is holomorphic at $s=\frac{k}2$ (see e.g. \cite[Section 5 and Theoren 1) h)]{MR0499010}) and that the image of $M(\epsilon,-\frac{k}2)$ is $\pi$.
If $k\equiv 3-2\epsilon\mod 4$, it follows from \eqref{eq jfegl2r} that $\jmath(\epsilon,-s)$ is holomorphic at $s=\frac{k}2$. Assume that $k\equiv 1+2\epsilon\mod 4$. Applying \eqref{eq jfegl2r}  it suffices to show that $\jmath(\epsilon,\frac{k}2)$ vanishes on $\pi$. 

Based on \cite[Section 2.5]{MR1431508} the space of smooth vectors in $\pi$ is topologically spanned by $\varphi_{2n,\epsilon,\frac{k}2}$, where $2\abs{n}\ge k+1$. 
Applying the change of variables $x=\arctan(a)$ to \eqref{eq j-int-formula0} and \eqref{eq j-int-formula1} respectively for $s=\frac{k}2$ we have
\[
\jmath(\varphi_{2n,0,\frac{k}2},0,\frac{k}2)=
2(-1)^n\int_0^{\frac{\pi}2}  \sin(2x)^{\frac{k-1}2} \cos(2nx)\ dx =s=(-1)^n\int_0^{\pi}  \sin(x)^{\frac{k-1}2} \cos(nx)\ dx
\]
and
\[
\jmath(\varphi_{2n,1,\frac{k}2},1,\frac{k}2)=
2i(-1)^n\int_0^{\frac{\pi}2} \sin(2x)^{\frac{k-1}2}\sin(2nx)\ dx =i(-1)^n\int_0^{\pi} \sin(x)^{\frac{k-1}2}\sin(nx)\ dx. 
\]
The vanishing of the above two integrals and therefore the lemma now follows from Lemma \ref{lem calc1}.
\end{proof}

\subsubsection{The discrete case-proof in the archimedean case}\label{sss archpf}
Assume that $F$ is archimedean and let $\pi$ be an irreducible, distinguished square-integrable representation of $G'$. Set $\sigma=\pi\otimes \pi$. We show that $s\J_{\sigma}(s)$ is holomorphic at $s=0$.

Note that either $G'=\GL_2(\R)$ and we are in one of the two cases \namelink{lin} or \namelink{twlin2} or $G'=\D^\times$.
We start with the more difficult cases where $G'=\GL_2(\R)$ and $G=\GL_4(\R)$.

		
We freely use the notation introduced in Section \eqref{sss aux}. Recall that $\theta=\Ad(\diag(\gamma,\gamma))$ where we set $\gamma=\upsilon^\circ$. 
The group $H'$, the centralizer of $\gamma$ in $\GL_2(\BR)$ is $A_2(\BR)$ in case \namelink{lin} and the group
\[
\{\begin{pmatrix} t_1 & -t_2 \\ t_2 & t_1 \end{pmatrix} \in \CM_2(\BR), \ t_1^2+t_2^2\neq 0\}\simeq \C^\times
\]
in case \namelink{twlin2}. As pointed out in Section \eqref{sss aux} $\pi$ is the unique irreducible quotient of $\sigma_{\epsilon,-k/2}$ for some $\epsilon\in\{0,1\}$ and $k$ a positive odd integer. Set $\chi=\chi_0^\epsilon$ and $\tau_{\chi,k}=\sigma_{\epsilon,-k/2}$. Note that $\pi$ is also the unique submodule of $\tau_{\chi,k}^\vee$. 	
Note that 
\[
P_w=M_w=\{\diag(g,\gamma g\gamma^{-1}): g\in \GL_2(\BR)\}.
\]
Let $B$ be the standard Borel subgroup of $G$. Then setting 
\[
L(v\otimes w)=\int_{B_2(\BR)\bs \GL_2(\BR)} v(g)w(g\gamma^{-1})\ dg, \ \ \ v\in \tau_{\chi,k},\ w\in \tau_{\chi,k}^\vee
\]
we have that $0\ne L\in \Hom_{M_w}(\tau_{\chi,k}\otimes \tau_{\chi,k}^\vee,\C)$ 
and the restriction of $L$ to $\tau_{\chi,k}\otimes \pi$ gives rise to a non-zero element $\ell$ in the one dimensional space $\Hom_{M_w}(\pi\otimes\pi,\C)$. That is, if $I:\tau_{\chi,k}\to \pi$ is the, unique up to scalar, projection then the formula 
\[
\ell(I(v)\otimes w)=L(v\otimes w), \ \ \ v\in \tau_{\chi,k},\ w\in \pi\subseteq \tau_{\chi,k}^\vee
\] 
well-defines $\ell$ (the kernel of $I$ is irreducible and inequivalent to $\pi$). Consequently, up to a proportionality scalar we have $\J_{\sigma}(s)=J_P^G(w,\ell,\sigma,s)$ and therefore $\J_{\sigma}(s)$ has at most a simple pole if and only if the same is true for the restriction to $\tau_{\chi,k}\times \pi$ of $J_P^G(w,L,\tau_{\chi,k}\otimes \tau_{\chi,k}^\vee,s)$. Let $\xi=\diag(I_2,\gamma)$, $x=\xi^{-1}\cdot w$ and $\eta=\begin{pmatrix} I_2 & w' \\ I_2 & -w'\end{pmatrix}\in G$ with $w'=\begin{pmatrix} 0 & 1 \\ 1 & 0\end{pmatrix}$ so that $\eta\cdot e=w$. 
In terms of the transitivity of induction $\varphi\mapsto F_{\varphi}: I_B^G(\chi[-\frac{k}2]\otimes\chi[\frac{k}2]\otimes \chi[\frac{k}2]\otimes \chi[-\frac{k}2])\rightarrow I_P^G(\tau_{\chi,k}\otimes\tau_{\chi,k}^\vee)$ we have
\begin{multline*}
J_P^G(F_{\varphi},w,L,\tau_{\chi,k}\otimes \tau_{\chi,k}^\vee,s)=\int_{P_w\bs G_w} L((F_{\varphi})_s(g\eta))\ dg=\\ \int_{P_x \bs G_x} L((F_{\varphi})_s (\xi g\xi^{-1}\eta))\ dg=\int_{B_x\bs G_x} \varphi_{s} (g\xi^{-1}\eta))\ dg. 
\end{multline*}
For the last equality we explicate $L$ and observe that $\delta_P$ is trivial on 
\[P_x=\{\diag(g,g):g\in \GL_2(\R)\}.\] Identify 
\[
I_B^G((\chi[-\frac{k}2]\otimes\chi[\frac{k}2]\otimes \chi[\frac{k}2]\otimes \chi[-\frac{k}2]),s)\simeq I_B^G(\chi^{\otimes 4},\lambda(s))
\] 
where $\lambda(s)=(s-\frac{k}2,s+\frac{k}2, -s+\frac{k}2,-s-\frac{k}2)$ and define the linear form $J(s)$ on $I_B^G(\chi^{\otimes 4})$ by the meromorphic continuation of the integral
\[
J(\varphi,s)=\int_{B_x\bs G_x} \varphi_{\lambda(s)}(g)\ dg.
\]
It suffices to show that $sJ(s)$ is holomorphic at $s=0$ when restricted to $\tau_{\chi,k}\times \pi$.

Let $Q$ be the standard parabolic subgroup of $G$ of type $(1,2,1)$ and let $w_i$ be the permutation matrix in $G$ corresponding to the simple reflection $(i,i+1)$, $i=2,3$. Integrating in stages we have
\[
J(\varphi,s)
=\int_{Q_x\bs G_x}\int_{B_x \bs Q_x}\delta_{Q_x}(q)^{-1} \varphi_{\lambda(s)}(qg)\ dq\ dg.
\]
We observe that $Q_x= B_x V$ where $V=\{I_4+z(E_{3,2}-E_{1,4}):z\in \BR\}$ and $E_{i,j}$ denotes the $4\times 4$ matrix with one in the $(i,j)$-entry and zero in all other entries. 
Since
\[
\varphi_{\lambda(s)}(I_4+z(E_{3,2}-E_{1,4}))=\varphi_{\lambda(s)}(I_4+zE_{3,2})=\varphi_{\lambda(s)}(w_2(I_4+zE_{2,3})w_2)
\]
we conclude that
\[
J(\varphi,s)
=\int_{Q_x\bs G_x} (M(w_2,\lambda(s))\varphi)_{w_2\lambda(s)}(w_2g)\ dg.
\]
We point out that $M(w_2,\lambda(s))$ has a simple pole at $s=0$ (\cite[Section 5 and Theoren 1) h)]{MR0499010}) and in fact, since $\abs{\cdot}^{\frac{k}2} \times \abs{\cdot}^{\frac{k}2}$ is irreducible, $M'(w_2):=sM(w_2,\lambda(s))|_{s=0}$ is a non-zero scalar operator. 
Let $J'(s)$ be defined by the meromorphic continuation of 
\[
J'(\varphi,s)=\int_{Q_x\bs G_x} \varphi_{w_2\lambda(s)}(w_2g)\ dg
\]
so that $J(s)=J'(s)\circ M(w_2,\lambda(s))$.
It remains to show that $J'(s)$ at $s=0$ is holomorphic on sections in the subspace $\tau_{\chi,k}\times \pi$. 

In case \namelink{twlin2} the quotient $Q_x\bs G_x$ is compact and therefore $J'(s)$ is holomorphic at $s=0$. To see this let $\imath: \C\rightarrow \M_2(\BR)$ be the imbedding 
\[
\imath(x+iy)=\begin{pmatrix} x& y \\ -y & x\end{pmatrix},\ \ \ x,y\in \R.
\] 
It restricts to an identification of $\C^\times$ with $H'$. We continue to denote by $\imath:\GL_2(\C)\rightarrow G$ the isomorphism of $\GL_2(\C)$ with $H$ that is defined by applying $\imath$ to each entry. Note that $w_2 G_x w_2=H=\imath(\GL_2(\C))$ and $w_2Q_x w_2=\imath(A_2(\BR)N_2(\C))$. Furthermore $\GL_2(\C)=A_2(\R)N_2(\C) U(2)$ where $U(2)$ is the compact unitary group.

We turn to case \namelink{lin}. 
Let $\zeta=\begin{pmatrix}I_2 & \gamma\\ I_2 & -\gamma\end{pmatrix}$ and note that $\zeta\cdot x=\diag(I_2,-I_2)$ so that
\[
\zeta G_x\zeta^{-1}=\GL_2(\BR)\times \GL_2(\BR)\ \ \ \text{and}\ \ \ \zeta Q_x \zeta^{-1}=\Delta_{A_2(\BR)}(N_2(\BR)\times N_2(\BR))
\]
where $\Delta_{A_2(\BR)}=\{\diag(a,a): a\in A_2(\BR)\}$. It follows that
\[
J'(\varphi,s)=\int_{A_2(\BR) \times K_2\times K_2}\varphi_{w_2\lambda(s)}(w_2\zeta^{-1}\diag(k_1,ak_2)\zeta)\ \delta_{B_2(\BR)}^{-1}(a)\ d(a,k_1,k_2).
\] 
That is, $J'(s)=J''(s)\circ T(s)$ where $T(s)$ is the linear operator on $I_B^G(\chi^{\otimes 4})$ given by
\[T(g,\varphi,s)=\int_{K_2\times K_2}\varphi_{w_2\lambda(s)}(g\diag(k_1,k_2)\zeta)\ d(k_1,k_2)
\]
which is clearly holomorphic and 
\[
J''(\varphi,s)=\int_{A_2(\BR)}\varphi_{w_2\lambda(s)}(w_2\zeta^{-1}\diag(I_2,a))\ da.
\]
Clearly, $T(0)$ preserves the subspace $\tau_{\chi,k}\times \pi$ and it therefore remains to show that restricted to $\tau_{\chi,k}\times \pi$, $J''(s)$ is holomorphic at $s$. Note that for $a=\diag(a_1,a_2)$ we have
\[
w_2\zeta^{-1}\diag(I_2,a)=\diag(\vartheta,\gamma\vartheta)\diag(1,a_1,1,a_2)w_2
\]
and therefore $J''(\varphi,s)=(\jmath(s-\frac{k}2) \otimes \jmath(s+\frac{k}2))(I(w_2,w_2\lambda(s))\varphi)$. Here $\vartheta$ and $\jmath$ are given by \eqref{eq jgl2r}. It remains to show that the restriction of $\jmath(s-\frac{k}2) \otimes \jmath(s+\frac{k}2)$ to $\tau_{\chi,k}\times \pi$ is holomorphic. This follows from Lemma \ref{lem gl2rintper}.

It remains to consider the two cases where either $G'=\BH^\times$ and $H'\simeq \C^\times$ (case \namelink{twlin1}, in this case $G=\GL_2(\BH)$ and $H=\GL_2(\C)$) or $G'=\C^\times$ and $H'=\R^\times$ (case \namelink{gal2}, in this case $G=\GL_2(\C)$ and $H=\GL_2(\R)$). In case \namelink{twlin1} let $A=\BH$ be  embed in $B=\CM_2(\C)$ as usual and in case \namelink{gal2} let $A=\C$ be embed in $B=\CM_2(\BR)$ as usual. In both cases, in its cone of convergence, the intertwining period $\J_\sigma(s)$ has the form 
	\[J_P^G(\varphi;w,\ell,\sigma,s)=\int_{A^\times\backslash B^\times}\eta_s(uh)\ell(\varphi(uh))dh,\] where $u\in G$ is such that $u\cdot e=w$,
	$\eta_s$ is the spherical vector in $\nu^s\times \nu^{-s}$, and the function $g\to  \ell(\varphi(g))$ is bounded for $\ell\in\Hom_{M_w}(\sigma,\C)$ thanks to the Iwasawa decomposition and unitarity of $\pi$. Therefore this integral is dominated by \[J_0(s)=\int_{A^\times\backslash B^\times}|\eta_s(uh)|dh.\] The integral $J_0(s)$ is actually convergent for $\Re(s)>0$ and has a simple pole at $s=0$. This follows from \cite[(7.6)]{MR783512}) in case \namelink{gal2}. The argument of Jacquet and Lai can be adapted to case \namelink{twlin1} as well to show that $J_0(s)\sim \Gamma(2s)$ has a simple pole at $s=0$. 
 All together this implies that $J_P^G(\varphi;w,\ell,\sigma,s)$ has at most a simple pole at $s=0$.

\subsubsection{The discrete case-proof in the non-archimedean case}\label{sss nonarchpf}
Let $\rho$ be the cuspidal representation and $k\in \N$ be such that $\pi=\St_k(\rho)$ and similarly write $\JL(\rho)=\St_l(\rho')$ so that $\nu_\rho=\nu^\ell$. Then $\rho\simeq \rho^*$.
Recall that we fixed a non-zero $\ell=\ell_{\pi\otimes \pi}\in \Hom_{M_w}(\pi\otimes \pi,\C)=\Hom_{M_w}(\pi[t]\otimes \pi[-t],\C)$, $t\in \C$.
Applying \eqref{eq shiftJ} it suffices to show that $\J_{\pi[kl/2]\otimes \pi[-kl/2]}(s)$ has at most a simple pole at $s=-kl/2$.

For this we `double' the set-up again. Let $G_1=G_{4m}(F)=\GL_{4a}(\D)$, let $Q_1=L_1 V_1$ be the standard parabolic subgroup of $G_1$ of type $(a,a,a,a)$
and 
\[
w_1'=\begin{pmatrix} & & & I_a \\ & & I_a & \\ & I_a & & \\ I_a & & & \end{pmatrix}
\]
a representative of $w_{L_1}$ in $G_1$ that, in fact, lies in $G_1\cdot e$.
Let 
\[
\sigma_1=\pi[-kl/2]\otimes \pi[kl/2]\otimes \pi[-kl/2]\otimes \pi[kl/2],
\] it is a representation of $L_1$ and let $\ell_1'\in \Hom_{(L_1)_{w_1'}}(\sigma_1,\C)$ be defined by
\[
\ell_1'(v_1\otimes v_2\otimes v_3\otimes v_4)=\ell(v_1\otimes v_4)\ell(v_2\otimes v_3).
\] 
It follows from the proof of \cite[Proposition 10.10]{MR4308058} that for any $\varphi\in I_P^G(\pi[kl/2]\otimes \pi[-kl/2])$ there exists $\tilde\varphi\in I_{Q_1}^{G_1}(\sigma_1)$ such that
\[
\J_{\pi[kl/2]\otimes \pi[-kl/2]}(\varphi, s)=
J_{Q_1}^{G_1}(\tilde\varphi, w_1',\ell_1',\sigma_1,s\varpi_1).
\]
Here $\varpi_1\in (\fra_{M_1}^*)_{w_{M_1}}^-$ is defined by $e^{\sprod{\varpi}{H_{M_1}(\diag(g_1,g_2)}}=\nu(\det (g_1g_2^{-1}))$ for $g_1,g_2\in G$ where $P_1=M_1U_1$ is the parabolic of type $(2a,2a)$ of $G_1$. By abuse of notation we now also identify $\C$ with $(\fra_{M_1}^*)_{w_{M_1}}^-$ via $s\mapsto s\varpi_1$.
It therefore suffices to show that $J_{Q_1}^{G_1}( w_1',\ell_1,\sigma_1,s)$ has at most a simple pole at $s=-kl/2$. 

Next, we observe that the intertwining operator $M(w,\pi\otimes \pi,s)$ is holomorphic at $s=-kl/2$. Indeed, this follows from \cite[Theorem 7.1]{MR4308058} since $\pi[kl/2]=L(\Delta(\rho,\frac12, k-\frac12))$, $\pi[-kl/2]=L(\Delta(\rho,\frac12-k, -\frac12))$ and in the terminology of ibid. the corresponding cuspidal segments $\Delta(\rho,\frac12, k-\frac12)$ and $\Delta(\rho,\frac12-k, -\frac12)$ are juxtaposed.  The image of $M(w,\pi\otimes \pi,-kl/2)$ is the square integrable representation $\pi_1=\St_{2k}(\rho)$ of $G$. We have $\pi_1\simeq \pi_1^*$ since the same symmetry holds for $\rho$.

By the functorial nature of parabolic induction, it follows that the standard intertwining operator 
\[
I_{Q_1}^{G_1}(\pi[t]\otimes \pi[-t] \otimes \pi[-kl/2]\otimes \pi[kl/2])\rightarrow I_{Q_1}^{G_1}(\pi[-t]\otimes \pi[t] \otimes \pi[-kl/2]\otimes \pi[kl/2])
\] 
is holomorphic at $t=-kl/2$ and we denote by $\M_1$ its value at $t=-kl/2$. Thus 
\[
\M_1:I_{Q_1}^{G_1}(\sigma_1,s)\rightarrow I_{Q_1}^{G_1}(\pi[kl/2]\otimes \pi[-kl/2] \otimes \pi[-kl/2]\otimes \pi[kl/2], s)
\]
is a well defined intertwining operator independent of $s\in \C$ and its image is a subrepresentation isomorphic to $I_{P_1}^{G_1}(\pi_1\otimes (\pi[-kl/2]\times \pi[kl/2]),s)$.

Let 
\begin{align*}
			n = \left(\begin{smallmatrix}
				& I_a& & \\
				I_a&  &  & \\
				&  & I_a & \\
				&  &  & I_a
			\end{smallmatrix}\right)\in G_1 \ \ \ \text{and}\ \ \ w_1=\begin{pmatrix} & I_{2a} \\ I_{2a} & \end{pmatrix}\in G_1\cdot e
		\end{align*}
and note that $(L_1,w_1') \stackrel{n}{\searrow} (L_1,w_1)$ is an edge on the graph associated with $(G_1,\theta)$ as in \eqref{eq edge}. It follows from proposition \ref{prop::intt-period++FE+Adjacent} that 
\[
J_{Q_1}^{G_1}( w_1',\ell_1',\sigma_1,s)=J_{Q_1}^{G_1}( w_1,\ell_1',n\sigma_1,s)\circ \M_1.
\]
Applying the identity \eqref{formula::intt-period++PI-maximal} of Proposition \ref{prop::intt-period++PI} we have that
\[
J_{Q_1}^{G_1}(\varphi, w_1,\ell_1',n\sigma_1,s)=J_{P_1}^{G_1}(F_\varphi, w_1,\Lambda_{\ell_1'},I_{Q_1\cap M_1}^{M_1}(n\sigma_1),s)
\]
where $\Lambda_{\ell_1'}$ is defined as in  \eqref{formula::defn--Lambda-ell}.
Let $\Psi:\pi[kl/2]\otimes \pi[-kl/2]\otimes \pi[-kl/2]\otimes \pi[kl/2]\rightarrow n\sigma_1$ be the isomorphism 
\[
\Psi(v_1\otimes v_2\otimes v_3\otimes v_4)=v_2\otimes v_1\otimes v_3\otimes v_4. 
\]
Applying the functorial properties of parabolic induction we continue to denote by $\Psi$ the isomorphism
\[
\Psi: I_{Q_1}^{G_1}( \pi[kl/2]\otimes \pi[-kl/2]\otimes \pi[-kl/2]\otimes \pi[kl/2])\rightarrow I_{Q_1}^{G_1}(n\sigma_1).
\]
Then
\[
J_{Q_1}^{G_1}(\varphi, w_1,\ell_1',n\sigma_1,s)=J_{Q_1}^{G_1}(\Psi^{-1}(\varphi), w_1,\ell_1'\circ\Psi,\pi[kl/2]\otimes \pi[-kl/2]\otimes \pi[-kl/2]\otimes \pi[kl/2],s)
\]
and similarly
\begin{multline*}
J_{P_1}^{G_1}(F_{\varphi}, w_1,\Lambda_{\ell_1'},I_{Q_1\cap M_1}^{M_1}(n\sigma_1),s)=\\ J_{P_1}^{G_1}(F_{\Psi^{-1}(\varphi)}, w_1,\Lambda_{\ell_1'\circ\Psi},(\pi[kl/2]\times\pi[-kl/2])\otimes (\pi[-kl/2]\times \pi[kl/2]),s).
\end{multline*}

It is well known that $\pi_1$ is the socle, the maximal semisimple subrepresentation, of $\pi[kl/2]\times \pi[-kl/2]$ (it follows from \cite[Proposition 2.7]{MR1040995}) and furthermore appears there with multiplicity one. 
Also, for a representation $\Pi$ of $G$ there is a natural isomorphism $\Hom_{(M_1)_{w_1}}(\Pi\otimes \Pi^*,\C)\simeq \Hom_G(\Pi,\Pi)$. Applying this to $\Pi= \pi[kl/2]\times \pi[-kl/2]$, so that $\Pi^*\simeq \pi[-kl/2]\times \pi[kl/2]$, we conclude that restriction gives an isomorphism of one dimensional spaces 
\[
\Hom_{(M_1)_{w_1}}((\pi[kl/2]\times \pi[-kl/2])\otimes (\pi[-kl/2]\times \pi[kl/2]))\simeq \Hom_{(M_1)_{w_1}}(\pi_1\otimes (\pi[-kl/2]\times \pi[kl/2])).
\] 
The left hand side is spanned by $\Lambda_{\ell_1'\circ \Psi}$. Let $\ell_1$ be its image on the right hand side.
All together, it suffices to show that $J_{P_1}^{G_1}(w_1,\ell_1,\pi_1 \otimes (\pi[-kl/2]\times \pi[kl/2]),s)$ has at most a simple pole at $s=-kl/2$.

Let $\Phi$ be the projection
\[
\Phi:\pi_1 \otimes (\pi[-kl/2]\times \pi[kl/2])\rightarrow \pi_1\otimes \pi_1.
\]
For the reasons already explained above $\ell_1$ factors through $\Phi$ and we write $\L_1\in \Hom_{(M_1)_{w_1}}(\pi_1\otimes \pi_1,\C)$ for the linear form such that $\L_1\circ \Phi=\ell_1$. As above, we continue to denote by $\Phi$ the projection of induced representations 
\[
\Phi:I_{P_1}^{G_1}(\pi_1 \otimes (\pi[-kl/2]\times \pi[kl/2]))\rightarrow I_{P_1}^{G_1}(\pi_1\otimes \pi_1).
\]
Then 
\[
J_{P_1}^{G_1}(\varphi,w_1,\ell_1,\pi_1 \otimes (\pi[-kl/2]\times \pi[kl/2]),s)=\J_{\pi_1\otimes \pi_1}(\Phi(\varphi),s)
\]
where we set $\ell_{\pi_1\otimes \pi_1}=\ell_1$ on the right hand side. We conclude that it suffices to show that $\J_{\pi_1\otimes \pi_1}(s)$ has at most a simple pole at $s=-kl/2$.

Applying Proposition \ref{prop::FE--MultiOne} and its explication, Theorem \ref{thm::constant-FE}, we have
\[
\J_{\pi_1\otimes \pi_1}(-s)\circ M(w_1,\pi_1\otimes \pi_1,s)=\alpha(s)\beta(s)\J_{\pi_1\otimes \pi_1}(s)
\]
where 
\[\beta(s)=1\] in cases \namelink{gal1} and \namelink{gal2}, \[\beta(s)=\frac{L(-s+kl,\rho')L(-s+kl,\eta_0\otimes\rho')}{L(s+kl,\rho')L(s+kl,\eta_0\otimes \rho')}\]
 \namelink{lin},\namelink{twlin1} and \namelink{twlin2}, 

\[
\alpha(s)\sim \frac{L^-(-2s,\rho')}{L^{-}(-2s+2kl,\rho')}\frac{L^{+}(2s,\rho')}{L^{+}(2s+2kl,\rho')} 
\]
and where we set
\[
L^{+}(s,\rho')=\begin{cases} L(s,\rho',\wedge^2) & \text{in cases \namelink{lin},\namelink{twlin1} and \namelink{twlin2}}\\ L(s,\rho',\As^+)  & \text{in cases \namelink{gal1} and \namelink{gal2}}\end{cases}
\]
and
\[
L^{-}(s,\rho')=\begin{cases} L(s,\rho',\Sym^2) & \text{in cases \namelink{lin},\namelink{twlin1} and \namelink{twlin2}}\\ L(s,\rho',\As^-)  & \text{in cases \namelink{gal1} and \namelink{gal2}}.\end{cases}
\]

It follows from Lemma \ref{lem cusprs} and the decompositions \eqref{eq rs decomp} and \eqref{eq asai rs decomp} that $\alpha(s)$ is holomorphic and non-zero at $s=-kl/2$, and the same is true for $\beta(s)$ thanks to well-known properties of standard $L$-factors. 
It further follows from \cite[Theorem 7.1]{MR4308058} that $M(w_1,\pi_1\otimes \pi_1,s)$ has a simple pole at $s=-kl/2$.

It therefore suffices to show that $\J_{\pi_1\otimes \pi_1}(-s)$ is holomorphic at $s=-kl/2$
or equivalently (see \eqref{eq shiftJ}) that $\J_{\pi_1[kl/2]\otimes \pi_1[-kl/2]}(s)$ is holomorphic at $s=0$. 

Since $\pi_1[kl/2]$ is not distinguished, it follows from Proposition \ref{prop open contribution two discrete} that only the open $P_1$-orbit $P_1\cdot w_1$ contributes to $\pi_1[kl/2]\otimes \pi_1[-kl/2]$ and we can now apply Lemma \ref{lem::holomorphy} to deduce that $\J_{\pi_1[kl/2]\otimes \pi_1[-kl/2]}(s)$ is holomorphic at $s=0$. This completes the proof of 
Proposition \ref{prop::Discrete series case}.

\subsubsection{The second basic case}
	The following result is Theorem \ref{thm mainloc} for another basic case.
	
	\begin{proposition}\label{prop::Product case}
		Let $\pi$ be a representation of $G'$ of the form $\pi=\tau\times \tau^*$ for an irreducible, essentially square-integrable representation $\tau$ of $\GL_{a/2}(\D)$ that is not distinguished (in particular, $a$ is assumed even). Assume further that $\abs{r(\tau)}<\frac12$ and set $\sigma=\pi\otimes \pi$. Then \[\Ord_{s=0}(\J_\sigma(s))=1.\]
	\end{proposition}
	\begin{proof}
Note first that, by a closed orbit argument as in \cite[Proposition 7.1]{MR3541705},  $\pi$ is distinguished so that $\J_\sigma(s)$ makes sense and has a pole at $s =0$ by Lemma \ref{lem::intt-period-pole}. It therefore suffices to show that the pole is at most simple. Let $Q = LV$ be the standard parabolic subgroup of $G$ of type $(\frac{a}2,\frac{a}2,\frac{a}2,\frac{a}2)$. Its Levi subgroup $L$ is 
$\theta_x$-stable (see \eqref{eq def good x}) and
by \eqref{formula::intt-period--change-in-orbit} it suffices to show that $J_P^G(x,\CL,\sigma,s)$ has at most a simple pole at $s=0$ for a non-zero $\CL$ in the one dimensional space $ \Hom_{M_{x}}(\sigma,\C)$. 
		
		 Let $\sigma_1 = \tau \otimes \tau^* \otimes  \tau^*\otimes \tau$, a representation of $L$ and let $\ell \in \Hom_{L_x}(\sigma_1,\C)$ be defined by 
		\begin{align*}
			\ell( v_1 \otimes v_2 \otimes v_3 \otimes v_4)  = \ell_1(v_1\otimes v_3)  \cdot \ell_1(v_4\otimes v_2),
		\end{align*}
		where $0\ne \ell_1\in \Hom_{M_1}(\tau\otimes \tau^*,\C)$ and $M_1=\{\diag(g,\iota(g)): g\in \GL_{\frac{a}2}(\D)\}$. Define $\Lambda_{\ell}$ as in \eqref{formula::defn--Lambda-ell} to be the $M_{w'}$-invariant linear form on $I_{Q \cap M}^M(\sigma_1)$. Note that the role of $M$ and $L$ is reversed in our context. As $\ell$ is nonzero, $\Lambda_{\ell}$ is non-zero by \cite[Lemma 3.3]{MR4679384}. Set $\CL=\Lambda_\ell$. Since $\pi$ is irreducible, we have $\sigma \cong I_{Q \cap M}^M(\sigma_1)$ and by \eqref{formula::intt-period++PI-maximal} we have
		\begin{align*}
			J_P^G(F_{\varphi};x,\CL,\sigma,s) = J_Q^G(\varphi;x,\ell,\sigma_1,\underline{s}), 
		\end{align*}
		where $\underline{s} = (s,s,-s,-s) \in  (\fra_{L,\C}^*)^-_w$ viewed naturally as a subspace of $\fra_{L,\C}^*\simeq \C^4$. Let $\alpha$ be the unique element in $\Delta_M \subset \Delta_L$ and $s_{\alpha}\in W(L)$ be the elementary symmetry associated to $\alpha$. That is, $s_\alpha$ is the elementary symmetry represented by
		\begin{align*}
			n = \left(\begin{smallmatrix}
				I_{a/2} & & & \\
				&  & I_{a/2} & \\
				& I_{a/2} & & \\
				&  &  & I_{a/2}
			\end{smallmatrix}\right)   \in s_{\alpha}L.
		\end{align*}
		Note that $(L,x) \stackrel{n}{\searrow} (L,x_1)$ is an edge in $\mathfrak{G}$ with $x_1 = n \cdot x$ and $s_{\alpha}\underline{s}=(s,-s,s,-s)$. By Proposition \ref{prop::intt-period++FE+Adjacent}, we have
		\begin{align*}
			J_Q^G(x,\ell,\sigma_1,\underline{s})  =  J_Q^G ( x_1,\ell,\sigma_1,s_{\alpha}\underline{s})\circ M(n,\sigma_1,\underline{s}).
		\end{align*}
We claim that $M(n,\sigma_1,\underline{s})$ has at most a simple pole at $s=0$. Indeed, this follows from Lemma \ref{lem hol int op} in conjunction with Lemma \ref{lem pairs}.
Therefore, it suffices to show that $J_Q^G (x_1,\ell,\sigma_1,s_{\alpha}\underline{s})$ is holomorphic at $s = 0$.	
		
Note that $(L,x_1)$ is a minimal vertex in $\mathfrak{G}$ and $x_1\in M$ represents $w_L^M$. In particular, there exists $u_1\in M$ such that $u_1\cdot e=x_1$. Apply equation \eqref{formula::intt-period++PI-minimal-I} and note that the outer integral over $g$ on its right hand side is over a compact domain. It therefore suffices to prove the holomorphy at $s = 0$ of 
		\begin{align}\label{formula::2nd-case-I}
			\int_{ L_{x_1} \backslash M_{x_1} } \ell (\varphi_{s_{\alpha}\underline{s}} (u_1l) ) dl
		\end{align}
		for all $\varphi \in I_{Q \cap M}^M (\sigma_1)\simeq  (\tau \times \tau^*) \otimes ( \tau^* \times \tau )$. 
		That is, it suffices to prove that $J(\tau,s)\otimes J'(\tau^*,s)$ is holomorphic at $s=0$ where
		\[
		J(\tau,s)=J_{P'}^{G'}(x_1',\ell_1,\tau\otimes \tau^*,(s,-s)),\ \ \ J'(\tau^*,s)= J_{P'}^{G'}(x_1'',\ell_1',\tau^*\otimes \tau,(s,-s))
		\]
and we write $x_1=\diag(x_1',x_1'')$ with $x_1',x_1''\in G'$ and $\ell_1'(v'\otimes v)=\ell_1(v\otimes v')$ for $v$ in $\tau$ and $v'$ in $\tau^*$.

		When $F$ is $p$-adic we conclude from Proposition \ref{prop open contribution two discrete} and its proof that the unique open $P'$-orbit in $G'\cdot e$ is the only relevant orbit for either $\tau\otimes \tau^*$ or $\tau^*\otimes \tau$ and therefore from Lemma \ref{lem::holomorphy} that $J(\tau,s) \otimes J'(\tau^*,s)$ is holomorphic at $s = 0$. When $F$ is archimedean assume without loss of generality that $r(\tau)\le 0$, (if this is not the case the following argument still works by switching between $\tau$ and $\tau^*$). If  $J'(\tau^*,s)$ is not holomorphic at $s=0$ then its leading term at $s=0$ defines a non-zero element of $\Hom_{H'}(\pi,\C)$ that vanishes on sections with support on the open $(P',H')$-double coset. This contradicts Proposition \ref{prop arch geom lem}. It follows that $J'(\tau^*,s)$ is holomorphic at $s = 0$. For holomorphicity of $J(\tau,s)$ at $s=0$ we apply Corollary \ref{cor arch fe} to deduce that
\[
J(\tau,s)\sim \chi(s)\frac{L^+(1+2s,\JL(\tau)^\vee)L^-(1-2s,\JL(\tau)^\vee)}{L^+(-2s,\JL(\tau))L^-(2s,\JL(\tau))} J'(\tau^*,-s)\circ M(w,\tau\otimes \tau^*,s),
\]	
where 
\[\chi(s)=1\] in cases \namelink{gal1} and \namelink{gal2} and \[\chi(s)=\frac{L(\frac{1}{2}-s,\JL(\tau)^\vee)L(\frac{1}{2}-s,\eta_0\otimes\JL(\tau)^\vee)}{L(\frac{1}{2}+s,\JL(\tau))L(\frac{1}{2}+s,\eta_0\otimes \JL(\tau))}\]
in cases \namelink{lin},\namelink{twlin1} and \namelink{twlin2}.
We immediately observe that $\chi(s)$ is holomorphic and nonzero at $s=0$, by the usual properties of standard $L$-factors and thanks to our assumption that $\abs{r(\tau)}<\frac12$. Together with Lemma \ref{lem hol int op} it suffices to show that
\[
\frac{L^+(1+2s,\JL(\tau)^\vee)L^-(1-2s,\JL(\tau)^\vee)}{L^+(-2s,\JL(\tau))L^-(2s,\JL(\tau))} \frac{L(2s,\JL(\tau),\JL(\tau^*)^\vee)}{L(1+2s,\JL(\tau),\JL(\tau^*)^\vee)}
\] 
is holomorphic at $s=0$ where we write 
\[
(L^+(s,\Pi),L^-(s,\Pi))=\begin{cases}
(L(s,\Pi,\wedge^2),L(s,\Pi,\Sym^2)) & \text{in cases \namelink{lin}, \namelink{twlin1} and \namelink{twlin2}}\\
(L(s,\Pi,\As^+),L(s,\Pi,\As^-)) &  \text{in cases \namelink{gal1} and \namelink{gal2}}.
\end{cases}
\]		
Each of the terms $L^+(1+2s,\JL(\tau)^\vee)$, $L^-(1-2s,\JL(\tau)^\vee)$ and $L(1+2s,\JL(\tau),\JL(\tau^*)^\vee)$ is holomorphic at $s=0$ by Lemma \ref{lem dqr int@1rs}, \eqref{eq rs decomp} and \eqref{eq asai rs decomp}. Since furthermore, $f(s)/f(-s)$ is holomorphic at $s=0$ for any meromorphic function near zero, applying \eqref{eq rs decomp} in cases \namelink{lin}, \namelink{twlin1} and \namelink{twlin2} (resp. \eqref{eq asai rs decomp} in cases \namelink{gal1} and \namelink{gal2}) it suffices to show that 
\[
\frac{L(2s,\JL(\tau),\JL(\tau^*)^\vee)}{L(2s,\JL(\tau),\JL(\tau)^\vartheta)}.
\]
It remains to observe that $\JL(\tau^*)^\vee\simeq \JL(\tau)^\vartheta$. Indeed, $\JL(\tau^*)^\vee=\JL(\tau^\iota)$. In the non-Galois cases, $\tau^\iota\simeq \tau$ and the isomorphism is straightforward. In the Galois cases we must have $G'=\GL_2(\C)$ so that $\JL(\tau^\iota)=\tau^{\iota}$. Furthermore, in either cases \namelink{gal2} or \namelink{gal1} we have $\iota=\vartheta$. We conclude that $J(\tau,s)$ is holomorphic at $s=0$ and the  proposition follows.	\end{proof}
	
Together with Theorem \ref{thm order pole local L} we obtain the following special case of Theorem \ref{thm mainloc}.
	
	\begin{corollary}\label{cor::pole-case-II}
		Let $\pi =\tau \times \tau^*$ be an irreducible representation of $G'$ with $\tau$ irreducible and square-integrable such that $ \abs{r(\tau)} < 1/2$. Then 
		\[
\Ord_{s=0}(\J_{\pi\otimes \pi}(s))\le \Ord_{s=0}(\L(s,\pi,\theta))
\]
and equality holds if and only if $\pi$ is $H'$-compatible.
	\end{corollary}

	\subsubsection{Multiplicativity}
	
Let $m=m_1+m_2$ and accordingly let $a=a_1+a_2$ be the decomposition so that $G_{m_i}(F)=\GL_{a_i}(\D)$, $i=1,2$. The order of pole for the intertwining period at hand satisfies the following multiplicative property. 
	\begin{proposition}\label{prop::Induction-pole}
		Let $\pi_i$ be a distinguished representations of $G_{m_i}(F)$ in $\Pi(-\frac12,\frac12)$ (see \S\ref{ss normint}), $i = 1,2$, and let $\pi = \pi_1 \times \pi_2$. Then
\[
\Ord_{s=0}(\J_{\pi\otimes \pi}(s))=k_1 + k_2 + k
\]
where
\[
k_i=\ord_{s=0}(\J_{\pi_i\otimes \pi_i}(s)),\ i=1,2\ \ \ \text{and}\ \ \ k=\Ord_{s=0}(L(s,\jl(\pi_1),\jl(\pi_2)^\vartheta))
\]	
where $\vartheta$ is the $E/F$-Galois involution in the Galois cases and the identity automorphism otherwise.

\end{proposition}
	\begin{proof}
Let $Q=LV\subset G$ be the standard parabolic subgroup of type $(a_1,a_2,a_1,a_2)$. Denote by $\varsigma$ the representation $\pi_1 \otimes \pi_2 \otimes \pi_1 \otimes \pi_2$ of $L$. 
	Note that the map $\ell\mapsto \Lambda_\ell: \Hom_{L_w}(\varsigma,\C)\rightarrow \Hom_{M_w}(\pi\otimes \pi,\C)$ defined as in \eqref{formula::defn--Lambda-ell} is an isomorphism between two one dimensional spaces. 
	By Proposition \ref{prop::intt-period++PI}\eqref{part closed fe} and in its notation there exists a non-zero $\ell\in\Hom_{L_w}(\varsigma,\C)$ such that  
		\begin{align*}
			\J_{\pi\otimes \pi}(F_\varphi,s)=J_Q^G(\varphi;w,\ell, \varsigma,s),\ \ \ \varphi \in I_Q^G(\varsigma).
		\end{align*} 
We proceed by computing the order of pole of $J_Q^G(\varphi;w,\ell, \varsigma,s)$ at $s=0$.
We have an edge $(L,w) \stackrel{n}{\searrow} (L,w')$ in $\mathfrak{G}$ with 
		\begin{align*}
			w'= n \cdot w=\left(\begin{smallmatrix}
				& I_{a_1}  & & \\
				I_{a_1}&  &  & \\
				&  & & I_{a_2}\\
				&  &  I_{a_2} & 
			\end{smallmatrix}\right)\ \ \ \text{and}\ \ \ n = \left(\begin{smallmatrix}
				I_{a_1} & & & \\
				&  & I_{a_1} & \\
				& I_{a_2} & & \\
				&  &  & I_{a_2}
			\end{smallmatrix}\right).
		\end{align*}
		Write $Q'=L'V'$ for the standard parabolic of type $(a_1,a_1,a_2,a_2)$. By Proposition \ref{prop::intt-period++FE+Adjacent} we have
		\begin{align*}
			J_Q^G (w,\ell, \varsigma,s) = J_{Q'}^G (w',\ell, n\varsigma,\underline{s})\circ M(n,\varsigma,s),
		\end{align*}
		with $\underline{s} = (s,-s,s,-s)\in \C^4\simeq \fra_{L,\C}^*$. 
It follows from Lemma  \ref{lem hol int op} that
\[
\Ord_{s=0}(M(n,\varsigma,s))=\Ord_{s=0}(\frac{L(2s,\jl(\pi_1), \jl(\pi_2)^{\vee})}{L(1+2s,\jl(\pi_1) , \jl(\pi_2)^{\vee})}).
\]
It follows from Theorem \ref{thm mult 1 and ssduality} that $\jl(\pi_2)^{\vee})\simeq \JL(\pi^\iota)$. In cases \namelink{lin}, \namelink{twlin1} and \namelink{twlin2} $\iota$ is an inner automorphism of $G$ so that $\pi^\iota\simeq \pi$. In the Galois case, as explained in \S \ref{sss auxi}, we have $\pi_2^\iota\simeq \pi_2^\theta$ and therefore $\jl(\pi_2)^{\vee}\simeq \JL(\pi)^\vartheta$. By Lemma \ref{lem dqr int@1rs} $L(1+2s,\jl(\pi_1) , \jl(\pi_2)^{\vee})$ is holomorphic at $s=0$.
It follows that $k=\Ord_{s=0}(L(2s,\jl(\pi_1), \jl(\pi_2)^{\vee}))$.

We conclude that $\Ord_{s=0}(\J_{\pi\otimes \pi}(s))=k+\Ord_{s=0}(J_{Q'}^G (w',\ell, n\varsigma,\underline{s}))$ and it therefore suffices to show that 
\[
\Ord_{s=0}(J_{Q'}^G (w',\ell', \varsigma',\underline{s}))=k_1+k_2
\]
where $\varsigma' = \pi_1 \otimes \pi_1 \otimes \pi_2 \otimes \pi_2\simeq n \varsigma$ and $\ell'(v_1\otimes v_1'\otimes v_2\otimes v_2')=\ell(v_1\otimes v_2\otimes v_1'\otimes v_2')$ for $v_i,v_i'$ in the space of $\pi_i$, $i=1,2$. Note that there exists $c\ne 0$ such that $\ell'=c(\ell_{\pi_1\otimes \pi_1}\otimes \ell_{\pi_2\otimes \pi_2})$.

Let $P=MU$ be the standard parabolic of type $(2a_1,2a_2)$ and note that  $(L',w')$ is a minimal vertex in the graph $\mathfrak{G}$. 
Let $u  \in M$ be such that $u \cdot e= w' $.
It follows from Proposition \ref{prop::intt-period++PI} \eqref{part open fe}  that
\[
J_{Q'}^G (\varphi;w',\ell, \varsigma_1,\underline{s})= \int_{P_{w'} \backslash G_{w'} }\int_{L'_{w'} \backslash M_{w'}}\ell\left(((I_Q^G(gu,\sigma,\lambda)\varphi)[e])_{\underline{s}}(m) \right)\ dm \ dg.
\]
That is,
\[
J_{Q'}^G (\varphi;w',\ell, \varsigma_1,\underline{s})=J_P^G(\xi;w', J_{Q'\cap M}^{M,\theta_{w'}}(e,\ell,\varsigma_1,\underline{s}),I_{Q'\cap M}^M(\varsigma_1,\underline{s}), 0)
\]
where $\xi\in I_P^G(I_{Q'\cap M}^M(\varsigma_1,\underline{s}))$ is defined by $\xi(g)=(I_P^G(g,\underline{s})\varphi)[e]$.
It is a consequence of \cite[Lemma 3.3]{MR4679384} that the closed orbit intertwining period $J_P^G(w', \L,I_{Q'\cap M}^M(\varsigma_1,\underline{s}), \lambda)$ is holomorphic at $\lambda=0$ for any $\L\in \Hom_{M_{w'}}( I_{Q'\cap M}^M(\varsigma_1,\underline{s}),\C)$ (and any $s\in \C$) and that 
\[
\Ord_{s=0}(J_{Q'}^G (w',\ell, \varsigma_1,\underline{s}))=\Ord_{s=0}(J_{Q'\cap M}^{M,\theta_{w'}}(e,\ell,\varsigma_1,\underline{s})).
\]
Observing that
\[
J_{Q'\cap M}^{M,\theta_{w'}}(e,\ell,\varsigma_1,\underline{s})\circ I_{Q'\cap M}^M(u,\underline{s})=c\ \J_{\pi_1\otimes \pi_1}(s)\otimes \J_{\pi_2\otimes \pi_2}(s)
\]
the Proposition follows.

	\end{proof}

	\subsubsection{Completion of proof of Theorem \ref{thm mainloc}}
	Let $\pi$ be as in the statement of the theorem. It follows from Theorem \ref{thm classif dist S} that $\pi\simeq \delta_1\times \cdots\times \delta_k\times \tau_1\times \tau_1^* \times \cdots\times \tau_\ell\times \tau_\ell^*$ for some irreducible essentially square integrable representations $\delta_i, \tau_j$ such that $\delta_i$ is distinguished $\tau_j$ is not-distinguished and $\abs{r(\tau_i)}<\frac12$, $i=1,\dots,k,\ j=1,\dots,\ell$. The theorem is proved by induction on $k+\ell$. For the base of induction, $k+\ell=1$, apply Corollaries \ref{cor::pole--case-I} and \ref{cor::pole-case-II}. The induction step follows from Proposition \ref{prop::Induction-pole} and Theorem \ref{thm order pole local L} part \eqref{part multiplicative}.

	\subsubsection{The group case}\label{sec f}
	
	Finally, there is another type of intertwining period that shows up at half of the places when one is concerned with the global Galois case. The places in question are those places of the number field $F$ that split over the quadratic extension $E$. The result we need in this case follows directly from the properties of local intertwining operators. 
	
	Let $(G,H,\theta)=(G_{2m}(F),H_{2m}(F),\theta_{2m})_{\namelink{gp}}$ and $(G',H',\theta')=(G_m(F),H_m(F),\theta_m)_{\namelink{gp}}$ so that 
	$G'=\GL_m(D) \times \GL_m(D)$. Let $P=MU$ be the parabolic subgroup of $G$ of type $(m,m)$ so that $M=\{\diag(g_1,g_2):g_1,g_2\in G'\}\simeq \GL_m(D)^4$.
	Let $\pi=\pi_1 \otimes \pi_2$ be an irreducible, generic representation of $G'(F)$ where $\pi_1,\pi_2$ are generic, irreducible representations of $\GL_m(D)$.
	Note that $\pi^*=\pi_2^\vee \otimes \pi_1^\vee$ and let
	\[
	\sigma=\pi\otimes \pi^*
	\]
	be the corresponding representation of $M(F)\simeq G'(F)\times G'(F)$.

We further let $G_1=\GL_{2m}(D)$ and $P_1=M_1U_1$ be its parabolic subgroup of type $(m,m)$ so that $G=G_1\times G_1$, $P=P_1\times P_1$, $M=M_1\times M_1$ and $U=U_1\times U_1$.


	Set $w =  (w',w'), \ u=(I_{2m},w')\in G$ where $w'=\left( \begin{smallmatrix}
		& I_{m} \\
		I_{m} & 
	\end{smallmatrix}\right)\in G_1$ so that $u\cdot e=w$. Note that
	$G_w=\{(g,w'gw'):g\in G_1\}$ and $P_w=M_w=\{(g,w'gw'):g\in M_1\}$.

	Let $\ell_\sigma\in \Hom_{M_w}(\sigma,\C)$ be defined by
	\begin{align*}
		\ell_\sigma(v_1 \otimes v_2 \otimes v_2^\vee \otimes v_1^\vee)  \mapsto  \langle v_1, v_1^\vee \rangle \cdot \langle v_2, v_2^\vee \rangle
	\end{align*}
	for $v_i$ in the space of $\pi_i$ and $v_i^\vee$ in $\pi_i^\vee$, $i=1,2$.
	Identify $\C$ with $(\fra_{M,\C}^G)_{\theta_w}^-$ which is the diagonal imbedding of $(\fra_{M_1,\C}^{G_1})^*$ in $(\fra_{M,\C}^{G})^*=(\fra_{M_1,\C}^{G_1})^*\times (\fra_{M_1,\C}^{G_1})^*$, so that 
	\[
	e^{\sprod{s}{H_M(\diag(m_1,m_2),\diag(m_3,m_4)}}=\abs{\frac{\det (m_1m_3)}{\det (m_2m_4)}}^s, \ \ \ m_i\in \GL_m(D),\ i=1,2,3,4,\ s\in \C.
	\] 
	Note that $I_P^G(\sigma)=I_{P_1}^{G_1}(\mu) \otimes I_{P_1}^{G_1}(\mu^*)$ where $\mu=\pi_1\otimes \pi_2^\vee$ is the corresponding representation of $M_1$.

	The intertwining period $\J_\sigma(s)$ is defined on $I_P^G(\sigma)$ by the meromorphic continuation of 
	\[
	\J_\sigma(\varphi: s) = \int_{M_w \backslash G_w } \ell_\sigma (\varphi_s (gu )) dg.
	\]
	\begin{proposition}\label{prop above}
With the above notation assume that $\pi_1,\pi_2\in \Pi_D(-\frac12,\frac12)$. Then we have
	\[\Ord_{s=0}(\J_\sigma(s))=\Ord_{s=0}(\L(s,\pi,\theta)).\] 
	\end{proposition}
\begin{proof}
Note that $I_{P_1}^{G_1}(\mu^*,s)\simeq I_{P_1}^{G_1}(\mu,s)^\vee$ with the $G_1$-invariant pairing
\[
\sprod{\varphi_1}{\varphi_2}= \int_{P_1\bs G_1}\ell((\varphi_1)_s(g)(\varphi_2)_s(w'g))\ dg,\ \ \ \varphi_1\in I_{P_1}^{G_1}(\mu),\ \varphi_2\in I_{P_1}^{G_1}(\mu^*)
\]
that is independent of $s$. 
Since $I_P^G(\sigma)$ is spanned by pure tensors, it suffices to consider $\varphi=\varphi_1\otimes \varphi_2$ with $\varphi_1\in I_{P_1}^{G_1}(\mu)$ and $\varphi_2\in I_{P_1}^{G_1}(\mu^*)$. For such $\varphi$ we have
\begin{align*}
		\J_\sigma(\varphi: s) & = \int_{M_1\bs G_1} \ell((\varphi_1)_s(g)(\varphi_2)_s(w'g))\ dg \\
		               & =\int_{P_1\bs G_1} \int_{U_1} \ell((\varphi_1)_s(g)(\varphi_2)_s(w'ug))\ du\ dg \\
		               & =\sprod{\varphi_1}{M(w',\mu^*,s)\varphi_2}.
	\end{align*}
	The proposition now follows from Lemmas \ref{lem hol int op} and \ref{lem dqr int@1rs}.
	
	\end{proof}

	\section{Global theory: the Maass-Selberg relations}\label{sec::Global}
	
	Assume that $F$ is a number field. Let $(G,H,\theta)=(G_{2m},H_{2m},\theta_{2m})_\x$ where 
	\[
	\x\in\{ \namelink{lin}, \namelink{twlin1},  \namelink{twlin2}, \namelink{gal1},\namelink{gal2}\}
	\] 
	is one of the cases defined in Section \ref{sss the families}. Let $a\in \N$ and $\D$ be defined by \eqref{eq a&D} after Remark \ref{rem arithmetic vs geometric}, so that $G(F)=\GL_{2a}(\D)$. 	Let $P = M U$ be the standard parabolic $F$-subgroup of $G $ with $M =M_{(a,a)}.$ 
	
	In this section we compute the $H$-period of a truncated Eisenstein series on $G(\BA)$ induced from a maximal parabolic subgroup of type $(a,a)$, following \cite{MR1625060}, \cite{MR2010737} and \cite{MR4411860}. 
	
\subsection{Vanishing of linear periods}
This following result is documented in the literature when $\D=F$, however, its proof easily generalizes to inner forms of general linear groups.
\begin{lemma}\label{lem vanishing linear}
Let $k=a+b$ with $a,b\in \Z_{\ge 0}$, $H=M_{(a,b)}\simeq G_\D(a)\times G_\D(b)$ and $\chi$ an automorphic character of $H(\A)$. If $a\ne b$ then 
\[
\int_{H(F)\bs H(\A)\cap G_\D(k,\A)^1}\chi(h)\phi(h)\ dh=0
\]
for any cusp form $\phi$ on $G_\D(k,\A)$.
\end{lemma}
\begin{proof}
This follows from \cite[Proposition 6.2]{MR1241129} if $\D=F$, however, the argument is based on Fourier inversion on $\A$ and generalizes to our setting by using Fourier inversion on $\D\otimes_F \A$ instead.
\end{proof}
	\subsection{Induced representations and Eisenstein series}\label{sss global induced and Eis}
	
Let $A_P^+=\Res_{F/\mathbb{Q}}(A_P)(\R_{>0})$, naturally, a subgroup of the center of $M(\A)$. Let $\sigma$ be an irreducible, cuspdial automorphic representation of $M(\BA)$ with a central character trivial on $A_P^{+}$. Let $I_P^G(\sigma)$ be the space of functions \[\varphi : U(\BA) M(F) \backslash G(\BA)  \ra \C\] such that
	\begin{align*}
		m \mapsto \varphi[g](m) := \delta_P^{-1/2} (m) \varphi(mg)
	\end{align*}
	lies in the space of $\sigma$ for all $g \in G(\BA)$. Note that for $x,m \in M(\BA)$ and $u \in U(\BA)$ we have
	\begin{align*}
		\varphi [xug] (m) = \delta_P^{-1/2}(m) \varphi (m xug) = \delta_P^{1/2} (x) \varphi [g] (mx) = \delta_P^{1/2}(x) (\sigma(x) (\varphi [g])) (m).
	\end{align*}
	That is, we can realize the normalized parabolic induction from $\sigma$ as a representation on $I_P^G (\sigma)$.
	
	Since $P$ is a maximal parabolic subgroup of $G$, let $\varpi_P \in (\fra_P^G)^*$ be the corresponding fundamental weight. For $s \in \C$ and $\varphi \in I_P (\sigma)$ set
	\begin{align*}
		\varphi_s (g) = e^{ \langle s \varpi_P, H_M(g) \rangle} \varphi(g).
	\end{align*}
	Let $I_P^G (\sigma,s)$ be the representation of $G(\BA)$ on the space $I_P^G(\sigma)$ given by 
	\begin{align*}
		(I_P^G (g,\sigma,s) \varphi)_s (x)  = \varphi_s (xg), \quad g,\,x \in G(\BA).
	\end{align*}
	The Eisenstein series $E(\varphi,s)$ is the meromorphic continuation of the series
	\begin{align*}
		E(g,\varphi,s)  = \sum_{ \gamma \in P \backslash G } \varphi_s (\gamma g),
	\end{align*}
	which is convergent when $\Re s  \gg 0$. 
	
	Let $w = w_M = \left( \begin{smallmatrix}
		& I_a \\
		I_a & 
	\end{smallmatrix}\right)$. The standard intertwining operator \[M(s) : I_P^G (\sigma,s) \ra I_P^G (w \sigma,-s)\] is defined by the integral
	\begin{align*}
		M(s)\varphi (g) = e^{\langle s\varpi_P, H_M(g) \rangle}  \int_{U(\BA)} \varphi_s (w^{-1} ug) du.
	\end{align*}
	The integral converges absolutely when $\Re s \gg 0$ and admits a meromorphic continuation to all $s \in \C$. We have the functional equation
	\begin{align}\label{formula::FE+EisensteinSeries}
		E(M(s)\varphi, -s) = E(\varphi,s).
	\end{align}
	Recall that the constant term $E_Q(\cdot,\varphi,s)$ of $E(\cdot ,\varphi,s)$ along a parabolic subgroup $Q= LV$ is defined by
	\begin{align*}
		E_Q(g,\varphi,s) = \int_{V(F) \backslash V(\BA)} E(vg,\varphi,s) dv.
	\end{align*}
	The constant terms of $E(\cdot , \varphi,s)$ are computed in \cite[II. 1.7]{MR1361168}.  We have
	\begin{align}\label{formula::cons-term+Eis-series}
		E_P(g,\varphi,s) = \varphi (g) e^{ \langle s \varpi_P, H_M(g) \rangle}  + e^{ \langle -s \varpi_P, H_M(g) \rangle} M(s)\varphi (g), \quad g \in G(\BA).
	\end{align}
	For any other proper standard parabolic subgroup $Q$ of $G$, we have
	\begin{align}\label{formula::cons-term+Eis-series-II}
	 	E_Q (g, \varphi,s) = 0.
	\end{align}
	
	\subsection{Regularized periods of Eisenstein series}\label{sec::reg-periods}
	
	In \cite{MR4411860}, Zydor defined a relative truncation operator, denoted by $\Lambda^{T,H}$, from functions on $G(F)\bs G(\A)$ to functions on $H(F)\bs H(\A)$ where $T \in \fra_{0,H} := \fra_{P_0^H}$ where $P_0^H=P_0\cap H$ is the minimal parabolic subgroup of $H$ consisting of upper triangular matrices. The truncation depends on the choice of a good maximal compact subgroup $K_H$ of $H(\BA)$. For sufficiently positive $T \in \fra_{0,H}$, the truncation operator $\Lambda^{T,H}$ carries automorphic forms on $G$ to functions of rapid decay on $H(F) \backslash H(\A)^{1,G}$ where $H(\A)^{1,G}=H(\BA)\cap G(\BA)^1$.
	
Zydor's truncation is expressed as a sum over semi-standard parabolic subgroups of $G$ that contain $P_0^H$. Recall that a semi-standard parabolic subgroup of $G$ is of the form $\varsigma(Q)=\varsigma Q\varsigma^{-1}$ for a unique standard parabolic subgroup $Q=LV$ of $G$ and $\varsigma\in W$ determined uniquely modulo $W_L$. Furthermore, for an automorphic form $\phi$ on $G$ the constant term satisfies
\begin{equation}\label{eq ss const term}
\phi_{\varsigma(Q)}(g)=\phi_Q(\varsigma^{-1}g),\ \ \ g\in G(\BA).
\end{equation}
Let $\Sigma_P$ be a set of representatives in $W/W_M$ for Weyl elements $\varsigma\in W$ such that $P_0^H\subseteq \varsigma(P)$.

In what follows we maintain the notation introduced in Section \ref{sss global induced and Eis}.	By the definition of the truncation operator $\Lambda^{T,H}$ in \cite[Section 3.7]{MR4411860} combined with \eqref{formula::cons-term+Eis-series-II} and \eqref{eq ss const term} we have
	\begin{align}\label{formula::inversion+Eis-series}
		E(h,\varphi, s ) = \Lambda^{T,H} E(h,\varphi,s) +  \sum_{\varsigma\in \Sigma_P}\sum_{\gamma \in \varsigma(P)_H \backslash H} E_P(\varsigma^{-1}\gamma h,\varphi,s)\hat{\tau}_{\varsigma(P)}(H_{P_0^H}(\gamma h)_P^G - T_P^G),
	\end{align}
	where $\hat{\tau}_P(\cdot )$ is the characteristic function of the relative interior of the cone
	\begin{align*}
		\{ X \in \fra_{0,H} \mid \langle X, z \rangle \geqslant 0,\quad \forall z \in \overline{\fra_P^+} \}
	\end{align*} 	
and $\hat\tau_{\varsigma(P)}=\hat\tau_P\circ \varsigma^{-1}$.
	Note that there are finitely many non-zero terms in the summation over $\gamma$ in \eqref{formula::inversion+Eis-series}, as explained in \cite[Section 3.7]{MR4411860}.  
	 In \cite[Theorem 4.1]{MR4411860}, Zydor defined the regularized period $\CP_H(\phi)$ of an automorphic form $\phi$ of $G(\A)$ that is $H$-regular (see \cite[Section 4.5]{MR4411860} for the definition, it amounts to avoiding certain closed conditions on the exponents of $\phi$). The Eisenstein series $E(\varphi,s)$ above is $H$-regular for almost all $s$ and $\CP_H(E(\varphi,s))$ is a meromorphic function of $s$ (see \cite[Theorem 8.4.1 (4)]{MR2010737}). The following formula is a consequence of \cite[Corollary 4.2]{MR4411860}.
	For a subgroup $M_1$ of $M$ set 
	\[
	M_1(\A)^{1,M}=\{\diag(g_1,g_2)\in M_1(\A): g_1,g_2\in G_\D(a,\A)^1\}. 
	\]
	\begin{lemma}\label{lem::regularized-period+Eis}
		For a sufficiently positive $T \in \fra_{0,H}$, set $t = \langle \varpi_P, T_P^G \rangle$. We have
		\begin{align*}
			\CP_H ( E(\varphi,s)) = & \int_{H(F) \backslash H(\BA)^{1,G}} \Lambda^{T,H} E(h,\varphi,s) dh - \frac{e^{st}}{s} \int_{K_H} \int_{M_H \backslash M_H(\BA)^{1,M}} \varphi(mk) dmdk \\
			& + \frac{e^{-st}}{s} \int_{K_H} \int_{M_H \backslash M_H(\BA)^1} (M(s)\varphi)(mk) dmdk.
		\end{align*}
		In particular, the right-hand side of the identity is independent of $T$.
	\end{lemma}
	\begin{proof}
Explicating \cite[(4.3)]{MR4411860} we have
	\begin{align*}
			\CP_H ( E(\varphi,s)) = & \int_{H(F) \backslash H(\BA)^{1,G}} \Lambda^{T,H} E(h,\varphi,s) dh \\& -\sum_{\sigma\in\Sigma_P}\frac{e^{(z_\sigma+s)t}}{z_\sigma+s} \int_{K_H}  \int_{M_{\sigma,H} \backslash M_{\sigma,H}(\BA)^{1,M}}\delta^{-1}_{P\cap \sigma^{-1}H\sigma}(m)\varphi(m\sigma^{-1}k) dmdk    \\
			& -\sum_{\sigma\in\Sigma_P} \frac{e^{(z_\sigma-s)t}}{z_\sigma-s} \int_{K_H} \int_{M_{\sigma,H} \backslash M_{\sigma,H}(\BA)^{1,M}}\delta^{-1}_{P\cap \sigma^{-1}H\sigma}(m) (M(s)\varphi)(m\sigma^{-1}k) dmdk,
		\end{align*}
where $M_{\sigma,H}=M\cap \sigma^{-1}H\sigma$ and  $\rho_P-2(\rho_{P\cap \sigma^{-1}H\sigma})_P=z_\sigma\varpi_P$. In all cases, except case \namelink{lin}, the set $\Sigma_P=\{e\}$ is a singleton, $\delta_{P_H}$ is trivial on $M_H(\A)^{1,M}=M_{e,H}(\A)^{1,M}$ and $z_e=0$ so that the lemma follows. In case \namelink{lin} (where the permutation $s_m$ is also defined), the set $\Sigma_P$ consists of $a+1$ elements. As representatives we may choose $\Sigma_P=\{\sigma_i:i\in[0,a]\}$ where
\[
\sigma_i=s_{2m}\begin{pmatrix} I_i & & & \\ & & I_{a-i} & \\ & I_{a-i} & & \\ & & & I_i \end{pmatrix}\diag(s_m,s_m)^{-1}.
\] 
Then $\diag(s_m,s_m)M_{\sigma_i,H}\diag(s_m,s_m)^{-1}=(G_\D(i) \times G_\D(a-i)) \times (G_\D(a-i) \times G_\D(i))$. If $i\ne a-i$ then the inner period integral over $M_{\sigma_i,H}$ vanishes on cusp forms by Lemma \ref{lem vanishing linear}.
 If $i=a-i$ then $i=m$.
We observe that 
	\begin{multline*}
\diag(s_m,s_m)^{-1}(P\cap \sigma_m^{-1}H\sigma_m)(\A)\diag(s_m,s_m)^{-1}=\\\{\begin{pmatrix} g_1 & & x& \\ & g_2 & & y \\ & & g_3 & \\ & & & g_4
\end{pmatrix}: g_i\in G_\D(m,\A), \ i=1,2,3,4,\ x,y\in \M_m(\D\otimes_F \A)\}
	\end{multline*} 
from which it easily follows that $\delta_{P\cap \sigma_m^{-1}H\sigma_m}$ is trivial on $M_{\sigma_m^{-1},H}(\A)^{1,M}$ and $z_{\sigma_m}=0$. Since $M_{\sigma_m^{-1},H}=M_H$ the lemma follows.
	\end{proof}
	
	In the remainder of this section, we will compute the regularized period of Eisenstein series in terms of intertwining periods, following the arguments in \cite{MR2010737}.
	
	\subsection{Admissible double cosets inside $P \backslash G /H$}\label{ss explicit adm orbits}
	
The discussion in Section \ref{ss orbits} applies equally well to the number field situation. From this we make the following explicit choices of representatives of the $P$-admissible $P$-orbits in $G\cdot e$:

\begin{enumerate}
\item In all cases the unique open $P$-orbit in $G\cdot e$ is $P\cdot w$ and we have 
\[
P_w=M_w=\{\diag(g,\theta(g)): g\in G_\D(a)\}. 
\]
\item In case \namelink{lin}, the other $P$-admissible orbits are parameterized by the integer interval $[0,2m]$. These are the orbits $P\cdot x_j$ where we set $x_m=e$ and
\[x_j=\diag(I_j,-I_{2m-j},I_{2m-j},-I_j)[\nu_0]_{2m},\ \ \ m\ne j\in [0,2m],\ \ \ \text{where}\ \ \ \nu_0=\diag(1,-1).
\]
We have $P_{x_j}=M_{x_j}U_{x_j}$ and $M_{x_m}=M_H$ and $M_{x_j}=M_{(j,2m-j,2m-j,j)}$, $j\ne m$.

\item In cases \namelink{twlin1}, \namelink{twlin2},\namelink{gal1}, \namelink{gal2} the closed orbit $P\cdot e$ is the only $P$-admissible orbit in $G\cdot e \setminus P\cdot w$.

\end{enumerate}

	\subsection{The open intertwining periods}
	
Let
	\begin{align*}
		H_{\eta}^P = H \cap \eta^{-1} P \eta  = \eta^{-1} P_w\eta. 
	\end{align*}
	\begin{lemma}\label{lem::convergence+intt-period}
		There exists $s_0 > 0$ such that whenever $\Re s > s_0$ we have
		\begin{align*}
			\int_{H_{\eta}^P (\BA) \backslash H(\BA)} e^{\langle s \varpi_P, H_M( \eta h) \rangle} dh  < \infty.
		\end{align*}
	\end{lemma}
	\begin{proof}
		The argument is a straightforward adaptation of \cite[Lemma 27]{MR1625060}, as in the proof of \cite[Lemma 3.2]{MR4721777} (see also the proof of \cite[Proposition 4.5]{MR3719517}). It relies on expressing the spherical vector in the integrand as an integral of a Schwartz function, and using the basic properties of the Godement-Jacquet L-functions.
	\end{proof}

	Since $P_w=M_w$ and $\fra_{M_w}=\fra_M^G$, in the notation of Section \ref{sss global induced and Eis} for $\varphi \in I_P^G (\sigma)$ the function 
	\[
	g\mapsto \int_{M_w(F)\bs M_w(\A)^1}\varphi_s(mg)\ dm
	\]
	on $G(\A)$ is left $M_w(\A)$ invariant. Whenever convergent, define the open intertwining period
	\begin{align}\label{eq def global J}
		J_\sigma(\varphi, s) &= \int_{M_w(\BA) \backslash G_w(\BA)} \int_{ M_w(F) \backslash M_w(\BA)^1 } \varphi_s ( mg \eta ) \ dm \ dg
					\end{align}
					\begin{align*}
					& &= \int_{ H_{\eta}^P(\BA) \backslash H(\BA)} \int_{ M_w(F) \backslash M_w(\BA)^1 } \varphi_s ( m \eta h) \ dm \ dh. 
\end{align*}
Since elements of $I_P^G (\sigma)$ are bounded and $M_w(F) \backslash M_w(\BA)^1 $ has finite volume it follows from Lemma \ref{lem::convergence+intt-period} that there exists $t>0$ such that $J_\sigma(\varphi,s)$ is defined by an absolutely convergent integral for $\Re(s)>t$ and is holomorphic in $s$ in this domain. 
Furthermore, for such $s$, the linear form $J_\sigma(s)$ on $I_P^G(\sigma,s)$ is $H(\A)$-invariant. It is not identically zero if and only if the inner period integral is non-vanishing, that is, if and only if $\sigma=\pi_1\otimes \pi_2$ where $\pi_2=\pi_1^*$.	For the rest of this section we maintain the notation of Section \ref{sss global induced and Eis}  and fix $t$ as above once and for all.
	
	\subsection{Periods of pseudo Eisenstein series}
	
	Consider a test function $f$ in the Paley-Wiener space of $\C$. That is, $\hat{f} \in C_c^{\infty} (\BR)$ where
	\begin{align*}
		\hat{f} (t)  =  \int_{\Re s = s_0} e^{st} f(s) ds
	\end{align*}
	is independent of $s_0 \gg 0$. In what follows we always assume that $s_0>t$. For $\varphi \in I_P^G (\sigma)$ let
	\begin{align*}
		\CE (g, f, \varphi) = \int_{\Re s = s_0} f(s) E(g,\varphi,s) ds = \sum_{ \gamma \in P \backslash G } \hat{f}( t (\gamma g)) \varphi (\gamma g)
	\end{align*}
	be the associated pseudo Eisenstein series (independent of $s_0 \gg 0$), where $t(\gamma g) : = \langle \varpi_P , H_M(\gamma g) \rangle$. It is an automorphic function of rapid decay so that its $H$-period integral is absolutely convergent. 

For a subgroup $Q$ of $G$ write 
\[Q(\A)^{1,G}=Q(\A)\cap G(\A)^1.
\]	
	\begin{lemma}\label{lem::period+pseudo-Eis}
	For $\varphi\in I_P^G(\sigma)$ we have:
		\begin{align*}
			\int_{H(F) \backslash H (\BA)^{1,G}} \CE ( h, f, \varphi) = \int_{\Re s = s_0} f(s) J(\varphi,s) ds + f(0) \int_{K_H} \int_{M_H(F) \backslash M_H(\BA)^{1,M}} \varphi (m  k)\ dm\ dk.
		\end{align*}
	\end{lemma}
	\begin{proof}
We carry out the standard unfolding of the pseudo-Eisenstein series by summing over $P \backslash G$ along representatives $\delta$ of the double coset space $P \backslash G / H$. Let $H_{\delta}^P = H \cap \delta^{-1} P \delta $. We have
		\begin{align*}
			\int_{H(F) \backslash H (\BA)^{1,G}} \CE (h,f,\varphi) dh = \sum_{\delta} \int_{H_{\delta}^P(F) \backslash H(\BA)^{1,G}} \int_{\Re s = s_0} f(s) \varphi_s (\delta h) \ ds\ dh.
		\end{align*}
		Fix a double coset $P \delta H$ and write $x = \delta \theta(\delta)^{-1}$. According to \cite[Section 3]{MR3541705}, there is a unique $\xi \in {}_M W_M w_{\star}$ such that $P x P = P \xi P $ and $L = M \cap \xi M \xi^{-1}$ is a standard Levi subgroup of $M$. By (the global analog of) \cite[Lemma 3.2]{MR3541705} the representative $\delta$ can be chosen so that $ x \in L \xi$. We have that
		\begin{align*}
			&\int_{ H_{\delta}^P(F) \backslash H(\BA)^{1,G}}  \int_{\Re s = s_0} f(s) \varphi_s (\delta h) \ ds\ dh \\
			& = \int_{ H_{\delta}^P(\BA)^{1,G}\backslash H(\BA)^{1,G}} \int_{H_{\delta}^P(F)  \backslash H_{\delta}^P(\BA)^{1,G}}  \int_{\Re s = s_0} \delta_{H_{\delta}^P} (h_1)^{-1} f(s) \varphi_s (  \delta h_1h) \ ds\ dh_1\ dh
			\\
			& = \int_{ P_x(\A)^{1,G} \backslash G_x(\BA)^{1,G}} \int_{P_x(F) \backslash P_x(\BA)^{1,G}}  \int_{\Re s = s_0} \delta_{P_x} (g)^{-1} f(s) \varphi_s (g_1g \delta )\ ds\ dg_1\ dg.
		\end{align*}
		The first identity applies integration in stages and the second the variable change $h_1h\mapsto \delta^{-1} h_1h \delta$.
		By \cite[Lemma 3.3]{MR3541705}, $P_x$ admits a Levi decomposition $P_x = L_x \ltimes R$ with unipotent radical $R$, we have
		\begin{align*}
			&\int_{P_x(F) \backslash P_x(\BA)^{1,G}} \int_{\Re s = s_0} f(s) \varphi_s (p g \delta h) \ ds\ dp \\
			&= \int_{L_x(F) \backslash L_x(\BA)^{1,G}}\int_{\Re s = s_0} f(s)  \int_{R(F) \backslash R(\BA)} \varphi_s (r m g \delta)\ dr\ ds\ dm
		\end{align*}
		and as in \cite[Proposition 4.2.2]{MR2010737} we deduce from cuspidality of $\sigma$ that
		\begin{align*}
			\int_{R(F) \backslash R(\BA)} \varphi_s ( r m g \delta) dr = 0
		\end{align*}
		unless $L = M$. That is, the summand associated to $\delta$ only contributes if $x$ is $P$-admissible.
		The $P$-admissible orbits are explicated in Section \ref{ss explicit adm orbits} and we choose $\delta$ so that $x$ is in the explicated list. That is, either $x\in \{e,w\}$ or $x=x_j$ in case \namelink{lin}. In particular, we see that $M_x(\A)$ contains the center of $G(\A)$ whenever $x$ is $P$-admissible and the outer integration over $P_x(\A)^{1,G} \backslash G_x(\BA)^{1,G}$ can be replaced by integration over $P_x(\A) \backslash G_x(\BA)$. 
		
		 For the open orbit with $x=w$ (that is, $\delta=\eta$) we have $P_w(\BA)^1 = M_w(\BA)^1$. Applying the convergence in Lemma \ref{lem::convergence+intt-period} and Fubini's theorem to change order of integration, the term associated to it equals 
		\begin{align*}
			\int_{\Re s = s_0} f(s) J_\sigma(\varphi,s) ds.
		\end{align*}
For the other orbits we have that $R\subseteq U$ ao that with respect to the probability measure
\[
\int_{R(F) \backslash R(\BA)} \varphi_s ( r m g \delta) dr = \varphi_s (m g \delta).
\]
Furthermore, $A_M$ is contained in $M_x$ and therefore $M_x(\A)^{1,G}=(A_M^+)^{1,G} M_x(\A)^{1,M}$. Recall that by assumption, the central character of $\sigma$ is trivial on $A_M^+$. Integrating in stages and applying the Fourier inversion formula for $f$ we have
\[
\int_{M_x(F) \backslash M_x(\BA)^{1,G}}\int_{\Re s = s_0} f(s)  \int_{R(F) \backslash R(\BA)} \varphi_s (r m g \delta)\ dr\ ds\ dm=f(0)\int_{M_x(F) \backslash M_x(\BA)^{1,M}}\varphi(mg\delta)\ dm.
\]
If $x\ne e$ (in particular, in case \namelink{lin}) it follows from Lemma \ref{lem vanishing linear} that the period integral on the right hand side vanishes. It remains to observe that $M_e(\A)^{1,M}=M_H(\A)^1$ and that $\delta_P^{\frac12}|_{P_e(\A)}=\delta_{P_e}$ so that by the Iwasawa decomposition, the integration over $P_e(\A)\bs G_e(\A)$ can be realized as integration over $K_H$.
		

	\end{proof}
	
	Next, we would like to be able to interchange the order of integration in the iterated integral
	\begin{align*}
		\int_{H(F) \backslash H(\BA)^1} \int_{\Re s= s_0} f(s) E(h,\varphi,s) \ ds\ dh.
	\end{align*}
	Since the period integral of an Eisenstein series does not converge this can only be achieved via its regularization.
	\begin{lemma}\label{lem::pseudo==period-DistributionFormula}
		Suppose that $f(0) = 0$. Then for sufficiently large $s_0$, we have
		\begin{align}\label{formula::pseudo=period-DistributionFormula}
			\int_{H(F) \backslash H (\BA)^{1,G}} \CE(h,f,\varphi) dh = \int_{\Re s = s_0} f(s) \CP_H (E(\varphi,s)) ds.
		\end{align}
	\end{lemma}
	\begin{proof}
		We follow closely the argument of \cite[Lemma 9.1.1]{MR2010737}. By definition,
		\begin{align*}
			\int_{H(F) \backslash H (\BA)^{1,G}} \CE(h,f,\varphi) dh  =  \int_{H(F) \backslash H(\BA)^1} \int_{\Re s= s_0} f(s) E(h,\varphi,s) \ ds \ dh.
		\end{align*}
		By the inversion formula \eqref{formula::inversion+Eis-series} and the constant term formula \eqref{formula::cons-term+Eis-series}, the right hand side of the above identity is equal to $I_1+I_2+I_3$ where
		\begin{align*}
			&I_1= \int_{H \backslash H(\BA)^{1,G}} \int_{\Re s= s_0} f(s) \Lambda^{T,H}E(h,\varphi,s) \ ds\ dh, \\ 
			& I_2=\int_{P_H(F) \backslash H (\BA)^{1,G}} \int_{\Re s= s_0} f(s)  e^{ \langle s \varpi_P, H_M(h) \rangle}  \varphi(h) \hat{\tau}_P (H_{P_0^H} (h)_P^G - T_P^G )\ ds\ dh,  \\
			&I_3= \int_{P_H(F) \backslash H (\BA)^{1,G}} \int_{\Re s= s_0} f(s)  e^{- \langle s \varpi_P, H_M(h) \rangle}  M(s)\varphi(h) \hat{\tau}_P (H_{P_0^H} (h)_P^G - T_P^G )\ ds\ dh. 
		\end{align*}
		By the property of the truncation operator, for  $s_0 \gg 0$, $\Lambda^{T,H} E(h,\varphi,s)$ is rapidly decreasing in $h$ and the integral
		\begin{align*}
			\int_{H(F) \backslash H (\BA)^{1,G}} \Lambda^{T,H}E(h,\varphi,s)\ dh
		\end{align*}
		is rapidly decreasing in the imaginary part of $s$. Hence we can interchange the order of integration and obtain that
		\[
		I_1= \int_{\Re s= s_0} f(s) \int_{H \backslash H(\BA)^{1,G}} \Lambda^{T,H}E(h,\varphi,s) \ dh\ ds.
		\]
Note that $A_M$ is contained in $P_H$. For $I_2$ and $I_3$ we perform the outer integral over $P_H(F)\bs H(\A)^{1,G}$ in stages, over $P_H(A)^{1,G}\bs H(\A)^{1,G}$ and $P_H(F)\bs P_H(\A)^{1,G}$. By the Iwasawa decomposition, the first is realized as integration over $K_H$ while for the second, the integrand is independent of $U_H(\A)$. We obtain that
		\begin{align*}
			I_2=\int_{K_H} \int_{M_H \backslash M_H(\BA)^{1,M}} \int_{\fra_P^G} \int_{\Re s= s_0} f(s)  e^{ \langle s \varpi_P, X \rangle}  \varphi(mk) \hat{\tau}_P ( X - T_P^G)\ ds\ dX\ dm\ dk 
		\end{align*}
		and similarly
		\begin{align*}
			I_3=\int_{K_H} \int_{M_H \backslash M_H(\BA)^{1,M}} \int_{\fra_P^G} \int_{\Re s= s_0} f(s)  e^{- \langle s \varpi_P, X \rangle}  M(s)\varphi(mk) \hat{\tau}_P ( X - T_P^G)\ ds\ dX\ dm\ dk. 
		\end{align*}
		
		We identify $\fra_P^G$ with $\BR$ by letting the fundamental coweight dual to $\varpi_P$ correspond to $1$ and denote by $t>0$ the image of $T_P^G$ in $\BR$. Then
		\begin{align*}
			\int_{\fra_P^G} \int_{\Re s= s_0} f(s)  e^{ \langle s \varpi_P, X \rangle}  \varphi(mk) \hat{\tau}_P ( X - T_P^G)\ ds\ dX=\int_t^{\infty} \int_{\Re s= s_0} f(s) e^{sx} \ ds\ dx.
		\end{align*}
		As $f$ is holomorphic, we first shift the integral over $\Re s = s_0$ to $\Re s = s_1$ with $s_1 < 0$. The resulting double integral is absolutely convergent and is equal to 
		\begin{align*}
			- \int_{\Re s = s_1} f(s) \frac{e^{st}}{s} ds.
		\end{align*}
		As $f(0) = 0$, the function $f(s)/s$ is also holomorphic and hence we can shift the integral back to $\Re s = s_0$. Therefore,
		\begin{align*}
			I_2=- \int_{\Re s= s_0}  f(s) \frac{e^{st}}{s} \int_{K_H} \int_{M_H \backslash M_H(\BA)^{1,M}}      \varphi(mk)  \ dm\ dk \ ds. 
		\end{align*}
		The computation of $I_3$ is similar but simpler. In this case the integral is absolutely convergent and there is no shift of contour. After changing the order of integration we get that 
		\begin{align*}
			I_3=\int_{\Re s= s_0}f(s) \int_{K_H} \int_{M_H \backslash M_H(\BA)^{1,M}}     \frac{e^{-st}}{s}    (M(s)\varphi) (mk)  \ dm\ dk\ ds.
		\end{align*}
		The lemma then follows from Lemma \ref{lem::regularized-period+Eis}.
	\end{proof}
	
	\begin{corollary}\label{cor::period==intt-period}
		We have
		\begin{align*}
			\CP_H(E(\varphi,s)) = J_\sigma(\varphi,s).
		\end{align*}
In particular, $J_\sigma(s)$ admits a meromorphic continuation to $\C$.
	\end{corollary}
	\begin{proof}
		Combining Lemmas \ref{lem::period+pseudo-Eis} and \ref{lem::pseudo==period-DistributionFormula}, for any Paley-Wiener function $f$ on $\C$ such that $f(0) = 0$ and for $s_0 \gg 0$ we have 
		\begin{align*}
			\int_{\Re s = s_0} f(s) \CP_H (E(\varphi,s)) ds = \int_{\Re s =s_0} f(s) J_\sigma(\varphi,s)ds. 
		\end{align*}
		The corollary now follows from a simple distributional density argument \cite[Lemma 9.1.2]{MR2010737} .
	\end{proof}
	
	\begin{corollary}\label{cor:FEof GIP}
		We have the functional equation
		\begin{align}
			J_\sigma(\varphi,s)  =  J_{w\sigma}(M(s)\varphi,-s).
		\end{align}
	\end{corollary}
	\begin{proof}
		This follows from the functional equation \eqref{formula::FE+EisensteinSeries} of the Eisenstein series and Corollary \ref{cor::period==intt-period}.
	\end{proof}

	\subsection{The Maass-Selberg relations}\label{sec::maass-selberg}
	
	As a consequence of Lemma \ref{lem::regularized-period+Eis} and Corollary \ref{cor::period==intt-period} and in the notation of the lemma we have
	\begin{equation*} 
		\begin{aligned}
			\int_{H(F) \backslash H (\BA)^{1,G}} \Lambda^{T,H}E(h,\varphi,s) dh =& J_\sigma(\varphi,s) + \frac{e^{st}}{s} \int_{K_H} \int_{M_H \backslash M_H(\BA)^{1,M}} \varphi(mk)\ dm\ dk \\
			& - \frac{e^{-st}}{s} \int_{K_H} \int_{M_H \backslash M_H(\BA)^{1,M}} (M(s)\varphi)(mk)\ dm\ dk.  
		\end{aligned}
	\end{equation*}

	\section{Global distinction and poles of open gobal intertwining periods}
	
		We retain the notations from Section \ref{sec::Global}. 
In this section we deduce the following theorem from the Maass-Selberg relations. 
	
	\begin{theorem}\label{thm pole of GIP}
	Let $\pi\in \mathcal{C}^*(G_m(\BA))$ be such that $\pi=\pi^*$ and the central character of $\pi$ is trivial on $\R_{>0}I_m$. Then $\pi$ is distinguished if and only if the open intertwining period $J_{\sigma}(s)$ has a  pole at $s=0$. When this is the case, the pole is simple.
	\end{theorem}

	First, we make use of the holomorphicity of Eisenstein series at zero. 
	\begin{lemma}\label{lem::holomorphic-TruncPeriods}
		With the above notation for $\varphi\in I_P^G(\sigma)$ the function 
		\begin{align*}
			s\mapsto  \int_{H(F) \backslash H(\BA)^{1,G}} \Lambda^{T,H} ( E (h,\varphi,s)) dh 
		\end{align*}
	is holomorphic at $s=0$.
	\end{lemma}
	\begin{proof}
		The Eisenstein series $E(\varphi,s)$ is holomorphic at $s=0$ \cite[Proposition IV.1.11 (b)]{MR1361168} and the same is therefore true for  $\Lambda^{T,H} ( E (\varphi,s))$. As explained in the proof of \cite[Lemma 3.1]{MR650368}, for any closed compact curve $C$ in a sufficiently small neighborhood of $s=0$ the double integral
		\[
		\int_C\int_{H(F) \backslash H(\BA)^{1,G}} \Lambda^{T,H} ( E (h,\varphi,s)) \ dh \ ds
		\]
is absolutely convergent so that we can change order of integration and furthermore, the integral over $C$ further commutes with the truncation operator $\Lambda^{T,H}$. The lemma is therefore a consequence of Morera's criterion for holomorphicity.
		
	\end{proof}
	
	The following lemma is a key for our methods.
	\begin{lemma}\label{lem::Keys-Shahidi}
		Suppose that $\pi\in \mathcal{C}^*(G_m(\BA))$. Then $M(0) = -\Id$.
	\end{lemma}
	\begin{proof}
		The intertwining operator $M(s)$ is holomorphic at $s=0$ by \cite[Proposition IV.1.11 (b)]{MR1361168}. Since $\pi$ is unitary, the induced representation $I_P^G(\sigma)$ is irreducible and it follows that $M(0)$ acts as a scalar on $I_P^G(\sigma)$. For the special case where $D = F$, by \cite[Proposition 6.3]{MR944102}, $M(0) = (-1)^n$, where $n$ is the order of the pole of $L (s,\pi, \pi^{\vee})$ and we have $n=1$ by  \cite[Proposition 3.6]{MR623137}. The general case follows from \cite[Corollary 5.4]{MR2390289}, which asserts that the computations in \cite{MR944102} transfer to the case of inner forms.
	\end{proof}

	Let $Z_\sigma$ be the closed orbit linear form on $I_P^G(\sigma)$ defined by
	\[
	Z_\sigma(\varphi)= \int_{K_H} \int_{M_H \backslash M_H(\BA)^{1,M}}  \varphi(mk) \ dm\ dk.
	\]
	\begin{lemma}\label{lem closed JR}
	Let $\pi\in \mathcal{C}^*(G_m(\BA))$ be such that $\pi=\pi^*$ and the central character of $\pi$ is trivial on $\R_{>0}I_m$. The closed intertwining period	$Z_\sigma$ is not identically zero on $I_P^G(\sigma)$ if and only if $\pi$ is distinguished.
	\end{lemma}
	\begin{proof}
	Note that $M_H=H_m\times H_m$ and therefore the inner period integral is non-vanishing if and only if $\pi$ is distinguished. The lemma therefore follows from  the argument of Jacquet and Rallis in \cite[Proposition 2]{MR1142486}.
	\end{proof}
	\begin{proof}[Proof of Theorem \ref{thm pole of GIP}]
	It follows from Corollary \ref{cor::period==intt-period} that $J_\sigma(s)$ is meromorphic in $s$.
	Applying the Maass-Selberg relation in Section \ref{sec::maass-selberg} together with Lemmas \ref{lem::holomorphic-TruncPeriods} and \ref{lem::Keys-Shahidi} we deduce that
		\begin{align*}
			\lim_{s \ra 0} sJ_\sigma(\varphi,s)  = -2 Z_\sigma(\varphi).
		\end{align*}
The theorem is now a consequence of Lemma \ref{lem closed JR}.
	\end{proof}

	\section{The local-global principle and its consequences}\label{sec::LGp}
	
Assume that $F$ is a number field. Let $(G,H,\theta)=(G_{m},H_{m},\theta_{m})_\x$ for 
\[
\x\in  \{\namelink{lin},\namelink{twlin1},\namelink{twlin2},\namelink{gal1},\namelink{gal2}\}
\] 
be defined as in Section \ref{sss the families}. Let $a\in \N$ and $\D$ be defined by \eqref{eq a&D} after Remark \ref{rem arithmetic vs geometric}, so that $G(F)=\GL_{a}(\D)$.
We say that an irreducible cuspidal automorphic representation $\pi=\otimes'_v \pi_v$ is \emph{compatible} if $\pi_v$ is $H(F_v)$-compatible for every place $v$ of $F$ (see \ref{def compatible}).

We are now ready to formulate our main result, the local-global principle for distinction.	
	\begin{theorem}\label{thm::Main}
	For $\pi\in \mathcal{C}^*(G(\BA))$, the following assertions are equivalent:
	\begin{enumerate}[i)]
	\item \label{i} $\pi$ is distinguished,
	\item \label{ii} $\pi$ is locally distinguished and compatible and $\L(s,\pi,\theta)$ has a pole at $s=0$. 
	\end{enumerate}
	When the equivalent conditions are satisfied the pole is simple.
	\end{theorem}
	\begin{proof}
The last part of the theorem follows from Theorem \ref{thm non vanishing s=1}.
Note first that if either point \eqref{i} or point \eqref{ii} holds then $\pi$ is locally distinguished and it follows from Theorem \ref{thm mult 1 and ssduality} that $\pi_v\simeq \pi_v^*$ for all places $v$ of $F$.
By strong multiplicity one (\cite[Theorem 18.1]{MR2684298}) we conclude that $\pi=\pi^*$. Furthermore, if either point \eqref{i} or point \eqref{ii} holds then the central character of $\pi$ is trivial on $A_G^+$. We assume from now on that $\pi=\pi^*$ and that the central character of $\pi$ is trivial on $A_G^+$ (in particular, $\pi$ is unitary) and let $\sigma=\pi\otimes \pi$. 

It follows from Theorem \ref{thm pole of GIP} that \eqref{i} holds if and only if $J_{\sigma}(s)$ has a pole at $s=0$ and that when this is the case the pole is simple. 

Next, we observe that the meromorphic family $J_{\sigma}(s)$ of linear forms is factorizable. Indeed, in its definition in \eqref{eq def global J}, the outer integral is adelic and the inner period integral is the unique invariant pairing of $\pi$ and $\pi^\theta=\pi^\vee$. 
For $\varphi=\prod'_v \varphi_v\in \Ind_{P_{(a,a)}(\A)}^{G_{2m}(\A)}(\pi\otimes \pi)$ a factorizable section write $J_{\sigma}(\varphi;s)=\prod_v J_v(\varphi_v;s)$. 

In order to identify the local factors $J_v$ in terms of the local intertwining periods $\J_{\sigma_v}$ studied in Section \ref{sec local intper} we go back to the observations in Section \ref{sss loc-glob triples} and apply its notation. Write $(G',H',\theta')=(G_{2m},H_{2m},\theta_{2m})_\x$ for the global triple defined with respect to $F,E,D$. For a place $v$ of $F$ let $g_v\in G(F_v)$ satisfy \eqref{eq shifted triple}. Then $y_v=\diag(g_v,g_v)\in M_{(a,a)}(\A)$ is such that
\[(G'_v,\Ad(y_v)(H'_v), \Ad(y_v\theta_v'(y_v)^{-1})\circ \theta_v')=(G'',H'',\theta'')
\] 
where the right hand side is the local triple $(G'',H'',\theta'')=(G_{2n},H_{2n},\theta_{2n})_{\x_v}$ defined with data $F_v,E_v,R_v$. Let $z_v=g_v\theta_v(g_v)^{-1}$ so that $y_v\theta_v'(y_v)^{-1}=\diag(z_v,z_v)$, recall that $w=\begin{pmatrix} & I_a \\ I_a & \end{pmatrix}$ and note that $\diag(z_v,I_a)\cdot w=w\diag(z_v,z_v)$ and therefore 
\[
\ell\mapsto \ell\circ (\pi(z_v)\otimes \Id_{\pi_v})\in \Hom_{M^{\theta''_w}(F_v)}(\pi\otimes \pi,\triv)\rightarrow \Hom_{M^{\theta_w'}(F_v)}(\pi\otimes \pi,\triv)
\]
is an isomorphism of the one dimensional $\Hom$ spaces. Note further that $M^{\theta_w'}=M_w$ in the notation of \eqref{eq def global J}.
Consequently, the period integral
\[
b(\phi)=\int_{M_w(F)\bs M_w(\A)^1}\phi(m)\ dm,\ \ \ \phi\in \pi\otimes \pi
\]
has a factorization of the form $b=\otimes'_v \ell_{\sigma_v}\circ (\pi(z_v)\otimes \Id_{\pi_v})$ with $\ell_{\sigma_v}\in \Hom_{M_w(F_v)}(\pi_v\otimes \pi_v,\triv)$. We fix such a factorization to define $\J_{\sigma_v}(s)$ as in Section \ref{sec local intper}. By definition
\[
J_v(\varphi_v;s)=\int_{M^{\theta'_w}(F_v)\bs G^{\theta'_w}(F_v)}\ell_{\sigma_v}\circ (\pi(z_v)\otimes \Id_{\pi_v})[\varphi_{v,s}(h\eta)]\ dh
\]
where $\eta\in G'(F)$ is such that $\eta\cdot e=w$. After the variable change $h\mapsto \diag(z_v,I_a)^{-1} h\diag(z_v,I_a)$ this becomes
\[
\int_{M^{\theta''_w}(F_v)\bs G^{\theta''_w}(F_v)}\ell_{\sigma_v}[\varphi_{v,s}(h\diag(z_v,I_a)\eta)]\ dh.
\]
Since for $\xi=\diag(z_v,I_a)\eta y_v^{-1}$ we have $\xi \theta''(\xi)^{-1}=w$
we conclude that
\[
J_v(s)=\J_{\sigma_v}(s)\circ I_P^G(y_v,s)
\]
and in particular that $J_v(s)$ and $\J_{\sigma_v}(s)$ have the same order of pole at every point.
Taking the last part of Section \ref{sss loc-glob triples} into consideration, there is a finite set of places $S$ of $F$ containing the archimedean ones and such that $\pi_v$ is unramified, $g_v\in K_v$ and $\varphi_v$ is the normalized spherical section for all $v\not\in S$. Then $J_{\sigma}(\varphi;s)=J^S(\varphi^S;s)J_S(\varphi_S;s)$ where $J_S(\varphi_S;s)=\prod_{v\in S} J_v(\varphi_v;s)$.
It now follows from Proposition \ref{prop ur comp} that
\[
J^S(\varphi^S;s)=\frac{\L^S(s,\pi,\theta)}{\DL^S(s,\pi,\theta)}
\]
and therefore
\[
\Ord_{s=0}(J_{\sigma}(s))= \Ord_{s=0}(\L^S(s,\pi,\theta))-\Ord_{s=0}(\DL^S(s,\pi,\theta))+\Ord_{s=0}(J_S(s)).
\]
From Theorem \ref{thm order pole local L} \eqref{part denominatorL} we have that $\Ord_{s=0}(\DL^S(s,\pi,\theta))=\Ord_{s=0}(\DL(s,\pi,\theta))$.

Assume that \eqref{ii} holds. It follows from Theorem \ref{thm non vanishing s=1} that $\Ord_{s=0}(\L(s,\pi,\theta))=1$ and $\Ord_{s=0}(\DL(s,\pi,\theta))=0$.
It follows from Theorem \ref{thm mainloc} and Proposition \ref{prop above} that 
\[
\Ord_{s=0}(J_S(s))=\Ord_{s=0}(\L_S(s,\pi,\theta))
\]
and therefore $\Ord_{s=0}(J_{\sigma}(s))=\Ord_{s=0}(\L(s,\pi,\theta))=1$.
We conclude that \eqref{i} holds.

Assume now that \ref{i} holds. 
It similarly follows from the first part of Theorem \ref{thm mainloc} and Proposition \ref{prop above} that 
\[
1=\Ord_{s=0}(J_{\pi\otimes \pi}(s))\le  \Ord_{s=0}(\L(s,\pi,\theta))-\Ord_{s=0}(\DL(s,\pi,\theta)).
\]
It therefore follows from Theorem \ref{thm non vanishing s=1} that $\Ord_{s=0}(\L(s,\pi,\theta))=1$ and $\Ord_{s=0}(\DL(s,\pi,\theta))=0$. It now further follows from Theorem \ref{thm mainloc} that $\pi$ is compatible.
Point \eqref{ii} follows.
	\end{proof}
	
Next, we explain how to deduce our main results in the introduction from the local-global principle. 
The condition about a pole at $s=0$ transfers to $s=1$ by the functional equation, Theorem \ref{thm gfe LS} in all cases. We begin with the Galois case. 
\begin{proof}[Proof of Theorems \ref{thm main gal odd} and \ref{thm main gal even}]
In both theorems the equivalence of the first two conditions is Theorem \ref{thm::Main} and it is a tautology that the second implies the third condition.  By \cite{MR984899} and \cite{MR1344660}, it is known that $\jl(\pi)$ is distinguished if and only if $L(s,\jl(\pi),\as^+)$ has a pole at $s=1$ and that when this is the case the pole is simple. Consequently, the third condition implies that $\JL(\pi)$ is distinguished.
It further follows from Corollary \ref{cor distinction-transfer-generic} that if $\JL(\pi)$ is locally distinguished then so is $\pi$. Since global distinction implies local distinction, this shows that the third condition implies the second.
In light of Lemma \ref{lem compatible} it also shows that if $d$ is odd then the forth condition of Theorem \ref{thm main gal odd} implies the third. The two theorems follow.
\end{proof}
We proceed with the linear period case.
\begin{proof}[Proof of Theorem \ref{thm main lin}]
The proof is completely parallel to that of Theorems \ref{thm main gal odd} and \ref{thm main gal even}, and uses Corollary \ref{cor distinction-transfer-generic} and Lemma \ref{lem compatible} in the exact same manner. The results of Flicker and Flicker-Zinoviev in the split case, are replaced by the results of Friedberg-Jacquet as discussed at the end of Section \ref{ss main lin}.
\end{proof}

Finally we discuss the case of twisted linear periods. 
\begin{proof}[Proof of Theorem \ref{thm main twlin}]
The Theorem is immediate from Theorem \ref{thm::Main}.
\end{proof}

As a final application, we prove a partial converse theorem to the so-called Guo-Jacquet conjecture, using our local global principle. 

\section{A partial converse for the Guo-Jacquet Conjecture}\label{sec::GJconverse}
In this section we prove a weak, yet new, form of the converse implication of the Guo-Jacquet conjecture.
\subsection{Local preparation: the $\epsilon$-dichotomy conjecture}
Let $F$ be a local field of charactersitic zero and $E/F$ be a quadratic \'etale algebra.
Let $\psi$ be a non-trivial character of $F$. 

We recall that $\eta_{E/F}$ designates the quadratic character of $F^\times$ attached to $E/F$ by local class field theory when $E/F$ is a field extension, whereas by convention $\eta_{E/F}$ is defined as the trivial character of $F^\times$ when $E\simeq F\times F$. 

Let $\D$ be a central division algebra of degree $d$ over $F$. Let $m$ be an integer and assume further that $m$ is even if either $d$ is odd or $E\simeq F\times F$. 
If $d$ is even and $E$ is a field then it imbeds in $\D$ and we denote by $\mathcal{C}$ the centralizer of $E$ in $\D$ and set $H=\GL_m(\mathcal{C})$. Otherwise, $m$ is even and we set $H=\GL_{m/2}(\D_E)$.
We recall the following $\epsilon$-dichotomy theorem. When $E$ is a field it was formulated by Prasad and Takloo-Bighash as a conjecture in \cite{MR2806111}. 

\begin{theorem}[$\epsilon$-dichotomy]\label{thm C1}
For an irreducible, essentially square-integrable representation $\delta$ of $\GL_m(\D)$ we have that $\delta$ is $H$-distinguished if and only if the following two conditions are satisfied:
\begin{enumerate}
\item $\JL(\delta)$ is symplectic, that is, its Langlands parameter preserves a symplectic form;
\item the following equality holds: 
\begin{equation}\label{eq eps dich}
\epsilon(1/2,\JL(\delta),\psi)\epsilon(1/2,\eta_{E/F}\otimes\JL(\delta),\psi)=(-1)^m\eta_{E/F}(-1)^{md/2}.
\end{equation}
\end{enumerate}
\end{theorem}
\begin{proof}
When $E/F$ is a quadratic extension of $p$-adic fields, $\delta$ is cuspidal, and either $p$ is odd or $d\le 2$, it is a consequence of \cite{MR4226986} when $d\leq 2$, and of \cite{MR4773305} when $p$ is odd. A general proof of the cuspidal $p$-adic case lifting these restrictions has been announced in \cite{epsilonintertwining}. The $p$-adic discrete case then follows from \cite{MR4721777}. When $E/F$ is a quadratic extension of archimedean fields, the theorem is proved in \cite{MR4659866}. When $E=F\times F$ it is proved in \cite{Sign} and \cite{MR3953435}. 
\end{proof}

\begin{remark}\label{rem eps}
When $E\simeq F\times F$ the equality \eqref{eq eps dich} is satisfied whenever $\JL(\delta)$ is symplectic. 
\end{remark}

For the following results it will be convenient to make the following definition.

\begin{definition}\label{df odd esi}
We say that an irreducible, generic representation $\pi$ of $\GL_n(F)$ has odd essentially square-integrable support if $\pi\simeq\delta_1\times \dots \times \delta_r$ where $\delta_i$ is an essentially square-integrable representation of $\GL_{n_i}(F)$ and $n_i$ is odd for every $i=1,\dots,r$. 
\end{definition}

The setting being as above, the following result also follows from \cite{MR3430877}, \cite{Sign}, \cite{MR4157701} and \cite{MR4659866}.
\begin{theorem}\label{thm C2}
Let $n$ be an even integer and $\pi$ an irreducible, generic representation of $\GL_n(F)$  that has odd essentially square-integrable support. 
Then $\pi$ is $\GL_{n/2}(E)$-distinguished if and only if its Langlands parameter is symplectic.  Moreover, the epsilon dichotomy equality \eqref{eq eps dich} for $\delta=\pi$ (with $d=1$ and $m=n$) holds automatically.
\end{theorem}
\begin{proof}
Write $\pi=\delta_1\times\cdots\times \delta_r$ as in Definition \ref{df odd esi}. Comparing the classification of generic $\GL_{n/2}(E)$-distinguished and $\GL_{n/2}(F)\times \GL_{n/2}(F)$-distinguished representations recalled in Theorem \ref{thm classif dist S}, we see that $\pi$ is  $\GL_{n/2}(E)$-distinguished if and only if it is $\GL_{n/2}(F)\times \GL_{n/2}(F)$-distinguished, if and only if $r$ is even, and up to re-ordering, one has $\delta_{2i}\simeq \delta_{2i-1}^*$ for $i=1,2,\dots,r/2$. Now when $F$ is $p$-adic, the representation $\pi$ is $\GL_{n/2}(F)\times \GL_{n/2}(F)$-distinguished if and only if its Langlands parameter is symplectic according to \cite[Corollary 3.15]{MR3430877}. When $F$ is archimedean, $n_i=1$, i.e. each $\delta_i$ is a character $\chi_i$ of $F^\times$. It is then a simple exercise to check that the Langlands parameter $\oplus_{i=1}^r \chi_i$ of $\pi$, where each $\chi_i$ is identified with a character of the Weil group of $F$, is symplectic if and only if up to re-ordering, $r$ is even and $\chi_{2i}\simeq \chi_{2i-1}^{-1}$ for $i=1,2,\dots,r/2$. The last part of the statement on the epsilon dichotomy equality follows from \cite[Theorem 1.4 and its proof]{MR4157701} when $F$ is $p$-adic (observe that no restriction on $p$ is required there), and \cite[Theorem 1.1]{MR4659866} when $F$ is archimedean.
\end{proof}

\begin{remark}\label{rem odd esi}
Theorem \ref{thm C2} applies when $\pi$ is a generic principal series and in particular when $\pi$ is generic and unramified. 
\end{remark}

\subsection{On the converse implication of the Guo-Jacquet Conjecture}
Thanks to the local global principle in Theorem \ref{thm main twlin}, we can finally prove the following form of converse to the Guo-Jacquet conjecture. 
Let $E/F$ be a quadratic extension of number fields. 

\begin{theorem}\label{thm::GJconverse}
 Let $\pi$ be a cuspidal automorphic representation of $\GL_n(\BA)$ where $n$ is even, and write $n=2^a b$ where $a\geq 1$ and $b$ is odd. Assume moreover the following:
\begin{enumerate}
\item  at all places $v$ of $F$ that are inert in $E$ (i.e. such that $E_v$ is a field)
\begin{itemize}
\item either $\pi_v$ is a discrete series,
\item or $\pi_v$ has odd essentially square-integrable support (see Definition \ref{df odd esi} and observe that this condition holds at almost all places $v$ by Remark \ref{rem odd esi}),
\end{itemize}
\item \label{global shal} $L(\frac12,\BC_F^E(\pi))\ne 0$ and $L(s,\pi,\wedge^2)$ has a pole at $s=1$. 
\end{enumerate}

Then either $\pi$ is $\GL_{n/2}(\A_E)$-distinguished or there exists a central $F$-division algebra $D$ of degree $2^a$ such that $E$ imbeds in $D$, and there exists  an irreducible, cuspidal automorphic representation $\pi'$ of 
$\GL_b(D_{\BA})$ with $\JL(\pi')=\pi$, such that $\pi'$ is $\GL_b(C_{\A_E})$-distinguished where $C$ is the centralizer of $E$ in $D$.
\end{theorem}
\begin{proof}
First we observe that each local component of $\pi$ has a symplectic parameter. Indeed by \cite{MR1044830}, we deduce from condition \eqref{global shal} that $\pi$ has a Shalika period. In particular we infer that all its local components have a Shalika model, hence a linear model thanks to \cite{MR1394521} and \cite{MR4047168}. We conclude for example thanks to \cite[Corollary 3.4]{Sign}. 

Note that every irreducible, cuspidal automorphic representation of $\GL_n(\A)$ is, by definition, automatically $\GL_{n/2}(\A_E)$-compatible. 

In order to discuss the local identity \eqref{eq eps dich} at every place $v$ of $F$, we fix once and for all a non-trivial additive character $\psi$ of $\A/F$ and consider the identity at $v$ with the additive character $\psi_v$.

Now if $\pi$ is locally $\GL_{n/2}$-distinguished, then it is $\GL_{n/2}(\A_E)$-distinguished thanks to Theorem \ref{thm main twlin}. 
Otherwise, let $S$ be the non-empty set of places $v$ of $F$ such that $\pi_v$ is not $\GL_{n/2}(E_v)$-distinguished. For every place $v\in S$, by Theorem \ref{thm C1} and Remark \ref{rem eps}, $v$ is inert in $E$ and by Theorem \ref{thm C2}, $\pi_v$ is square-integrable. Furthermore, equation \eqref{eq eps dich} fails by a sign for $\delta=\pi_v$. That is, the left hand side equals negative the right hand side.

Since, by assumption, $L(\frac12,\BC_F^E(\pi))=L(1/2,\pi)L(1/2,\eta_{E/F}\otimes\pi)\neq 0$, we deduce from the functional equation of standard $L$-functions (Theorem \ref{thm gfe LS}) that 
\[
\epsilon(1/2,\pi)\epsilon(1/2,\eta_{E/F}\otimes\pi)=1.
\] 
Furtheremore, automorphy of the quadratic character $\eta_{E/F}$ and the pairity of $n$ imply that $(-1)^n \eta_{E/F}((-1)^{n/2})=1$. We therefore further conclude from Theorem \ref{thm C2} at all $v\not\in S$ that $\abs{S}$ is even.
In particular, $S$ contains at least two places. 

Applying the Brauer-Hasse-Noether Theorem \cite[Theorem 1.12]{MR1278263}, there exists a degree $2^a$ central division $F$-algebra $D$ such that $D_v$ is a division algebra for every $v\in S$ and $D_v$ splits over $F_v$ for every $v\not\in S$. It follows from \cite[Theorem 1.1]{MR3217641} that $E$ embeds in $D$. Indeed, applying the notation of the theorem in loc. cit. with $K=E$ and $A=D$, if $v$ is a place of $F$ not in $S$ then $d_v=1$ and therefore $d_w'=1$ for every place $w$ of $E$ that lies over $v$ and if $v\in S$ and $w$ is the unique place of $E$ above $v$ then $d_v=2^a$  while $d_w'=2^{a-1}$. Part 3 of the theorem gives that the capacity of $D_E$ equals $2$ and therefore part 2 of the theorem says that $E$ imbeds in $D$. Consider now $E$ as a subfield of $D$ and denote by $C$ the centralizer of $E$ in $D$. Let $H=\Res_{E/F}(G_C(b))$ so that $H(F)=\GL_b(C)$.

It follows from \cite[Theorem 18.1]{MR2684298} that $\pi=\JL(\pi')$ for a unique irreducible, cuspidal automorphic representation $\pi'$ of $\GL_b(D_{\BA})$ and that furthermore, $\pi_v=\JL(\pi'_v)$ for every place $v$ of $F$. In particular, $\pi_v'$ is square-integrable for every $v\in S$.

Since for $v\in S$ the equation \eqref{eq eps dich} for $\delta=\pi_v$ fails by a sign and since $(-1)^b=-1$ it follows that \eqref{eq eps dich} holds for $\delta=\pi'_v$  for every $v\in S$ and hence, in fact, for every place $v$ of $F$. By the assumption of the theorem and Theorem \ref{thm C2} it follows that $\pi'$ is locally $H$-distinguished.

Since, by definition, every irreducible square integrable representation of $\GL_b(D_v)$ that is $\GL_b(C_v)$-distinguished is $\GL_b(C_v)$-compatible this is the case for $\pi_v'$ for every $v\in S$. Since furthermore for $v\not\in S$ every irreducible, generic representation of $\GL_b(D_v)\simeq \GL_n(F_v)$ is $\GL_b(C_v)$-compatible we conclude that $\pi'$ is $H$-compatible.

Together, it now follows from Theorem \ref{thm main twlin} that $\pi'$ is $\GL_b(C_{\A_E})$-distinguished. The theorem follows.

\end{proof}

\appendix
\section{Failure of the naive local-global principle in the Galois case}\label{app failure}

Here we give examples of representations as in Theorem \ref{thm::Main}, which are everywhere locally distinguished by the Galois involution, but not distinguished. This does not seem to appear in the literature so far. Note that in the proof of Theorem \ref{thm::Main}, we used local distinction everywhere plus some extra condition, to claim global distinction. 
 Our examples come from discussions with Beuzart-Plessis, and the simplest form that we present here is essentially due to him. Note that for $m=2$, in the case of linear and twisted linear periods, the failure of the naive local-global principle is well-known from Waldspurger's work \cite{MR783511}. Note further that for $\SL_n$ and the Galois involution, this phenomenon is also well-known to occur and is discussed in details in \cite{MR2275907}, \cite{MR3077658} and \cite{MR4549279}.

For $n=md$ odd, the local global principle holds in case \namelink{gal2}. Indeed local distinction of $\pi$ implies that $\jl(\pi)=\jl(\pi)^*$ by strong multiplicity one. By the results of Jacquet-Shalika this implies that $L(s,\jl(\pi),\jl(\pi)^\vartheta)$ has a pole at $s=1$ and therefore exactly one of $L(s,\jl(\pi),\As^{\pm})$ has a pole at $s=1$. By results in \cite{MR984899} and \cite{MR1344660} this implies that $\jl(\pi)$ is either distinguished or $\eta_{E/F}$-distinguished, but since $n$ is odd it must be distinguished for central character reasons. We deduce that $\pi$ is itself distinguished thanks to Theorem \ref{thm main gal odd}. 

Now let us give our family of examples for which the local global principle fails. In particular $n$ is even, and observe that we may as well assume that $G$ is split thanks to Theorems \ref{thm main gal odd} and \ref{thm main gal even} and Corollary \ref{cor distinction-transfer-generic}. Start with $E/F$ an unramified quadratic extension of number fields (in particular split at infinite places). Take a finite place $v_0$ of $F$ which is split in $E$, and $\pi_{v_0}$ a cuspidal representation of $G_{v_0}=H_{v_0}\times H_{v_0}$ which is $H_{v_0}$-distinguished. Then ``globalize" $\pi_{v_0}$ into an irreducible, $\eta_{E/F}$-distinguished cuspidal automorphic representation $\pi$ of $G(\BA)$ which is unramified at every finite place different from $v_0$ thanks to \cite[Theorem 4.1]{MR2369492}. Finally, observe that at all places $v$, the $\eta_{E_v/F_v}$-distinction property is actually equivalent to distinction:
\begin{itemize} 
\item if $v$ is split it follows from the fact that $\eta_{E_v/F_v}=1$;
\item otherwise $v$ is inert and unramified and $\pi_v$ is unramified as well. In this case the observation follows from the classification theorem \cite[Theorem 5.2]{MR2755483} as the only unramified character of $E_v^\times$ trivial on $F_v^\times$ is the trivial character, which is distinguished.
\end{itemize}

As a conclusion, $\pi$ is everywhere locally distinguished, but it can't be distinguished according to Flicker's result (\cite{MR984899}).


\newcommand{\etalchar}[1]{$^{#1}$}
\def\cprime{$'$} \def\Dbar{\leavevmode\lower.6ex\hbox to 0pt{\hskip-.23ex
  \accent"16\hss}D} \def\cftil#1{\ifmmode\setbox7\hbox{$\accent"5E#1$}\else
  \setbox7\hbox{\accent"5E#1}\penalty 10000\relax\fi\raise 1\ht7
  \hbox{\lower1.15ex\hbox to 1\wd7{\hss\accent"7E\hss}}\penalty 10000
  \hskip-1\wd7\penalty 10000\box7}
  \def\polhk#1{\setbox0=\hbox{#1}{\ooalign{\hidewidth
  \lower1.5ex\hbox{`}\hidewidth\crcr\unhbox0}}} \def\dbar{\leavevmode\hbox to
  0pt{\hskip.2ex \accent"16\hss}d}
  \def\cfac#1{\ifmmode\setbox7\hbox{$\accent"5E#1$}\else
  \setbox7\hbox{\accent"5E#1}\penalty 10000\relax\fi\raise 1\ht7
  \hbox{\lower1.15ex\hbox to 1\wd7{\hss\accent"13\hss}}\penalty 10000
  \hskip-1\wd7\penalty 10000\box7}
  \def\ocirc#1{\ifmmode\setbox0=\hbox{$#1$}\dimen0=\ht0 \advance\dimen0
  by1pt\rlap{\hbox to\wd0{\hss\raise\dimen0
  \hbox{\hskip.2em$\scriptscriptstyle\circ$}\hss}}#1\else {\accent"17 #1}\fi}
  \def\bud{$''$} \def\cfudot#1{\ifmmode\setbox7\hbox{$\accent"5E#1$}\else
  \setbox7\hbox{\accent"5E#1}\penalty 10000\relax\fi\raise 1\ht7
  \hbox{\raise.1ex\hbox to 1\wd7{\hss.\hss}}\penalty 10000 \hskip-1\wd7\penalty
  10000\box7} \def\lfhook#1{\setbox0=\hbox{#1}{\ooalign{\hidewidth
  \lower1.5ex\hbox{'}\hidewidth\crcr\unhbox0}}}

\end{document}